\documentclass[11pt,leqno]{amsart}%
\usepackage{amsmath}
\usepackage{amsfonts}
\usepackage{amssymb}
\usepackage{graphicx}
\usepackage{dsfont}
\usepackage{xcolor}
\usepackage{hyperref}
\usepackage{bm}
\usepackage[margin=2.50cm]{geometry}
\providecommand{\U}[1]{\protect\rule{.1in}{.1in}}
\newtheorem{theorem}{Theorem}[section]
\newtheorem*{acknowledgement*}{Acknowledgement}

\newtheorem{corollary}[theorem]{Corollary}

\newtheorem{lemma}[theorem]{Lemma}

\newtheorem{problem}[theorem]{Problem}

\newtheorem{proposition}[theorem]{Proposition}
\newtheorem{remark}[theorem]{Remark}

\newcommand{\B}{\mathbb{B}}
\newcommand{\R}{\mathbb{R}}
\newcommand{\hh}{\hat{h}}
\renewcommand{\th}{\tilde{h}}
\newcommand{\tr}{{\tilde{r}}}

\newcommand{\hH}{\hat{H}}
\newcommand{\hg}{\hat{g}}
\newcommand{\p}{\partial}

\renewcommand{\l}{\lambda}
\newcommand{\ton}[1]{\left(#1\right)}
\newcommand{\qua}[1]{\left[#1\right]}
\newcommand{\abs}[1]{\left|#1\right|}
\newcommand{\norm}[1]{\left\|#1\right\|}
\newcommand{\RR}{\mathrm{R}}
\newcommand{\tg}{{\tilde g}}

\title{Higher order distance-like functions and Sobolev spaces}

\author[Debora Impera]{Debora Impera}
\address[Debora Impera]{Dipartimento di Scienze Matematiche "Giuseppe Luigi Lagrange", Politecnico di Torino, Corso Duca degli Abruzzi, 24, Torino, Italy, I-10129}
\email{debora.impera@polito.it}

\author[Michele Rimoldi]{Michele Rimoldi}
\address[Michele Rimoldi]{Dipartimento di Scienze Matematiche "Giuseppe Luigi Lagrange", Politecnico di Torino, Corso Duca degli Abruzzi, 24, Torino, Italy, I-10129}
\email{michele.rimoldi@polito.it}

\author {Giona Veronelli}
\address[Giona Veronelli]{Dipartimento di Matematica e Applicazioni, Universit\`a di Milano Bicocca, via R. Cozzi 53, I-20125 Milano, Italy}
\email{giona.veronelli@unimib.it}

\begin{document}
\begin{abstract}
On non-compact Riemannian manifolds, we construct distance-like functions with derivatives controlled up to some order $k$ assuming bounds on the growth of the derivatives of the curvature up to order $k-2$ and on the decay of the injectivity radius. This construction extends previously known results in various directions, permitting to obtain consequences which are (in a sense) sharp. 
As a first main application, we give refined conditions guaranteeing the density of compactly supported functions in the Sobolev space $W^{k,p}$ on the manifold. Contrary to all previously known results this can be obtained also on manifolds with possibly unbounded geometry.
In the particular case $p=2$, making use of the Weitzenb\"ock formula for a Lichnerowicz Laplacian acting on $k$-covariant totally symmetric tensor fields, we can weaken the assumptions needed to obtain the density property, avoiding any condition on the highest order derivatives of the curvature.
Distance-like functions are also used to obtain new disturbed Sobolev inequalities,  disturbed $L^{p}$-Calder\'on-Zygmund inequalities and the full Omori-Yau maximum principle for the Hessian under weak assumptions.
\end{abstract}

\date{\today}

\maketitle
\tableofcontents

\section{Introduction and main results}\label{Intro}

\subsection{Higher order distance-like functions}
Let $(M, g)$ be a complete non-compact Riemannian manifold and let $r(x)$ be the distance from $x\in M$ to a fixed reference point $o\in M$. The behaviour of the function $r(x)$ detects how much the manifold differs from the Euclidean space. It turns out that in general $r(x)$ is $1$-Lipschitz on $M$, but only a.e. differentiable. This happens as soon as the point $o$ has a cut-locus, and in particular whenever $M$ has non-trivial topology. This lack of regularity of $r$ represents an obstacle to perform a certain kind of smooth analysis. In order to overcome this problem, it is often enough to use, in place of $r$, a suitable smooth approximation of the distance function. A first evidence in this direction is given by a well-known result by M. P. Gaffney \cite{Gaffney} (often attributed in the literature to R. E. Greene and H. Wu, \cite{GWAnnSci}), who showed the existence of a function $H(x)$ on $M$ which is smooth, distance-like (i.e. $C^{-1}r(x)<H(x)<Cr(x)$ outside a compact set) and whose gradient has norm uniformly bounded by 1. This can be obtained by smoothing $r$ in local coordinate charts by mollification.

In order to adapt to Riemannian manifolds certain techniques of $C^k$ (Euclidean) analysis, one would like to have at disposal on $M$ a smooth distance-like functions of which one controls not only the gradient, but also higher order (covariant) derivatives up to the order $k$. It is thus a natural question to what extent one can generalize Gaffney's result, providing sufficiently general geometric assumptions which guarantee the existence of a distance-like function with controlled higher order derivatives.

A first achievement in this direction is due to J. Cheeger and M. Gromov, \cite{CheegerGromov} (see also \cite{Shi} for a detailed proof), who remarked that a uniform double-sided bound on the sectional curvatures of $M$ permits to get a uniform bound also on the norm of the second order derivatives of Gaffney's first-order distance-like function through the local smoothing process, thanks to the Hessian comparison. More recently, L. F. Tam provided a new completely different proof of Cheeger-Gromov's result: start with the first-order distance-like function by Gaffney-Greene-Wu and let it evolve by the heat flow of $M$; the regularizing property of the heat flow, together with the sectional curvature bound, permits to control also the second order derivatives along the evolution, so that the evolution at a fixed time (say $t=1$) gives the aimed second-order distance-like function. Contrary to the proof by Cheeger-Gromov, Tam's approach is more flexible, in the sense that it can be adapted to other sets of assumptions. This was first observed by the second and third authors in \cite{RimoldiVeronelli}, where the existence of a second order-distance like function was obtained when the underlying manifold has (double-sided) bounded Ricci curvature and positive injectivity radius. For a further recent result, see also \cite{Huang}, where the assumption of positive injectivity radius is replaced by the request that the volumes of unitary balls are almost Euclidean. This seems a pretty strong assumption, which however a priori could not imply $\mathrm{inj}_g(M)>0$ in general.

As we will see later, for most applications the uniform bound on the second-order derivatives is a property stronger than what really needed, in that we can actually allow derivatives to explode at infinity, provided that their growth is appropriately controlled. Starting from this observation, in the previous paper \cite{IRV-HessCutOff} we found conditions which ensure the existence of a distance-like function with uniformly bounded gradient and sub-linear growth of the second-order derivatives. Namely, this holds true provided that: a) the norm of the Riemann tensor grows sub-quadratically or b) the norm of the Ricci tensor grows sub-quadratically and the injectivity radius can possibly vanish at infinity, but no faster than $1/r(x)$. In order to obtain this improvement, we introduced a new strategy: the distance-like function is no more given by a solution of the linear homogeneous parabolic heat flow, but instead by the solution of a non-linear non-homogeneous elliptic equation. 

Concerning the existence of higher order distance-like functions (i.e. with controlled derivatives up to the order $k>2$), 
up to our knowledge, the only known achievement is due to L. F. Tam ( see \cite[Remark 26.50]{Chowetal-III}) who proved the following
\begin{proposition}\label{prop_Tam}
Let $k\geq 2$.	Let $(M^m, g)$ be a complete Riemannian manifold such that, for some $D>0$, we have $|\nabla^{j}\mathrm{Riem}|\leq D^{2}$ for $j=0,\ldots,k-2$. Then there exists a distance-like function $H\in C^{\infty}(M)$ such that $|\nabla^{j}H|\leq C$, for $j=1,\ldots,k$, with $C$ a positive constant depending only on $m, k, D$.
\end{proposition}

Let $\lambda:[0,+\infty)\to[0,+\infty)$ be a $C^\infty$ function. Before stating the first main result of this paper, let us introduce the following assumptions:
\begin{itemize}
\item[(A1)] There exists a constant $R_1>0$ such that $\lambda$ is strictly positive and non-decreasing on $[R_1,+\infty)$. 
\item[(A2)] For every $\delta >0$, there exist constants $R_2=R_2(\delta)>0$ and $M_2=M_2(\delta)>1$ such that 
\[\forall\,t>R_2,\quad M_2(\delta)^{-1}\leq \frac{\lambda(\delta t)}{\lambda(t)}\leq M_2(\delta),\quad \text{and}\quad M_2(\delta)^{-1}\leq \frac{\lambda'(\delta t)}{\lambda'(t)}\leq M_2(\delta).
\]
\item[(A3)] There exist constants $R_{3}>0$ and $M_{3}>1$ such that 
\[\forall\,t>R_{3},\quad M_3^{-1}\leq \frac{t\lambda'(t)}{\lambda(t)}\leq M_3.
\]
Moreover, for every integer $j\geq 1$ we introduce the assumptions
\item[(A4($j$))] There exist constants $R_{4(j)}>0$ and $M_{4(j)}>1$ such that 
\[\forall\,t>R_{4(j)},\quad  \frac{t\lambda^{(j)}(t)}{\lambda^j(t)}\leq M_{4(j)}.
\]
\end{itemize}

\medskip

The first main theorem of this paper is the following.

\begin{theorem}\label{th_distancefunction}
Let $(M, g)$ be a complete Riemannian manifold and $o\in M$ a fixed reference point, $r(x)\doteq \mathrm{dist}(x,o)$. Let $k\geq 2$ be an integer. Let $\lambda$ satisfy the assumptions (A1), (A2), (A3) and (A4($j$)), $j=1,\dots,k$.
	Suppose that one of the following curvature assumptions holds
	\begin{itemize}
		\item[(a)] for some $i_0>0$,
		\[
		\ |\nabla^j\mathrm{Ric}|(x)\leq \lambda(r(x))^{2+j},\ 0\leq j \leq k-2,\quad\mathrm{inj}_{g}(x)\geq \frac{i_0}{\lambda(r(x))}>0\quad\mathrm{on}\,\,M.
		\]
		\item[(b)] 
		\[
		\ |\nabla^j\mathrm{Ric}|(x)\leq \lambda(r(x))^{2+j},\ 1\leq j \leq k-2,\ |\mathrm{Sect}|(x)\leq D^2\lambda(r(x))^2.
		\]
	\end{itemize}
	Then there exists an exhaustion function $H\in C^{\infty}(M)$ such that for some positive constant $C>1$ independent of $x$ and $o$, we have on $M$ that 
	\begin{itemize}
		\item[(i)]$C^{-2}r(x)\leq H(x)\leq \max\left\{r(x), 1\right\}$;
		\item[(ii)] $|\nabla H|(x)\leq 1$;
		%$|\nabla H|(x)\leq \max\{r^\eta(x),1\}$;
		\item[(iii)] for $2\leq j \leq k$, $\left|\nabla^{j} H\right|(x)\leq C\max\{\lambda(r(x))^{j-1},1\}$.
	\end{itemize}
\end{theorem}

More comments about assumptions (A1) to (A4) will be given in Subsection \ref{subsec_assumptions}. For the moment let us only point out that admissible choices for $\lambda$ which are relevant in the applications are, for instance, constant multiples of
\[
\lambda(r(x))=r(x)^\eta,\ \eta\geq 0,\qquad\text{or}\qquad \lambda(r(x))=
r(x)\prod_{i=1}^{\bar{i}}\ln^{[i]}(r(x))\]
where $\ln^{[i]}$ stands for the $i$-th iterated logarithm (e.g. $\ln^{[2]}(t)=\ln\ln t$, etc.) and $\bar{i}$ is some positive integer. See Lemma \ref{l: lambda} for more details. 
\medskip

Theorem \ref{th_distancefunction} extends the previously known results in several directions. Namely note that:
\begin{itemize}
	\item The Proposition \ref{prop_Tam} due to Tam is still true if we assume only $|\nabla^{j}\mathrm{Ric}|\leq D^{2}$ and either a uniform bound on all the sectional curvatures or the positivity of the injectivity radius;    
	\item as in \cite{IRV-HessCutOff}, we can assume a controlled growth of the curvatures and\slash{}or a controlled decay of the injectivity radius and get an estimate for the growth of the derivatives, which is enough for some important applications, as for instance for proving the density of compactly supported functions in Sobolev spaces;
	\item we can admit more general (non-necessarily sub-polynomial) growth functions. This generalizes also the second order result obtained in \cite{IRV-HessCutOff}, and permits to get consequences which are (in a sense) sharp. See for instance Remark \ref{rmk sharp_sob} and Subsection \ref{subsec_OY}.
\end{itemize} 
 
We present here the strategy of the proof of Theorem \ref{th_distancefunction} in the special easy case where $\lambda(r(x))=r(x)$. Reasoning by induction, suppose that we have a distance-like function $H_{k-1}$ which satisfies $|\nabla^{j} H_{k-1}|\leq r(x)^{j-1}$ outside a compact set for $j=1,\dots,k-1$. According to (a generalization of) a theorem due to D. Bianchi and A. G. Setti, \cite{BianchiSetti}, we can prove the existence of a smooth function $h$ such that
\begin{itemize}
	\item[(i)] $C_h^{-1}r^2(x)\leq h(x) \leq C_hr^2(x)$;
	\item[(ii)] $|\nabla h|\leq C_h r(x)$;
	\item[(iii)] $\Delta h = |\nabla h|^2-C_\theta H_{k-1}^2(x)$,
\end{itemize}
outside a compact set; see Theorem \ref{refinedBS}. Suppose that we already have obtained an estimate for $|\nabla^j h|$ for $j=2,\dots, k-1$ (formally, this is done via a second induction process). Derivating the equation at (iii), we get 
\begin{equation}\label{eq intro}
\nabla^{k-2}\Delta h = \nabla^{k-2}|\nabla h|^2-C_\theta \nabla^{k-2} H_{k-1}^2(x).
\end{equation}
The curvature and injectivity radius assumptions, together with an appropriate local rescaling of the metric, guarantee the existence of a harmonic atlas with respect to which the metric is uniformly controlled in $C^{k-1,\alpha}$. Accordingly, in any local harmonic coordinate chart $\Omega$, equation \eqref{eq intro} writes
$\Delta_0 \nabla^{k-2} h = F$ on $\Omega$, where $\Delta_0$ is the Euclidean Laplacian and $F$ a certain (vector-valued) function which depends on $\|g\|_{C^{k-1,\alpha}(\Omega)}$, $\|h\|_{C^{k-1}(\Omega)}$ and $\|H_{k-1}\|_{C^{k-1}(\Omega)}$. Since all these three norms are controlled, we are in the position to implement standard Euclidean elliptic theory and deduce a 
$C^{2,\alpha}$ control on $\nabla^{k-2} h$, that is, a $C^{k,\alpha}$ control on $\nabla^k h$. The desired function $H$ is finally obtained letting $H=h^2$. 

%%%%%%%%%%%%%%%%%%%%%%%%%%%%%%%%%%%%%%%%
\subsection{The density property for Sobolev spaces}\label{Densprop}

As alluded to above, distance-like functions guarantee that the underlying manifold $M$ is not too different from the Euclidean space (in a suitable sense). Namely, when $M$ supports a distance-like function with controlled derivatives, certain properties and tools from classic analysis on $\R^n$ can be proved to have a Riemannian counterpart on $M$. In this sense, one of the main applications is to density results of smooth compactly supported functions in Sobolev spaces, 
as probably first observed in \cite{Guneysu}, \cite{GuneysuPigola}.

Given $(M^{m}, g)$ a smooth, complete, possibly non-compact Riemannian manifold without boundary, $k$ an integer, and $p\geq 1$, one can define the Sobolev space $W^{k,p}(M)$ as the space of functions on $M$ all of whose (weak) derivatives of order $0$ to $k$ have finite $L^{p}$ norm. By a generalised Meyers-Serrin-type theorem (see e.g. \cite{GuidettiGuneysuPallar}) this coincides with the completion of the space 
\[
\mathcal{C}_{k}^{p}(M)\doteq \left\{u\in C^{\infty}(M): \int_{M}|\nabla^{j}u|^{p}d\mathrm{vol}<+\infty,\,\,\forall\,j=0,\ldots,k\right\}
\]
with respect to the norm
\[
\ \left\|u\right\|_{k,p}=\sum_{j=0}^{k}\left(\int_{M}|\nabla^{j}u|^{p}d\mathrm{vol}\right)^{\frac{1}{p}}.
\]
Moreover, we can define $W_{0}^{k,p}(M)\subseteq W^{k,p}(M)$ as the closure of the space of smooth compactly supported functions $C_{c}^{\infty}(M)$ with respect to the same norm.
\medskip

In case $M=\R^m$, it is well known that $W_{0}^{k,p}(\mathbb{R}^{m})=W^{k,p}(\mathbb{R}^{m})$ for all $k$ integer and $p\in[1,\infty)$, so that one could expect the same equivalence to hold true also on complete Riemannian manifold $M$. However, quite surprisingly, the problem seems to remain open in general so far; see for instance \cite[p. 49]{HebeyCourant}.

\begin{problem}
Given an arbitrary complete Riemannian manifold $(M,g)$, is it true that
\begin{equation}\label{eq_dens}
W_{0}^{k,p}(M)=W^{k,p}(M)
\end{equation} 
for all integer $k\geq 0$ and $p\in[1,\infty)$?
\end{problem}

Of course, the answer is positive under certain additional assumptions. First, \eqref{eq_dens} is satisfied for all $k$ and $p$ when $M$ is compact; see for instance \cite[Theorem 2.9]{Aubin}. Concerning complete non-compact manifolds, it is a standard fact that $W_{0}^{0,p}(M)=W^{0,p}(M)=L^p(M)$, and with some effort one can also prove that $W_{0}^{1,p}(M)=W^{1,p}(M)$ for all $k$ and $p\in[1,\infty)$; \cite{aubin-bull}. 

Regarding the first non-trivial order, i.e. $k=2$, in our previous work \cite{IRV-HessCutOff} we showed that $W_{0}^{2,p}(M)=W^{2,p}(M)$, $p\in[1,+\infty)$,  for complete manifolds with either a sub-quadratic growth of the norm of the Riemann curvature, or a subquadratic growth of both the norm of the Ricci curvature and the squared inverse of the injectivity radius. Previously known results, obtained in \cite{HebeyCourant}, \cite{Guneysu-Book}, were assuming uniform constant bounds on either the Ricci tensor and the injectivity radius or the Riemann tensor. As far as concerns higher order Sobolev spaces the most up to date result is  the following proposition due to E. Hebey, \cite{HebeyCourant}.

\begin{proposition}[Proposition 3.2 in \cite{HebeyCourant}]\label{PropHebey}
Let $(M, g)$ be a smooth, complete Riemannian manifold with positive injectivity radius, and let $k\geq 3$ be an integer. We assume that for $j=0,\ldots, k-2$, $|\nabla^{j}\mathrm{Ric}|$ is bounded. Then for any $p\in[1,\infty)$, $W_{0}^{k,p}(M)=W^{k, p}(M)$. 
\end{proposition}

As a first main application of Theorem \ref{th_distancefunction}, we will prove the density property on complete Riemannian manifolds with a suitable growth condition on the derivatives of the Ricci tensor and either the Riemannian curvature tensor or the inverse of the injectivity radius. The admissible growth rate depends explicitly on the order of the Sobolev space we are dealing with.  Note that, contrary to Proposition \ref{PropHebey}, the following Theorem \ref{th_main1} concerns also manifolds with possibly unbounded geometry. 

\begin{theorem}\label{th_main1}
Let $k\geq 2$. In the assumptions of Theorem \ref{th_distancefunction}, if we further assume that $\lambda^{1-k}\not\in L^1([1,+\infty))$, then we have that $W^{k,p}(M)=W^{k,p}_0(M)$ for all $p\in[1,+\infty)$.
\end{theorem}

\begin{remark}\label{rmk sharp_sob}{\rm 
Choosing the function $\lambda$ in Theorem \ref{th_main1} in essentially the best admissible way, we get that $W^{k,p}(M)=W^{k,p}_0(M)$ for all $p\in[1,+\infty)$ for instance if 	\[
\ |\nabla^j\mathrm{Ric}|(x)\leq \left(r(x)\prod_{i=1}^{\bar{i}}\ln^{[i]}(r(x))
\right)^{(2+j)/(k-1)},\ 0\leq j \leq k-2,\] 
and either
\begin{itemize}
			\item[(a)] for some $i_0>0$,
$\mathrm{inj}(x)\geq i_{0}\left(r(x)\prod_{i=1}^{\bar{i}}\ln^{[i]}(r(x))\right)^{-1/(k-1)}>0$, or
	\item[(b)] for some $D>0$,
$|\mathrm{Riem}|(x)\leq D^2\left(r(x)\prod_{i=1}^{\bar{i}}\ln^{[i]}(r(x))\right)^{2/(k-1)}.$
\end{itemize}
Accordingly, this theorem improves also our previous \cite[Theorem 1.4]{IRV-HessCutOff} when $k=2$, permitting to achieve the density result $W_{0}^{2,p}(M)=W^{2, p}(M)$ on manifolds whose curvatures can grow more than quadratically.}\end{remark} 

To get a flavour of the proof of Theorem \ref{th_main1}, consider here the easy case $k=3$ and $\lambda(r) =r^{1/2}$ (so that $\lambda^{1-k}\not\in L^1([1,+\infty))$). Following a standard construction, let $\phi_{R}\in C^\infty_c(\R_+)$ be a family of cut-off functions such that $\phi_R\equiv 1$ on $[0,R]$ and $|\phi^{(j)}_R|\leq CR^{-j}\mathds{1}_{[R,2R]}$. Let $H$ be the 3rd-order distance-like function given by Theorem \ref{th_distancefunction} and define $\chi_{R}\doteq\phi_R\circ H$. Then $\{\chi_R\}_{R>1}$ is a family of 3rd-order cut-off functions, namely $|\nabla^j\chi_R|$ are uniformly bounded as $R\to\infty$ for $j=1,2,3$. Once one has 3rd-order cut-off functions, it is a standard fact that $f\in W^{3,p}(M)$ can be approached by $\chi_R f$ in $W^{3,p}(M)$-norm. In this process, the main term to control is of the form $|\phi_R'(H(x))\nabla^3H(x)|\lesssim R^{-1}\lambda^2 = R^{-1}R=1$. Here, this latter control was deduced from uniform estimates which we have separately on $\phi_R'(H(x))$ and on $\nabla^3H(x)$ in $B_{2R}(o)\setminus B_{R}(o)$. One can look instead for a more careful point-wise control. Namely one imposes that $|\phi_R'(H(x))|$ is smaller and smaller as $H(x)\to\infty$ in order to compensate the growth of $\nabla^3H(x)$. This leads us to a more careful construction of the real cut-offs $\phi_R$, in such a way that $|\phi_R'(H(x))| \sim \frac{1}{\lambda^2(H(x))}\sim \frac{1}{|\nabla^3H(x)|}$. This can be done as soon as $\lambda^{1-k}\not\in L^1(+\infty)$, and permits us to get the results under the sharper assumptions which we described in Remark \ref{rmk sharp_sob}.

\subsection{Sampson-Lichnerowicz Laplacian and the special case $W^{k,2}(M)$}
In the particular case $p=2$, the assumptions of Theorem \ref{th_main1} can be weakened. For $k=2$, L. Bandara \cite{Bandara} proved that $W^{2,2}(M)=W^{2,2}_0(M)$ under the sole lower bound $\mathrm{Ric}(x)\geq -C$ for some constant $C>0$. In \cite{IRV-HessCutOff}, we pointed out that in fact $\mathrm{Ric}(x)\geq -Cr^2(x)$ suffices. The point is that in these weaker assumptions one can still prove the existence of a distance-like function $H$ with $|\Delta H|\leq C r^2$, and hence of Laplacian cut-offs with $|\Delta \chi_n|$ uniformly bounded, \cite{BianchiSetti}. Given $f\in W^{2,2}(M)$, the Weitzenb\"ock formula applied to the Hodge Laplacian acting on the (skew-symmetric) one form $d(\chi_n f)$ permit to control $\|f\nabla^2 \chi_n\|_{L^2}$ in terms of $\|f\Delta \chi_n\|_{L^2}$, and thus to prove that $\chi_n f \to f$ in $W^{2,2}(M)$. Generalizing this approach to higher order Sobolev spaces is non-trivial. Indeed, one is led to apply Bochner techniques to the $(0,k-1)$-tensor $\nabla^{k-1}(\chi_n f)$, which is very far from being skew-symmetric when $k>2$. Accordingly, the Bochner formula for the Hodge Laplacian acting on $k$-forms can not be exploited. However, it turns out that $\nabla^2(\chi_n f)=\mathrm{Hess}\,(\chi_n f)$ is symmetric. This is no more true for $k>3$, but $\nabla^{k-1}(\chi_n f)$ remains almost symmetric, in the sense that it has a symmetric principal term, plus lower order terms which can be controlled. Thus, one can apply an analogous technique to J. H. Sampson's Laplacian $\Delta_{\mathrm{Sym}}$, \cite{Sampson}. This latter is a Laplace operator acting on the space of symmetric $(0,k)$-tensors which 
\begin{itemize} 
	\item[(a)] is a Lichnerowicz Laplacian, i.e. it satisfies a Weitzenb\"ock formula \begin{equation*}
	\Delta_{L}=\Delta_{B}+c\mathfrak{Ric},
	\end{equation*}
	where $c$ is a suitable (in this case negative) constant, $\Delta_{B}$ is the Bochner Laplacian and $\mathfrak{Ric}$ is the Weitzenb\"ock curvature operator (see Subsection \ref{subsec_Weitz} for more details).
	\item[(b)] has a Hodge-type decomposition $\Delta_{\mathrm{Sym}}=D_{S}^{*}D_{S}-D_{S}D_{S}^{*}$, where  $D_{S}$ is the symmetrized covariant derivative, and $D_{S}^{*}$ its formal adjoint.
\end{itemize}
Exploiting this, one can get a control on the $L^2$-norm $\|f\nabla^k \chi_n\|_{L^2}$ in terms of $\|f\Delta \nabla^{k-2}\chi_n\|_{L^2}$. Choosing the cut-offs $\chi_{n}$ to be what we called $k$-th order (rough) Laplacian cut-offs, this latter term can be estimated under assumptions which are weaker than the assumptions of Theorem \ref{th_main1} (in particular we do not need to control the highest order derivatives of the curvature). More precisely, in Section \ref{p2} we will prove the following. 

\begin{theorem}\label{Dens_p=2}
Let $(M, g)$ be a complete Riemannian manifold and $o\in M$ a fixed reference point $r(x)\doteq\mathrm{dist}_{g}(x,o)$. Let $k\geq 3$ be an integer. Let $\lambda$ satisfy assumptions (A1), (A2), (A3), and (A4(j)) for $j=1,\ldots,k-1$ above, and suppose that $\lambda^{1-k}\notin L^{1}([1,+\infty))$. Suppose that 
\begin{equation*}%\label{Ass_p=2}
|\nabla^{j}\mathrm{Riem}|\leq \lambda(r(x))^{2+j},\qquad 0\leq j\leq k-3.
\end{equation*}
Then we have that $W_{0}^{k,2}(M)=W^{k,2}(M)$.
\end{theorem}

\subsection{Other applications} 
Beyond the density problem for Sobolev spaces, other results can be obtained by means of Theorem \ref{th_distancefunction}, notably when $k=2$, i.e. for distance-like functions with controlled gradient and Hessian. Some of these applications were already observed in \cite{IRV-HessCutOff}, and include disturbed Sobolev inequalities, $L^2$-Calder\'on-Zygmund inequalities and the full Omori-Yau maximum principle for the Hessian. Since Theorem \ref{th_distancefunction} improves \cite[Theorem 1.5]{IRV-HessCutOff} also when $k=2$, by allowing for a wider class of admissible growth functions $\lambda$, the range of application of the results alluded to above is enlarged in this paper. This will be made explicit respectively in Theorem \ref{th_sob}, Theorem \ref{th_CZ_dist} and Subsection \ref{subsec_OY}.
On the way, we will take also the occasion to point out a couple of new applications of the 2nd-order distance-like functions which were not contained in \cite{IRV-HessCutOff}. The first one is a higher order version of the disturbed Sobolev inequality, i.e. 
\begin{align*}
\|\varphi\|_{L^{pm/(m-kp)}(M)}\leq A \sum_{r=-1}^{k-1}\left\| \left(\max\{1;r(x)\}\right)^{\frac\eta{mp}(r+1)(2m-rp)}|\nabla^{k-r-1}\varphi|\right \|_{L^p(M)},
\end{align*}
for all $\varphi\in C^\infty_c(M)$. Quite surprisingly, compared to the case $k=1$, this does not require any additional assumption; see Proposition \ref{th_sob_HO}. The second new application is the following disturbed global $L^p$ Calder\'on-Zygmund inequality for $p\neq 2$.
\begin{theorem}\label{th_CZp_dist}
Let $(M^m,g)$ be a smooth, complete non-compact Riemannian manifold without boundary. Let $o\in M$, $r(x)\doteq \mathrm{dist}_{g}(x,o)$ and suppose  that  for some $\eta>0$, some $D>0$ and some $i_0>0$,
\[
\ |\mathrm{Sect}|(x)\leq D^2(1+r(x)^2)^{\eta},\quad\mathrm{inj}_{g}(x)\geq \frac{i_0}{D(1+r(x))^{\eta}}>0\quad\mathrm{on}\,\,M.
\]
Then there exist constants $A>0$ depending on $m$, $\eta$, $D$, $i_0$ and the constant $C$ from Theorem \ref{th_distancefunction}, such that for all $\varphi\in C^\infty_c(M)$ it holds
\begin{align*}
\||\mathrm{Hess}\,\varphi|_g\|_{L^p}^p \leq  A \left[ \|H^{2\eta}\varphi\|_{L^p}^p + \|\Delta \varphi\|_{L^p}^p\right],
\end{align*}
where $H\in C^\infty(M)$ is the distance-like function given by Theorem \ref{th_distancefunction}, which satisfies in particular $H(x)\leq \max\{1;r(x)\}$.
\end{theorem}
Compared to the $L^2$-case alluded to above, the main difficulty here is to control the injectivity radius of the manifold at hand under a conformal deformation of the metric.  
  
%%%%%%%%%%%%%%%%%%%%%%%%%%%%%%%%%%%%%%%%%
\subsection{Structure of the paper}The rest of the paper is structured as follows.
In Section \ref{SectNotations}, for  the reader's convenience, we list some notations which will be used throughout the paper. Section \ref{SectHighOrdDistLike} is devoted to the construction of higher order distance-like functions described in  Theorem \ref{th_distancefunction}. These will be then used in Section \ref{SectCut&Dens} to construct the cut-off functions needed for the proof of the density property for Sobolev spaces stated in Theorem \ref{th_main1} which will be proved in the final part of the section. In Section \ref{p2} we focus on the special case $p=2$. We construct the  $k$-th order rough Laplacian cut-off functions, introduce Sampson's Lichnerowicz Laplacian acting on the space of smooth sections of symmetric $k$-covariant tensors and its Weitzenb\"ock formula, and give a proof of the density property stated in Theorem \ref{Dens_p=2}. The other applications of the higher order distance-like functions we mentioned above will be presented in Section \ref{SectSharpAppl}, which contains in particular a proof of Theorem \ref{th_CZp_dist}. For the sake of completeness, we end the paper with two appendices. In Appendix \ref{App_comm}, we give a proof of some commutation formulas which we use in Section \ref{p2}, while in Appendix \ref{AppB} we explicitly prove the Weitzenb\"ock formula for Sampson's Laplacian.
%%%%%%%%%%%%%%%%%%%%%%%%%%%%%%%%%%%%%%%%%
\section{Some notational conventions}\label{SectNotations}
The following are some notations and conventions which will be used throughout the paper. 
\smallskip

For any $\beta>0$, the Euclidean ball of radius $\beta$ centered at the origin will be denoted by $\B_{\beta}$. Given a smooth enough function $\lambda:\mathbb{R}\to\mathbb{R}$, its $j$-th derivative will be denoted by $\lambda^{(j)}$. Moreover, using standard notation,  given a smooth enough function $F:\mathbb R^m\to \mathbb R$  we  will denote by $DF$ the Euclidean gradient of $F$ and, given a multi-index $\gamma=(\gamma_1,\dots,\gamma_q)\in (\mathbb R^{m})^q$, we will denote $|\gamma|=q$ and $\partial_\gamma f=\partial^q_{\gamma_1\cdots \gamma_q}f\doteq\partial_{\gamma_1}\cdots\partial_{\gamma_q}f$. 
\smallskip

Given $(M^m, g)$ an $m$-dimensional complete Riemannian manifold, we will denote by $\nabla_{g}$, $\Delta_{g}\doteq\mathrm{div_{g}}\nabla_{g}$ and $\nabla^{j}_{g}$, respectively, the  corresponding Riemannian gradient, the (negative definite) Laplace-Beltrami operator and the natural $j$-th covariant derivative, thus specifying in the subscript the metric we are considering. Given a smooth enough  $f: M\to\mathbb{R}$, the Hessian of $f$, i.e. $\nabla_g^2f$, will be also denoted by $\mathrm{Hess}_{g}f$. An analogous convention will be used for various concepts associated to the metric such as the geodesic distance $d_{g}(\cdot, \cdot)$, open geodesic balls $B_{r}^{g}(x)$ of radius r centered at a point $x \in M$, the injectivity radius $\mathrm{inj}_{g}(x)$ at a point $x\in M$, the global injectivity radius $\mathrm{inj}_{g}(M)$, and the volume measure $d\mathrm{vol}_{g}$. Likewise, we will denote the Riemann curvature tensor of $(M,g)$ by $\mathrm{Riem}_{g}$, the Ricci tensor by $\mathrm{Ric}_{g}$ and the Sectional curvature of a two dimensional subspace $\pi\subset T_{x}M$ by $\mathrm{Sect}_{g}(\pi)$.  However, we will omit the subscritpt metric when the meaning is clear, as for instance for the fiber norms induced by $g$ (as in the expression $|\nabla_g f|$), as well as in all the Section \ref{p2} and in the Appendix \ref{App_comm} and \ref{AppB}, where we are considering a unique Riemannian metric $g$ on $M$.

Note also that when writing an inequality of the form $|\mathrm{Sect}_{g}|(x)\leq a(x)$  on $M$ we will mean that all the sectional curvatures at the point $x$ are dominated, in absolute value, by $a(x)$ for every $x\in M$.

In some parts of the paper, especially in Section \ref{p2} and in the appendices, some of the formulae which we will need can be significantly simplified using the following "$*$" notation, used e.g. in \cite{Topping}. We denote by $A*B$ any tensor field which is a real linear combination of tensor fields, each formed by starting with the tensor field $A\otimes B$, using the metric to switch the type of any number of $T^{*}M$ components to $TM$ components, or vice versa, taking any number of contractions, and switching any number of components in the product. 

Finally, all over the paper, $C$ will denote real constants, whose explicit value can possibly change from line to line. When it will seem appropriate we will specify some dependences of these constants at the subscript.

%%%%%%%%%%%%%%%%%%%%%%%%%%%%%%%%%%%%%%%%%
\section{Higher order distance-like functions}\label{SectHighOrdDistLike}

\subsection{On the assumptions}\label{subsec_assumptions}

Let $\lambda:[0,+\infty)\to[0,+\infty)$ be a $C^\infty$ function. In Section \ref{Intro}  we introduced the following assumptions:
\begin{itemize}
\item[(A1)] There exists a constant $R_1>0$ such that $\lambda$ is strictly positive and non-decreasing on $[R_1,+\infty)$. 
\item[(A2)] For every $\delta >0$, there exist constants $R_2=R_2(\delta)>0$ and $M_2=M_2(\delta)>1$ such that 
\[\forall\,t>R_2,\quad M_2(\delta)^{-1}\leq \frac{\lambda(\delta t)}{\lambda(t)}\leq M_2(\delta),\quad \text{and}\quad M_2(\delta)^{-1}\leq \frac{\lambda'(\delta t)}{\lambda'(t)}\leq M_2(\delta).
\]
\item[(A3)] There exist constants $R_{3}>0$ and $M_{3}>1$ such that 
\[\forall\,t>R_{3},\quad M_3^{-1}\leq \frac{t\lambda'(t)}{\lambda(t)}\leq M_3.
\]
%\item[(A4)] There exist constants $R_{4}>0$ and $M_{4}>1$ such that 
%\[\forall\,t>R_{4},\quad 0\leq \frac{t\lambda''(t)}{\lambda(t)}\leq M_4.
%\]
Moreover, for every integer $j\geq 1$ we have defined the property
\item[(A4($j$))] There exist constants $R_{4(j)}>0$ and $M_{4(j)}>1$ such that 
\[\forall\,t>R_{4(j)},\quad  \frac{t\lambda^{(j)}(t)}{\lambda^j(t)}\leq M_{4(j)}.
\]

%\item[(A5($j$))] $\lambda^{-j}\not\in L^1([1,+\infty))$
\end{itemize}

Let $\theta(t)\dot = t\lambda(t)$. We have the validity of the following
\begin{lemma}\label{lem_A7}
Assumption (A3) implies that for all $t>R_{3}$,
\[
(1+M_3^{-1})\lambda(t) \leq \theta'(t)\leq (1+M_3)\lambda(t).
\]
\end{lemma}
Let $j\geq 1$ be an integer. Since $\theta^{(j)}= t\lambda^{(j)}(t)+j\lambda^{(j-1)}(t)$, we have 
\begin{lemma}\label{lem_A8}
Assumptions (A4($j-1$)) and (A4($j$)) imply that there exist constants $R_{5(j)}>0$ and $M_{5(j)}>1$ such that \[\forall\,t>R_{5(j)},\quad  \theta^{(j)}(t)\leq M_{5(j)}\lambda^j(t).\]
\end{lemma}
Now, let $j\geq 0$. Define $\Theta(t)\dot=(\theta'(t))^2=(t\lambda'(t)+\lambda(t))^2$. Note that $\Theta^{(j)}$ is a linear combination of terms of the form $\theta^{(j_1)}\theta^{(j_2)}$, with $j_1$ and $j_2$ positive integers satisfying $j_1+j_2=j+2$. Then we have also the following
\begin{lemma}\label{lem_A6}
Suppose that assumptions (A4($s$)), $s=1,\dots,j+1$ are satisfied. Then there exist constants $R_{6(j)}>0$ and $M_{6(j)}>1$ such that \[\forall\,t>R_{6(j)},\quad  \Theta^{(j)}(t)\leq M_{6(j)}\lambda^{j+2}(t).\]
\end{lemma}

\begin{lemma}\label{l: lambda}
\begin{itemize}
\item[(i)] The function $\lambda(t)\dot = t^\eta$ with $\eta>0$ satisfies assumptions (A1), (A2), (A3) and (A4($j$)) for $j\geq 2$.
\item[(ii)] Suppose that
\[
\lambda(t)=
 \alpha t\prod_{j=1}^{\bar{j}}\ln^{[j]}(t)\]
for $t$ large enough, where $\alpha>0$ is a constant, $\bar{j}$ a positive integer and  we recall that $\ln^{[j]}$ stands for the $j$-th iterated logarithm. Then $\lambda$ satisfies assumptions (A1), (A2), (A3) and (A4($j$)) for $j\geq 2$.
\end{itemize}
\end{lemma}
\begin{proof} (i) (A1), (A2), (A3) are trivially satisfied. Concerning (A4($j$)) for $j\geq 2$, we have
\[
\frac{t\lambda^{(j)}(t)}{\lambda^j(t)}=C\frac{t^{\eta-j+1}}{t^{j\eta}}
\]
If $\eta \leq 1$, then $t^{\eta-j+1}\leq 1$ when $t>1$, whence the thesis. For $\eta>1$, we have $t^{\eta+1}\leq t^{\eta j}$ when $t>1$, which permits to conclude.\\

(ii) (A1) is trivial. Concerning (A2) note that
\begin{align*}
\frac{\lambda(\delta t)}{\lambda(t)}= \delta \prod_{j=1}^{\bar{j}} \frac{\ln^{[j]}(\delta t)}{\ln^{[j]}(t)},
\end{align*}
and $\frac{\ln^{[j]}(\delta t)}{\ln^{[j]}(t)} \to 1$ as $t\to\infty$ for every $j$. Similarly, $\lambda'(t)=\alpha+\alpha\sum_{q=1}^{\bar{j}}\prod_{j=q}^{\bar{j}}\ln^{[j]}(t)$, so that 
$\lambda'(t) \sim_{t\to\infty} \alpha\prod_{j=1}^{\bar{j}}\ln^{[j]}(t)$, and
\[
\lim_{t\to\infty} \frac{\lambda'(\delta t)}{\lambda'(t)} = \lim_{t\to\infty}  \frac{\prod_{j=1}^{\bar{j}}\ln^{[j]}(\delta t)}{\prod_{j=1}^{\bar{j}}\ln^{[j]}(t)}=1.
\]
Concerning (A3), we have that 
\[
\frac{t\lambda'(t)}{\lambda(t)} = \frac{\alpha t\left(1+\sum_{q=1}^{\bar{j}}\prod_{j=q}^{\bar{j}}\ln^{[j]}(t) \right)}{\lambda(t)} \sim_{t\to\infty}\frac{\alpha t\prod_{j=1}^{\bar{j}}\ln^{[j]}(t)}{\lambda(t)} =1
\]
Concerning (A4($j$)) it is enough to remark that $\lim_{t\to\infty}\lambda^{(j)}(t)=0$ and that $\lambda^j(t)\geq \alpha^j t^j\geq t$ when $t$ is large enough.
\end{proof}
%%%%%%%%%%%%%%%%%%%%%%%%%%%%%%%%%%%%%%%%%%
\subsection{A generalization of a result by Bianchi-Setti}

Schoen and Yau proved in \cite{SY} the existence of families of Laplacian cut-off functions on manifolds with lower bounded Ricci curvature. In \cite[Theorem 2.1]{BianchiSetti}, Bianchi and Setti generalized this result in two directions. On the one hand they got a better control on the Laplacian of the cut-offs when the negative part of the Ricci curvature vanishes at infinity. On the other hand they showed that one can get Laplacian cut-offs even when the negative part of the Ricci curvature explodes at infinity, provided that its growth is suitably controlled, i.e. $\mathrm{Ric}_g\geq - C r(x)^2$. In the following theorem, we extend this latter result to the assumption $\mathrm{Ric}_g\geq - \lambda^2(r(x))$ for a larger class of growth functions $\lambda$. This in particular will permit to sharpen applications to the density problem for Sobolev spaces, see Subsection \ref{ProofDens}.

\begin{theorem}\label{refinedBS}
Let $\lambda:[0,+\infty)\to[0,+\infty)$ be a $C^\infty$ function satisfying assumptions (A1), (A2), (A3). Let $(M^m, g)$ be a complete Riemannian manifold and $o\in M$ a fixed reference point, $r(x)\doteq \mathrm{dist}_{g}(x,o)$.
Suppose that for some $R_0>0$,
\[
\ \mathrm{Ric}_{g}(x)\geq -\lambda^2(r(x)),\quad\mathrm{on}\,\,M\setminus B_{R_0}(o).
\]
Let $\tilde r\in C^\infty (M \setminus {o})$ be an exhaustion function on $M$ such that for some $C_\tr>1$
\[\begin{cases}
C_\tr^{-1}r(x)\leq \tr(x) \leq C_\tr r(x)\\
\tr(x)=r(x),\quad\text{on }B_2(o)\setminus B_{1/2}(o)\\
|\nabla_g \tr|\leq 1.
\end{cases}
\]
Then there exists $h\in C^{\infty}(M)$ such that 
\begin{itemize}
\item[(i)] $C_h^{-1}r(x)\lambda(r(x))\leq h(x) \leq C_h\max\{1;r(x)\lambda(r(x))\}$;
\item[(ii)] $|\nabla_g h|\leq C_h\lambda(r(x))$ on $M$;
\item[(iii)] $\Delta_g h = |\nabla_g h|^2-C_\theta (\theta'(\tr(x)))^2$ on $M\setminus B_{R_\theta}(o)$ 
\end{itemize}
for some constants $C_\theta>0$, $R_\theta >0$ and $C_h>1$. Here $\theta$ is the function defined on $[0,+\infty)$ by $\theta(t)\dot =t\lambda(t)$.
\end{theorem}
Beyond the more restrictive choice for the growth function $\lambda$, in the original version of Theorem \ref{refinedBS} given in \cite[Theorem 2.1]{BianchiSetti}, the first-order distance-like function $\tilde r$ was required to satisfy $\|\tilde r -r\|_{L^\infty(M)}<\epsilon$ for an $\epsilon$ arbitrarily small. This is the case for instance if $\tilde r$ is a distance-like function provided by Gaffney's result \cite{Gaffney}. Instead, our Theorem \ref{refinedBS} works for general first-order distance-like functions. This represents a minor technical improvement, which will reveal however essential in proving Theorem \ref{th_distancefunction}.

\begin{proof}[Proof (of Theorem \ref{refinedBS})]
The proof mimics the proof of Theorem 2.1 in \cite{BianchiSetti}. The main difference is to modify the equation of which $\log h$ will be a solution, according to our more general set of assumptions. We will sketch the relevant changes. Also, if $\lambda$ is bounded then the result is contained in \cite[Theorem 2.1]{BianchiSetti}. We can thus assume from now on that $\lambda$ is unbounded and $\lambda(t)\to\infty$ as $t\to\infty$.
\smallskip

We start pointing out the following lemma. 
\begin{lemma}\label{lem_SY}
Let $\lambda:[0,+\infty)\to[0,+\infty)$ be a $C^\infty$ function satisfying assumptions (A1) and (A2). Let $(M^m, g)$ be a complete Riemannian manifold and $o\in M$ a fixed reference point, $r(x)\doteq \mathrm{dist}_{g}(x,o)$. Suppose that for some $R_0>0$,
\[
\ \mathrm{Ric}_{g}(x)\geq -\lambda^2(r(x)),\quad\mathrm{on}\,\,M\setminus B_{R_0}(o).
\]
Then there exists $\alpha>1$, $C_{SY}>0$ and $R_\alpha>0$ such that for every $x\in M\setminus B_{R_\alpha}(o)$ it holds
\[
\mathrm{vol}_g(B_{1/4}(x))\geq C_{SY}e^{-\alpha r(x)\lambda(r(x))}.\]
\end{lemma}

The original idea is due to Schoen and Yau, \cite{SY}, while in \cite[Proposition 2.11]{BianchiSetti} the authors consider bounds given by (possibly negative) powers of $r(x)$. 

\begin{proof}[Proof (of Lemma \ref{lem_SY})]
Set $G_x\dot = \lambda^2(2r(x)+1)$. By continuity, since $\lambda(t)\to\infty$ as $t\to\infty$, if $r(x)$ is large enough, then $\mathrm{Ric}_g\geq - G_x$ on $B_{2r(x)+1}(o)$. By the Bishop-Gromov comparison,
\begin{align*}
\mathrm{vol}_g(B_{1/4}(x))\geq& \mathrm{vol}_g(B_{r(x)+1}(x))\frac{V_{G_x}(1/4)}{V_{G_x}(r(x)+1)}\geq \mathrm{vol}_g(B_{1}(x))\frac{V_{G_x}(1/4)}{V_{G_x}(r(x)+1)}
\\
\geq &\mathrm{vol}_g(B_{1}(x))\frac{V_0(1/4)}{V_{G_x}(r(x)+1)},
\end{align*}
where $V_{G_x}(t)$ represents the volume of a ball of radius $t$ in the $m$-dimensional simply connected space-form of constant negative curvature $-G_x$, and $V_0(1/4)$ is the the volume of a ball of radius $1/4$ in $\R^m$. A standard computation using assumption (A2) shows that there exist constants $C_1>0$, $R_\alpha>0$ and $\alpha>1$ independent of $x$ such that
\begin{align*}
V_{G_x}(r(x)+1)\leq C_1 e^{\alpha r(x)\lambda(r(x))},
\end{align*}
as soon as $r(x)>R_\alpha$. Setting $C_{SY}=C_1^{-1}\mathrm{vol}_g(B_{1}(x))V_0(1/4)$ concludes the proof of the lemma.
\end{proof}

Recall that $\theta(t)=t\lambda (t)$. Proceeding as in \cite[Theorem 2.1]{BianchiSetti} we can produce a function $\omega\in C^\infty(M\setminus B_1(o))$ solution of
\begin{align}\label{eq-omega}
\begin{cases}
\Delta_g \omega= C^2_e(\theta'(\tr(x)))^2\omega,\\
\omega|_{\partial B_1}=1,\\
0<\omega|_{M\setminus\bar B_1(o)}<1,
\end{cases}
\end{align}
with $C_e$ a large enough constant to be fixed later. Computing as in \cite[page 6]{BianchiSetti} we get that for every $0<\gamma<\sqrt{C_e^2}/2$ and every $x\in M\setminus B_2(o)$, 
\begin{equation}\label{loc-int}
\int_{B_{1/4}(x)}(\theta'(\tr(y)))^2 e^{\gamma \theta(\tr(y))}\omega^2 d\textrm{vol}_g(y)\leq \int_{M\setminus B_1(o)}(\theta'(\tr(y)))^2 e^{\gamma \theta(\tr(y))}\omega^2 d\textrm{vol}_g(y) \leq \frac{A}{C_e^2-4\gamma^2},
\end{equation}
where $A=e^{\lambda(1)}\int_{\partial B_1(o)}|\nabla_g\omega|d\textrm{vol}_{m-1}$ .
We need the following Li-Yau gradient estimate.
\begin{lemma}\label{lem_LY}
There exists a constant $C_{LY}>0$ and $R_{LY}>0$ such that for every $R>R_{LY}$ and for every $x\in B_{R+1}(o)\setminus B_{R_{LY}}(o)$, 
\[
\frac{|\nabla_g\omega|^2}{\omega^2}\leq C_{LY}^2 C^2_e\lambda^2(R).
\]
\end{lemma}
\begin{proof}
A special case of Theorem 2.8 in \cite{BianchiSetti} ensures the existence of a constant $C_{LY}'>0$ such that
\[
\frac{|\nabla_g\omega|^2}{\omega^2}\leq C'_{LY}\max\{\lambda^2(R); C_e^2(\theta'(R))^2 \}.
\]
By Lemma \ref{lem_A7}, $\theta'(R)\leq (1+M_3) \lambda(R)$ for $R$ large enough, which proves the lemma.
\end{proof} 
Reasoning as in \cite[page 6]{BianchiSetti}, thanks to Lemma \ref{lem_LY}, for every $x\in M\setminus B_3(o)$ and every $y\in B_{1/4}(x)$, we have 
\begin{equation}\label{loc-diff}
\omega(y)\geq\omega(x)e^{-\frac{1}{4}C_{LY}C_e\lambda(r(x))}.
\end{equation}
Note also that there exists $\delta>0$, $\kappa>0$ and $R_\delta>0$ such that for every $x\in M\setminus B_{R_\delta}(o)$ and every $y\in B_{1/4}(x)$,
\begin{equation}\label{loc-th}
(\theta'(\tr(y)))^2 e^{\gamma \theta(\tr(y))}\geq \delta(\theta'(\tr(x)))^2 e^{\kappa\gamma \theta(\tr(x))}.
\end{equation}
This follows from assumption (A2) and the fact that $|\tr(x)-\tr(y)|<1/4$ (because $|\nabla_g \tr|\leq 1$). Inserting \eqref{loc-diff} and \eqref{loc-th} in \eqref{loc-int} gives
\[
\omega^2(x)\delta(\theta'(\tr(x)))^2 e^{\kappa\gamma \theta(\tr(x))}e^{-\frac{1}{2}C_{LY}C_e\lambda(r(x))}\mathrm{vol}_g(B_{1/4}(x)) \leq \frac{A}{C_e^2-4\gamma^2}.
\]
From Lemma \ref{lem_SY}, we deduce that
\begin{align*}
\omega^2(x)\leq C_{SY}^{-1}\delta^{-1}\frac{A}{C_e^2-4\gamma^2}(\theta'(\tr(x)))^{-2}e^{-\kappa\gamma\theta(\tr(x))+\frac{1}{2}C_{LY}C_e\lambda(r(x))+\alpha r(x)\lambda(r(x))}
\end{align*}
whenever $r(x)>R_1 \dot =\max \{R_\alpha;R_\delta;R_{LY};3\}$. Recalling assumption (A1) and (A2), we get
\[
\theta(\tr(x))\geq \theta (C_\tr^{-1} r(x)) \geq \zeta \theta(r(x)),
\]
with $\zeta= C_\tr^{-1}M_2(C_\tr^{-1})^{-1}>0$ independent of $x$. Hence
\[
-\kappa\gamma\theta(\tr(x))+\frac{1}{2}C_{LY}C_e\lambda(r(x))+\alpha r(x)\lambda(r(x))\leq -\left[\kappa\gamma\zeta-\alpha\right] r(x)\lambda(r(x))+\frac{1}{2}C_{LY}C_e \lambda(r(x)).
\]
At this point we can make a choice of $\gamma > \frac{3\alpha}{\kappa\zeta}$, so that $\kappa\gamma\zeta-\alpha>2\alpha$, and a choice of $C_e> 2\gamma$. There exists a radius $R_2>0$ such that 
\[ -\left[\kappa\gamma\zeta-\alpha\right] r(x)\lambda(r(x))+\frac{1}{2}C_{LY}C_e \lambda(r(x))\leq \alpha r(x)\lambda(r(x))
\]
for all $x\in M\setminus B_{R_2}(o)$. Because of assumption (A1) and Lemma \ref{lem_A7}, $\theta'(t)$ is lower bounded. Hence we have obtained that 
\begin{align*}%\label{omega-bd}
\omega^2(x)\leq C_{SY}^{-1}\delta^{-1}\frac{A}{C_e^2-4\gamma^2}e^{-\alpha r(x)\lambda(r(x))}
\end{align*}
on $M\setminus B_{R_3}(o)$ for some $R_3\geq R_2$.
Define 
\[h(x)\dot = (\phi(x)-1)\ln(\omega(x))+\phi(x),\]
where $\phi\in C^\infty(M)$ is such that $\phi(x)\equiv 1$ in $B_2(o)$ and $\phi(x)\equiv 0$ in $M\setminus B_3(o)$. We have that on $M\setminus B_{R_3}(o)$
\begin{align*}
h(x)= -\ln(\omega(x))\geq \ln \left(C_{SY}\delta\frac{C_e^2-4\gamma^2}{A}\right) + \alpha r(x)\lambda(r(x)),
\end{align*}
from which we deduce the existence of a radius $R_4>R_3$ such that 
\begin{align}\label{ubdh}
h(x)\geq \frac\alpha 2 r(x)\lambda(r(x))
\end{align}
on $M\setminus B_{R_4}(o)$. On the other hand , since $h(o)=1$, thanks to Lemma \ref{lem_LY} we can compute
\begin{align}\label{lbdh}
h(x)=&h(o)+\int_0^{r(x)}|\nabla_g h(\sigma(s))|ds \leq 1+ R_{LY}\max_{B_{R_{LY}}}|\nabla_g h| + \int_{R_{LY}}^{r(x)} C_{LY}C_e \lambda(r(x)) ds\\
\leq& \alpha' r(x)\lambda(r(x)),\nonumber
\end{align}
for some constant $\alpha'$ independent of $x$. By definition of the smooth function $h$ and up to choose $C_h>\max\{\alpha',2/\alpha, C_{LY}C_e\}$ large enough, as a consequence of \eqref{ubdh}, \eqref{lbdh} and Lemma \ref{lem_LY}, we have thus obtained the validity of the points (i) and (ii) in the statement of Theorem \ref{refinedBS}. The point (iii) follows from \eqref{eq-omega} and a direct computation. 
\end{proof}
%%%%%%%%%%%%%%%%%%%%%%%%%%%%%%%%%%%%%%%%%

\subsection{Proof of Theorem \ref{th_distancefunction}: case (a)}\label{casea}
We are going to prove Theorem \ref{th_distancefunction} in the set of assumptions (a), i.e.
\begin{theorem}\label{th_distancefunctionRic}
Let $(M^m, g)$ be a complete Riemannian manifold and $o\in M$ a fixed reference point, $r(x)\doteq \mathrm{dist}_{g}(x,o)$.  Let $k\in \mathbb N^{+}$.  
If $k\geq 2$, suppose in addition that for some $i_0>0$,
\[
\ |\nabla_g^j\mathrm{Ric}_{g}|(x)\leq \lambda(r(x))^{2+j},\ 0\leq j \leq k-2,\quad\mathrm{inj}_{g}(x)\geq \frac{i_0}{\lambda(r(x))}>0\quad\mathrm{on}\,\,M,
\]
where the function $\lambda$ satisfies assumptions (A1), (A2), (A3), (A4($j$)) for $j=1,\dots,k$.
Then there exists an exhaustion function $H=H_k\in C^{\infty}(M)$ such that for some positive constant $C_k>1$ independent of $x$, we have on $M$ that 
\begin{itemize}
\item[(i)]$C_k^{-2}r(x)\leq H(x)\leq \max\left\{r(x), 1\right\}$;
\item[(ii)] for $1\leq j \leq k$, $|\nabla_g^{j} H|(x)\leq C_k^{j-1}\max\{\lambda(r(x))^{j-1},1\}$.
\end{itemize}
\end{theorem}
The proof is done by induction on $k$. For $k=1$, the function $\lambda$ is not involved in the statement, and this is the content of the well-known theorem by Gaffney, \cite{Gaffney,GWAnnSci}. Assume now that for some $k\geq 2$ and some $i_0>0$
\[
\ |\nabla_g^j\mathrm{Ric}_{g}|(x)\leq \lambda(r(x))^{2+j},\ 0\leq j \leq k-2,\quad\mathrm{inj}_{g}(x)\geq \frac{i_0}{\lambda(r(x))}>0\quad\mathrm{on}\,\,M.
\]
As induction hypothesis we suppose that the result holds true for $k-1$, i.e. there exists a distance-like function $H_{k-1}\in C^{\infty}(M)$ such that for some positive constant $C_{k-1}>1$ independent of $x$, we have on $M$ that 
\begin{itemize}
\item[(i)]$C_{k-1}^{-2}r(x)\leq H_{k-1}(x)\leq \max\left\{r(x), 1\right\}$;
\item[(ii)] for $1\leq j \leq k-1$, $|\nabla_g^{j} H_{k-1}|(x)\leq C_{k-1}^{j-1}\max\{\lambda(r(x))^{j-1},1\}$.
\end{itemize}
By Theorem \ref{refinedBS} we have the following

\begin{proposition}\label{BS}
In the assumptions of Theorem \ref{th_distancefunctionRic}, there exists a function $h=h_k\in C^{\infty}(M)$ and a constant $C_h>1$ such that 
\begin{itemize}
\item[(i)] $C_h^{-1}r(x)\lambda(r(x))\leq h(x) \leq C_h\max\{1;r(x)\lambda(r(x))\}$;
\item[(ii)] $|\nabla_g h|_g\leq C_h\lambda(r(x))$ on $M$;
\item[(iii)] $\Delta_g h = |\nabla_g h|^2_g-C_\theta (\theta'(H_{k-1}(x)))^2$ on $M\setminus B_{R_\theta}(o)$ for some constants $C_\theta>0$ and $R_\theta >0$.
\end{itemize}
\end{proposition}

Recall that a local coordinate system $\left\{x^{i}\right\}_{i=1}^m$ is said to be harmonic if for any $i$, $\Delta_{g}x^{i}=0$. 
As a consequence of \cite{Anderson} we have the validity of the following 

\begin{proposition}\label{HarmRadEst}
Let $\alpha\in(0,1)$, $Q>1$, $\delta>0$. Let $(M^m,g)$ be a smooth Riemannian manifold, and $\Omega$ an open subset of $M$. Set
\[
\ \Omega(\delta)=\left\{x\in M\quad\mathrm{s.t.}\quad d_{g}(x,\Omega)<\delta\right\}.
\]
Suppose that
\[
\ |\nabla_{g}^j\mathrm{Ric}_{g}|(x)\leq 1,\ 0\leq j\leq k-2\quad\mathrm{and}\quad\mathrm{inj}_{g}(x)\geq i\quad\mathrm{for\,\, all}\quad x\in\Omega(\delta),
\]
then there exists a positive constant $C_{HR}=C_{HR}(m,Q,k,\alpha,\delta, i)$, such that  on the geodesic ball $B_{C_{HR}}(x)$ of center $x$ and radius $C_{HR}$, there is a centered harmonic coordinate chart such that the metric tensor is $C^{k-1,\alpha}$ controlled in these coordinates. In particular, if $g_{ij}$, $i,j=1,\ldots,m$, are the components of $g$ in these coordinates, then
\begin{enumerate}
\item $Q^{-1}\delta_{ij}\leq g_{ij}\leq Q\delta_{ij}$ as bilinear forms;
\item $\sum_{1\leq|\gamma|\leq k-1}r_{H}^{|\gamma|}\sup_y\left|\partial_{\gamma}g_{ij}(y)\right|\leq Q-1$
\end{enumerate}
\end{proposition}

%Let $\lambda:[0,\infty)\to[0,\infty)$ be smooth positive increasing convex function. 
%Suppose that for some $D>0$ and some $i_0>0$,

By assumption, we have that
\[
\ |\nabla_{g}^j\mathrm{Ric}_{g}|(x)\leq \lambda(r(x))^{2+j},\ 0\leq j \leq k-2,\quad\mathrm{inj}_{g}(x)\geq \frac{i_0}{\lambda(r(x))}>0\quad\mathrm{on}\,\,M.
\]

Fix $x\in M\setminus B_2^g(o)$ and set $$\lambda_1=\lambda_1(x)\doteq\lambda(r(x)+1).$$ Since $\lambda$ is increasing, for all $y\in B_1^g(x)$ we have 
\[
\ |\nabla_{g}^j\mathrm{Ric}_{g}|(y)\leq \lambda_1^{2+j},\ 0\leq j\leq k-2,\quad\mathrm{and}\quad\mathrm{inj}_{g}(y)\geq \frac{i_0}{\lambda_1}.
\]
We define a new rescaled metric $g_\lambda\doteq \lambda_1^2g$ on $M$, so that for all $y\in B_{\lambda_1}^{g_\lambda}(x)$ we have 
\[
\ |\nabla_{g_{\lambda}}^j\mathrm{Ric}_{g_{\lambda}}|(y)\leq 1,\ 0\leq j\leq k-2,\quad\mathrm{and}\quad\mathrm{inj}_{g_{\lambda}}(y)\geq \lambda_{1}\mathrm{inj}_{g}(y)\geq = i_0.
\]

According to Proposition \ref{HarmRadEst} ( applied with $\delta=1/2$) there exists a constant $C_{HR}=C_{HR}(m,Q,i_0,k)$ such that on $B_{C_{HR}}^{g_\lambda}(x)$ there exists a centered harmonic chart $\varphi_H=(y^1,\dots,y^m):B_{C_{HR}}^{g_\lambda}(x)\to U\subset \mathbb R^m$ such that $\varphi_{H}(x)=0\in\R^m$ and, setting $\hat g_\l=g_\lambda \circ  \varphi_H^{-1}$, it holds

\begin{enumerate}
\item[$(1_{HC})$] $Q^{-1}\delta_{ij}\leq \ton{\hat g_\l}_{ij}\leq Q\delta_{ij}$ as bilinear forms;
\item[$(2_{HC})$] $\sum_{1\leq|\gamma|\leq k-1}C_{HR}^{|\gamma|}\sup_y\left|\partial_{\gamma}\ton{\hat g_\l}_{ij}(y)\right|\leq Q-1$,
\end{enumerate}
Since 
\[
\ \partial_{q}\hat g_\l^{ij}=-\partial_{q}\ton{\hat g_\l}_{lk}\hg_\l^{il}\hg_\l^{kj},
\]
we have also that
\begin{itemize}
\item[$(1'_{HC})$] $Q^{-1}\delta^{ij}\leq \hat g_\l^{ij}\leq Q\delta^{ij}$;
\item[$(2'_{HC})$] $\sum_{1\leq|\gamma|\leq k-1}\sup_y\left|\partial_{\gamma}\hg_\l^{ij}(y)\right|\leq C(Q)$,
\end{itemize}
for some constant $C(Q)$, depending only on $Q$.
\medskip

For $\beta\doteq C_{HR}/Q$, on $\mathbb B_{\beta}$ we set $\hat{g}=g\circ\varphi_{H}^{-1}$, $\hat h =h\circ \varphi_H^{-1}$ and $\hat H =H_{k-1}\circ \varphi_H^{-1}$.
\medskip

We need the following fundamental lemma.

\begin{lemma}\label{lem_main}
Let $i$ and $j$ be integers in $\{1,\dots,m\}$. Let $0\leq q \leq k-1$ and let $\gamma=(\gamma_1,\dots,\gamma_{q-1})\in \{1,\dots,m\}^{q-1}$ be a multi-index (possibly empty for $q=0,1$). There exists a positive constant $C>0$ such that 
\begin{equation}\label{1_lem_main}
\forall \,v\in\B_{\beta},\qquad
|\partial_i\hat h|(v)\leq C,
\end{equation}
and, if $q\geq 1$,
\begin{equation}\label{2_lem_main}
\forall \,v\in\B_{4^{-q}\beta},\qquad \left|\partial^2_{ij}\partial^{q-1}_{\gamma_1\cdots\gamma_{q-1}}\hat h\right| (v)\leq C.
\end{equation}
\end{lemma}

Before proving Lemma \ref{lem_main}, let us point out the following computational lemma which will be used repeatedly.
\begin{lemma}\label{lem_conf}
Let $F\in C^\infty(M)$. Set $\hat F\doteq F\circ\varphi_{H}^{-1}$ and let $v\in\B_{\beta}$. Let $d\leq k$. Then there exists a real positive constant $C_1$ independent of $x$ such that 
\begin{enumerate}
\item If $\abs{\p^n_{e_1\cdots e_n}\hat F}(v) \leq C_0$ for some constant $C_0>0$ (possibly depending on $x$), for all $n\in\{1,\dots,d\}$ and for all multi-index $(e_1,\dots,e_n)\in \{1,\dots,m\}^n$, then $\abs{\nabla_{\hg_\l}^n\hat F}_{\hg_\l}(v)\leq C_1C_0$ for all $n\in\{1,\dots,d\}$.
\item If $\abs{\nabla_{\hg_\l}^n\hat F}_{\hg_\l}(v) \leq C_0$ for some constant $C_0>0$ (possibly depending on $x$) and for all $n\in\{1,\dots,d\}$, then $\abs{\p^n_{e_1\cdots e_n}\hat F}(v)\leq C_1C_0$ for all $n\in\{1,\dots,d\}$ and for all multi-index $(e_1,\dots,e_n)\in \{1,\dots,m\}^n$.
\item $\l_1^{d}\abs{\nabla_{\hg_\l}^d\hat F}_{\hg_\l}= \abs{\nabla_{\hg}^d\hat F}_{\hg}$.
\item $ \abs{\nabla_{\hg}^d\hat F}_{\hg}(v)=\abs{\nabla_{g}^d F}_{g}(\varphi_H^{-1}(v))$ .
\end{enumerate}
\end{lemma}

\begin{proof}
As for (1) and (2) we proceed by induction on $d$. For $d=0$, the result is trivial with $C_1=1$. Suppose that we have proved the assertions for some $d-1$ and let us prove it for $d$. Clearly the only nontrivial case is $n=d$.

$(1)$. By definition
\[
\abs{\nabla_{\hg_\l}^d\hat F}_{\hg_\l}^2=\ton{\nabla_{\hg_\l}^d\hat F}_{e_1\dots e_d}\ton{\nabla_{\hg_\l}^d\hat F}_{f_1\dots f_d}\hg_\l^{e_1f_1} \cdots \hg_\l^{e_df_d}
\]
The expression in coordinates of $\ton{\nabla_{\hg_\l}^d\hat F}_{e_1\dots e_d}$ can be written as a linear combination of terms, each of which is given by the product of factors of the form $\p_{\alpha}\hat{F}$ with factors of the form $\p_{\beta}(^{\hat{g}_{\lambda}}\Gamma_{ij}^k)$, where the multi-indexes $\alpha$ and $\beta$ have orders, respectively, $|\alpha|\leq d$ and $|\beta|\leq d-2$. Note that $\p^{\beta}(^{\hat{g}_{\lambda}}\Gamma_{ij}^k)$ contains partial derivatives of $\hg_\lambda$ and $\hg_\lambda^{-1}$ of order at most $d-1$, which are uniformly bounded as soon as $d\leq k$ by the properties of the harmonic chart $\varphi_{H}$.  

$(2)$. Reasoning as for the previous implication, we get 
\begin{align*}
\abs{\p^d_{e_1\cdots e_n}\hat F}(v)&=\abs{\ton{D^d\hat F}_{e_1\dots e_d}}(v)\\
&\leq \abs{\ton{\nabla_{\hg_\l}^d\hat F}_{e_1\dots e_d}}(v)+C\left(\abs{\nabla_{\hg_\l}^{d-1}\hat F}+\dots+\abs{\nabla_{\hg_\l}^{1}\hat F}\right) \\
&\leq \abs{\nabla_{\hg_\l}^d\hat F}_{\hg_\l}\abs{\p_{e_1}}_{\hg_\l} \cdots \abs{\p_{e_d}}_{\hg_\l}+C\\
&\leq C.
\end{align*}

$(3)$. Note that
\[\ton{\nabla_{\hg_\l}^d\hat F}_{e_1\dots e_d}=\ton{\nabla_{\hg}^d\hat F}_{e_1\dots e_d}.\]
Accordingly,
\begin{align*}
\abs{\nabla_{\hg_\l}^d\hat F}_{\hg_\l}^2
&=\ton{\nabla_{\hg_\l}^d\hat F}_{e_1\dots e_d}\ton{\nabla_{\hg_\l}^d\hat F}_{f_1\dots f_d}\hg_\l^{e_1f_1} \cdots \hg_\l^{e_df_d}\\
&= \ton{\nabla_{\hg}^d\hat F}_{e_1\dots e_d}\ton{\nabla_{\hg}^d\hat F}_{f_1\dots f_d}\l_1^{-2}\hg^{e_1f_1} \cdots \l_1^{-2}\hg^{e_df_d} \\
&=\l_1^{-2d} \abs{\nabla_{\hg}^d\hat F}_{\hg}^2.
\end{align*}

$(4)$. By definition of $\hg$.
\end{proof}

\begin{proof}[Proof (of Lemma \ref{lem_main})]

The property \eqref{1_lem_main} is a direct consequence of (ii) in Proposition \ref{BS} and of Lemma \ref{lem_conf}. In particular, Lemma \ref{lem_main} is verified when $q=0$. It remains to prove \eqref{2_lem_main} for $q\geq 1$. Let us proceed by induction on $q$.
Suppose that the lemma is proved for all the integers $q$ between $0$ and ${\bar q}$, with $0\leq {\bar q} \leq k-2$. Let us prove it for ${\bar q}+1$. 

Define $\Theta(t)\dot=(\theta'(t))^2=(t\lambda'(t)+\lambda(t))^2$. By Proposition \ref{BS} (iii), $h$ satisfies 
\[
\Delta_gh(x)=|\nabla_g h|^2(x)-C_\theta \Theta(H_{k-1}(x))
\]
on $M\setminus B_{R_\theta}(o)$, which reads in harmonic coordinates as
\[
\lambda_1^{2}\hat g_\l^{ij}(v)\partial^2_{ij}\hat h (v) = \lambda_1^{2}\hg_\l^{ij}(v)\partial_i\hh(v)\partial_j\hh(v)-C_\theta \Theta(\hH(v)),\qquad \forall v\in \mathbb B_\beta.
\]
Applying the differential operator $\partial^{\bar q}_{\gamma_1\cdots\gamma_{\bar q}}$ to both sides, and multiplying by $\lambda_1^{-2}$, we get on $\B_{\beta}$
\begin{align}\label{eq_f}
\hg_\l^{ij}\partial^2_{ij}\left(\p^{\bar q}_{\gamma_1\cdots\gamma_{\bar q}}\hh\right)
=& \sum \left(\partial^n_{a_1\cdots a_n}\hg_\l^{ij}\right)\left(\p^l_{b_1\cdots b_l}\partial_i\hh\right)\left(\p^{{\bar q}-n-l}_{c_1\cdots c_{{\bar q}-n-l}}\partial_j\hh\right)
-C_\theta \lambda_1^{-2} \p^{\bar q}_{\gamma_1\cdots\gamma_{\bar q}}(\Theta(\hH))\\
&- \sum \left(\partial^n_{a_1\cdots a_n} \p^2_{ij} \hh\right)\left( \p^{{\bar q}-n}_{b_1\cdots b_{{\bar q}-n}}\hg_\l^{ij}\right) \nonumber\\
\doteq& \hat f.\nonumber
\end{align}
Here the first summation is taken over all the (possibly empty) subsets $\{a_s\}_{s=1}^n\subset\{\gamma_s\}_{s=1}^{\bar q}$ and $\{b_s\}_{s=1}^l\subset\{\gamma_s\}_{s=1}^{\bar q}\setminus \{a_s\}_{s=1}^n$, and we have set $\{c_s\}_{s=1}^{{\bar q}-n-l}\doteq\{\gamma_s\}_{s=1}^{\bar q}\setminus (\{a_s\}_{s=1}^n \cup \{b_s\}_{s=1}^l)$.
Similarly the second summation is taken over all the non-empty subsets $\{b_s\}_{s=1}^{{\bar q}-n}\subset\{\gamma_s\}_{s=1}^{\bar q}$ and we have set $\{a_s\}_{s=1}^{n}\doteq\{\gamma_s\}_{s=1}^{\bar q}\setminus \{b_s\}_{s=1}^{{\bar q}-n}$.
Fix $w\in \B_{4^{-{\bar q}}\beta}$. We need to prove that $\left|\partial^2_{ij}\partial^{{\bar q -1}}_{\gamma_1\cdots\gamma_{{\bar q-1}}}\hat h\right|(w) \leq C$ for some constant $C$ independent of $x$ and $w$. 
Define $\th\in C^\infty(\B_{\beta})$ by 
\[
\tilde h (v)\doteq\hh (v) - \sum_{j=0}^{\bar q} \frac{D^j \hh(w)(v,\dots,v)}{j!}.
\]
Note that $\th-\hh$ is a polynomial of degree ${\bar q}$, so that 
\[
\hg_\l^{ij}\p^2_{ij}\p^{\bar q}_{\gamma_1\cdots\gamma_{\bar q}}\th =\hg_\l^{ij}\p^2_{ij}\p^{\bar q}_{\gamma_1\cdots\gamma_{\bar q}}\hh = \hat f.
\]
Rescaling to $\B_{4^{-{\bar q}}\beta}(w)$ the elliptic Schauder estimates given for instance in \cite[Theorem 5.21]{HanLin}, we know that there exists $C>0$, depending on ${\bar q}$ (hence on $k$), but not on $w$, such that 
\begin{align}\label{schauder}
\abs{\p^2_{ij}\p^{\bar q}_{\gamma_1\cdots\gamma_{\bar q}}\hh(w)}
=&\abs{\p^2_{ij}\p^{\bar q}_{\gamma_1\cdots\gamma_{\bar q}}\th(w)}\\
 \leq& C\left\{\norm{\p^{\bar q}_{\gamma_1\cdots\gamma_{\bar q}}\th}_{L^\infty(\B_{4^{-{\bar q}}\beta}(w))} + \norm{\hat f}_{L^\infty(\B_{4^{-{\bar q}}\beta}(w))} + \qua{\hat f}_{C^{0,\alpha}(\B_{4^{-{\bar q}}\beta}(w))}\right\}.\nonumber
\end{align}
Note that $\B_{4^{-{\bar q}}\beta}(w)\subset \B_{4^{1-{\bar q}}\beta}(0)$. We will estimate the three terms at RHS of \eqref{schauder} separately.

\textbf{1st term.} \textit{Estimating $\norm{\p^{\bar q}_{\gamma_1\dots\gamma_{\bar q}}\th}_{L^\infty(\B_{4^{-{\bar q}}\beta}(w))}$.}

Let $v\in\B_{4^{-{\bar q}}\beta}(w)$. Note that
\[
\p^{\bar q}_{\gamma_1\dots\gamma_{\bar q}}\sum_{j=0}^{\bar q} \frac{D^j\hh(w)(v,\dots,v)}{j!} = 
\p^{\bar q}_{\gamma_1\dots\gamma_{\bar q}} \frac{D^{\bar q} \hh(w)(v,\dots,v)}{{\bar q}!} = D^{\bar q}\hh(w)(\partial_{\gamma_1},\dots,\partial_{\gamma_{\bar q}}).
\]
Hence
\begin{align*}
\abs{\p^{\bar q}_{\gamma_1\dots\gamma_{\bar q}}\th}(v) 
&= \abs{D^{\bar q}\th(v)(\partial_{\gamma_1},\dots,\partial_{\gamma_{\bar q}})}\\
&=\abs{D^{\bar q}\hh(v)(\partial_{\gamma_1},\dots,\partial_{\gamma_{\bar q}})-D^{\bar q}\hh(w)(\partial_{\gamma_1},\dots,\partial_{\gamma_{\bar q}})}\\
&\leq 2\beta\norm{D^{{\bar q}+1}\hh}_{L^\infty(\B_{4^{-{\bar q}}\beta}(w))}\\
& \leq C
\end{align*}
by the inductive hypothesis of the lemma.

\textbf{2nd term.} \textit{Estimating $\norm{\hat f}_{L^\infty(\B_{4^{-{\bar q}}\beta}(w))}$.}

Concerning the first addend in $\hat f$, since  $n\leq k-1$, by the property $(2'_{HC})$ of the harmonic chart we can compute
\begin{align}\label{Linfty2}
&\norm{\left(\partial^n_{a_1\cdots a_n}\hg_\l^{ij}\right)\left(\p^l_{b_1\cdots b_l}\partial_i\hh\right)\left(\p^{{\bar q}-n-l}_{c_1\cdots c_{{\bar q}-n-l}}\partial_j\hh\right)}_{L^\infty(\B_{4^{1-{\bar q}}\beta}(w))}\\
\leq& C \norm{\p^l_{b_1\cdots b_l}\partial_i\hh}_{L^\infty(\B_{4^{1-{\bar q}}\beta}(w))}
\norm{\p^{{\bar q}-n-l}_{c_1\cdots c_{{\bar q}-n-l}}\partial_j\hh}_{L^\infty(\B_{4^{1-{\bar q}}\beta}(w))}\nonumber\\
\leq& C, \nonumber
%&\leq C r^{({\bar q}+2)\eta}\nonumber,
\end{align}
where we have used the inductive hypothesis and the fact that $l+1\leq {\bar q}+1$ and ${\bar q}-n-l+1 \leq {\bar q}+1$.

We now consider the second addend of $\hat{f}$ and estimate $\norm{C_\theta \lambda_1^{-2} \p^{\bar q}_{\gamma_1\cdots\gamma_{\bar q}}(\Theta(\hH))}_{L^\infty(\B_{4^{1-{\bar q}}\beta}(w))}$. If $\bar q =0$, then 
\[
\norm{C_\theta \lambda_1^{-2} \p^{\bar q}_{\gamma_1\cdots\gamma_{\bar q}}(\Theta(\hH))}_{L^\infty(\B_{4^{1-{\bar q}}\beta}(w))}=\norm{C_\theta \lambda_1^{-2} \Theta(\hH)}_{L^\infty(\B_{4^{1-{\bar q}}\beta}(w))}\leq C
\]
as a direct consequence of Lemma \ref{lem_A6}. If $\bar q \geq 1$, then $\p^{\bar q}_{\gamma_1\cdots\gamma_{\bar q}}(\Theta(\hH))$ can be developed as a linear combinations of terms of the form
\[
\Theta^{(j)}(\hH) \prod_{s=1}^j\p^{e_s-e_{s-1}}_{\gamma_{\sigma(e_{s-1}+1)}\cdots \gamma_{\sigma(e_s)}} \hH,
\]
where $e_0=0$, $j\in\{1,\dots,{\bar q}\}$, $\{e_s\}_{s=1}^j$ is an increasing subset of $\{1,\dots,{\bar q}\}$ with $e_j={\bar q}$, and $\sigma\in \Pi_{\bar q}$ is a permutation of $\{1,\dots,{\bar q}\}$.  By the inductive hypothesis of Theorem \ref{th_distancefunctionRic},  
\[\abs{\nabla_g^{e_s-e_{s-1}} H_{k-1}}_g(v) \leq C\lambda(r(x))^{e_s-e_{s-1}-1}.
\]
Lemma \ref{lem_conf} and assumption (A2) then imply that
\[
\abs{\p^{e_s-e_{s-1}}_{\gamma_{\sigma(e_{s-1}+1)}\cdots \gamma_{\sigma(e_s)}} \hH}(v) \leq C\lambda_1^{-1},
\]
and, using also Lemma \ref{lem_A6},  
\begin{align*}
\norm{ \p^{\bar q}_{\gamma_1\cdots\gamma_{\bar q}}(\hH^{2\eta})}_{L^\infty(\B_{4^{1-{\bar q}}\beta}(w))}
&\leq C\max_{j,\{e_s\}_{s=1}^j,\sigma} \norm{\Theta^{(j)}(\hH) \prod_{s=1}^j\p^{e_s-e_{s-1}}_{\gamma_{\sigma(e_{s-1}+1)}\cdots \gamma_{\sigma(e_s)}} \hH}_{L^\infty(\B_{4^{1-{\bar q}}\beta}(w))}\\
&\leq C\max_j \lambda(\hH)^{2+j}\prod_{s=1}^j \lambda_1^{-1}\\
&\leq C\max_j \lambda_1^{2+j-j}\\
&\leq C\lambda_1^{2} .
\end{align*}
Accordingly 
\begin{equation}\label{Linfty1}
\norm{C_\theta \lambda_1^{-2} \p^{\bar q}_{\gamma_1\cdots\gamma_{\bar q}}(\Theta(\hH))}_{L^\infty(\B_{4^{1-{\bar q}}\beta}(w))}\leq C
\end{equation}
Finally, concerning the third addend in $\hat f$, since $n<{\bar q}$ and using again harmonic radius estimates we get 
\begin{align}\label{Linfty3}
\norm{\left(\partial^n_{a_1\cdots a_n} \p^2_{ij} \hh\right)\left( \p^{{\bar q}-n}_{b_1\cdots b_{{\bar q}-n}}\ton{\hg_\l}_{ij}\right)
}_{L^\infty(\B_{4^{1-{\bar q}}\beta}(w))} 
&\leq C \norm{\partial^n_{a_1\cdots a_n} \p^2_{ij} \hh}_{L^\infty(\B_{4^{1-{\bar q}}\beta}(w))} \\
&\leq C.\nonumber
\end{align}
Combining \eqref{Linfty2}, \eqref{Linfty1} and \eqref{Linfty3} we thus get
\begin{equation}\label{eq_hf}
\norm{\hat f
}_{L^\infty(\B_{4^{1-{\bar q}}\beta}(w))} \leq C.
\end{equation}

\textbf{3rd term.} \textit{Estimating $\qua{\hat f}_{C^{0,\alpha}(\B_{4^{-{\bar q}}\beta}(w))}$.}

First, note that for any function $F\in C^1 (\B_{4^{-{\bar q}}\beta}(w))$ it holds clearly that
\begin{equation}\label{C1a}
\qua{F}_{C^{0,\alpha}(\B_{4^{-{\bar q}}\beta}(w))}\leq (2^{1-2{\bar q}}\beta)^{1-\alpha}\norm{F}_{C^1(\B_{4^{-{\bar q}}\beta}(w))}.
\end{equation} 
Reasoning as for \eqref{Linfty1}, we get that, for any $s\in\{1,\dots,m\}$,
\[
\norm{C_\theta\l_1^{-2}\partial_s\partial^{\bar q}_{\gamma_1\cdots\gamma_{\bar q}}(\Theta(\hH))}_{L^\infty(\B_{4^{-{\bar q}}\beta}(w))}\leq C,
\]
from which, together with \eqref{C1a},
\begin{equation}\label{Ca1}
\qua{C_\theta\l_1^{-2} \p^{\bar q}_{\gamma_1\cdots\gamma_{\bar q}}((\Theta(\hH)))}_{C^{0,\alpha}(\B_{4^{-{\bar q}}\beta}(w))}\leq C.
\end{equation}
Moreover, as long as either $n\in\{1,\dots,{\bar q}\}$ or $n=0$ and $l\in\{1,\dots,{\bar q}-1\}$ we have that
\begin{align*}
&\norm{\p_s\qua{\left(\partial^n_{a_1\cdots a_n}\hg_\l^{ij}\right)\left(\p^l_{b_1\cdots b_l}\partial_i\hh\right)\left(\p^{{\bar q}-n-l}_{c_1\cdots c_{{\bar q}-n-l}}\partial_j\hh\right)}}_{L^\infty(\B_{4^{-{\bar q}}\beta}(w))}\\
\leq&\norm{\left(\partial^{n+1}_{sa_1\cdots a_n}\hg_\l^{ij}\right)\left(\p^l_{b_1\cdots b_l}\partial_i\hh\right)\left(\p^{{\bar q}-n-l}_{c_1\cdots c_{{\bar q}-n-l}}\partial_j\hh\right)}_{L^\infty(\B_{4^{-{\bar q}}\beta}(w))}\\
&+\norm{\left(\partial^n_{a_1\cdots a_n}\hg_\l^{ij}\right)\left(\p^{l+1}_{sb_1\cdots b_l}\partial_i\hh\right)\left(\p^{{\bar q}-n-l}_{c_1\cdots c_{{\bar q}-n-l}}\partial_j\hh\right)}_{L^\infty(\B_{4^{-{\bar q}}\beta}(w))}\\
&+\norm{\left(\partial^n_{a_1\cdots a_n}\hg_\l^{ij}\right)\left(\p^l_{b_1\cdots b_l}\partial_i\hh\right)\left(\p^{{\bar q}-n-l+1}_{sc_1\cdots c_{{\bar q}-n-l}}\partial_j\hh\right)}_{L^\infty(\B_{4^{-{\bar q}}\beta}(w))}\\
\leq & C
\end{align*}
for any $s\in\{1,\dots,m\}$, where each addend has been estimated reasoning as for \eqref{Linfty2}. Note that we used here the condition $n+1\leq {\bar q}+1\leq k-1$ which permits to have the needed control in the harmonic chart. Combining with  \eqref{C1a}, we get 
\begin{equation}\label{Ca2}
\qua{\left(\partial^n_{a_1\cdots a_n}\hg_\l^{ij}\right)\left(\p^l_{b_1\cdots b_l}\partial_i\hh\right)\left(\p^{{\bar q}-n-l}_{c_1\cdots c_{{\bar q}-n-l}}\partial_j\hh\right)}_{C^{0,\alpha}(\B_{4^{-{\bar q}}\beta}(w))}
\leq  C,
\end{equation}
Similarly, for any $s\in\{1,\dots,m\}$ and any $n\in\{1,\dots,{\bar q}-2\}$, it holds
\begin{align*}
&\norm{\p_s\qua{\left(\partial^n_{a_1\cdots a_n} \p^2_{ij} \hh\right)\left( \p^{{\bar q}-n}_{b_1\cdots b_{{\bar q}-n}}\hg_\l^{ij}\right)
}}_{L^\infty(\B_{4^{-{\bar q}}\beta}(w))}\\
\leq &\norm{\left(\partial^{n+1 }_{sa_1\cdots a_n} \p^2_{ij} \hh\right)\left( \p^{{\bar q}-n}_{b_1\cdots b_{{\bar q}-n}}\hg_\l^{ij}\right)
}_{L^\infty(\B_{4^{-{\bar q}}\beta}(w))}
+ \norm{\left(\partial^{n}_{a_1\cdots a_n} \p^2_{ij} \hh\right)\left( \p^{{\bar q}-n+1}_{sb_1\cdots b_{{\bar q}-n}}\hg_\l^{ij}\right)
}_{L^\infty(\B_{4^{-{\bar q}}\beta}(w))}\\
\leq & C,
\end{align*}
from which
\begin{equation}\label{Ca3}
\qua{\left(\partial^n_{a_1\cdots a_n} \p^2_{ij} \hh\right)\left( \p^{{\bar q}-n}_{b_1\cdots b_{{\bar q}-n}}\hg_\l^{ij}\right)}_{C^{0,\alpha}(\B_{4^{-{\bar q}}\beta}(w))}
\leq  C.
\end{equation}
Hence, in order to control $\qua{\hat f}_{C^{0,\alpha}(\B_{4^{-{\bar q}}\beta}(w))}$, it remains to estimate the LHS of \eqref{Ca2} for $l={\bar q}$ and $n=0$, and the LHS of \eqref{Ca3} for $n={\bar q}-1$. Namely, we have to bound
\begin{equation*}
\qua{\hg_\l^{ij}\left(\p^{\bar q}_{b_1\cdots b_{\bar q}}\partial_i\hh\right)}_{C^{0,\alpha}(\B_{4^{-{\bar q}}\beta}(w))}+\qua{\left(\partial^{{\bar q}-1}_{a_1\cdots a_{{\bar q}-1}} \p^2_{ij} \hh\right)\left( \p^{1}_{b_1}\hg_\l^{ij}\right)}_{C^{0,\alpha}(\B_{4^{-{\bar q}}\beta}(w))}.
\end{equation*}
By the properties of the harmonic chart, it is enough to obtain an upper bound for terms of the form
\begin{equation*}
\qua{\p^{{\bar q}+1}_{b_1\cdots b_{{\bar q}+1}}\hh}_{C^{0,\alpha}(\B_{4^{-{\bar q}}\beta}(w))},
\end{equation*}
for some indexes $b_1,\dots,b_{{\bar q}+1}\in\{1,\dots,m\}$. Letting $p=\frac{m}{1-\alpha}$, by the Euclidean Sobolev embeddings (see e.g \cite[p. 109]{Adams}), we have 
\begin{equation}
\label{Sobol}
[\p^{{\bar q}+1}_{b_1\cdots b_{{\bar q}+1}}\hh]_{C^{0,\alpha}(\B_{4^{-{\bar q}}\beta}(w))}\leq K_3\left\{\|\p^{{\bar q}+1}_{b_1\cdots b_{{\bar q}+1}}\hh\|^p_{L^p(\B_{4^{-{\bar q}}\beta}(w))}+\sum_{s=1}^m\|\partial_s \p^{{\bar q}+1}_{b_1\cdots b_{{\bar q}+1}}\hh\|^p_{L^p(\B_{4^{-{\bar q}}\beta}(w))}\right\}^{1/p},
\end{equation}
with $K_3$ a positive constant depending only on $m$, $\alpha$, $\beta$ and $k$ (via ${\bar q}$).
Concerning the first term at the RHS of \eqref{Sobol}, by the inductive hypothesis we have 
\begin{align}\label{Sob_pot}
\|\p^{{\bar q}+1}_{b_1\cdots b_{{\bar q}+1}}\hh\|^p_{L^p(\B_{4^{-{\bar q}}\beta}(w))} \leq \|\p^{{\bar q}+1}_{b_1\cdots b_{{\bar q}+1}}\hh\|^p_{L^\infty(\B_{4^{-{\bar q}}\beta}(w))} \beta^m\kappa_m\leq C,
\end{align}
where $\kappa_m$ is the volume of the $m$-dimensional unit ball $\B_1\subset \R^m$.
Concerning the second term at the RHS of \eqref{Sobol}, we note that 
$|\partial_s \p^{{\bar q}+1}_{b_1\cdots b_{{\bar q}+1}}\hh|\leq |D^2 \p^{{\bar q}}_{b_1\cdots b_{{\bar q}}}\hh|$, where, given a $C^2$ function $F$ on $\B_\beta$, $D^2F$ denotes the Euclidean Hessian matrix of $F$ and $|D^2F|^2\dot =\sum_{i,j=1}^m (\p_i\p_jF)^2$. According to the Euclidean Calder\'on-Zygmund inequality, \cite[Theorem 9.11]{GT}, we thus have
\begin{align}\label{CZ}
\|\partial_s \p^{{\bar q}+1}_{b_1\cdots b_{{\bar q}+1}}\hh\|_{L^p(\B_{4^{-{\bar q}}\beta}(w))}
&\leq \|D^2 \p^{{\bar q}}_{b_1\cdots b_{{\bar q}}}\hh\|_{L^p(\B_{4^{-{\bar q}}\beta}(w))} \\
&\leq C_1\left(\|\Delta_0\p^{{\bar q}}_{b_1\cdots b_{{\bar q}}}\hh\|_{L^p(\B_{4^{1-{\bar q}}\beta}(w))}+\|\p^{{\bar q}}_{b_1\cdots b_{{\bar q}}}\hh\|_{L^p(\B_{4^{1-{\bar q}}\beta})}\right),\nonumber
\end{align}
where $\Delta_0=\sum_i \partial_i\partial_i$ is the Euclidean Laplacian and the positive constant $C_1$ depends on $m$, $\alpha$ (via $p$), $\beta$ and $k$ (via ${\bar q}$). We have that \begin{equation}\label{eq_poten}
\|\p^{{\bar q}}_{b_1\cdots b_{{\bar q}}}\hh\|_{L^p(\B_{4^{1-{\bar q}}\beta})}\leq C
\end{equation}
by the inductive hypothesis. Concerning the other addend in \eqref{CZ}, using the properties of the harmonic chart, we get 
\begin{align*}
|\Delta_0\p^{{\bar q}}_{b_1\cdots b_{{\bar q}}}\hh| = |\sum_{s=1}^m\p^2_{ss}\p^{{\bar q}}_{b_1\cdots b_{{\bar q}}}\hh|  \leq C |\hg_\lambda^{ij}\p^2_{ij}\p^{{\bar q}}_{b_1\cdots  b_{{\bar q}}}\hh|=C|\hat f_b|,
\end{align*}
where 
$\hat f_b$ is defined as $\hat f$ in \eqref{eq_f}, up to replace the indexes $\gamma_1,\dots,\gamma_{\bar q}$ in the definition of $\hat f$ with $b_1,\dots,b_{\bar q}$. Accordingly, computing as for \eqref{eq_hf}, we get 
\begin{align*}
\|\Delta_0\p^{{\bar q}}_{b_1\cdots b_{{\bar q}}}\hh\|_{L^p(\B_{4^{1-{\bar q}}\beta}(w))}
%&\leq C \|\Delta_0\p^{{\bar q}}_{b_1\cdots b_{{\bar q}}}\hh\|_{L^p(\B_{4^{1-{\bar q}}\beta}(w))}
&\leq C\|\hat f_b\|_{L^p(\B_{4^{1-{\bar q}}\beta}(w))}\\
&\leq C\|\hat f_b\|_{L^\infty(\B_{4^{1-{\bar q}}\beta}(w))}\\
&\leq C.
\end{align*}
Inserting this latter and \eqref{eq_poten} in \eqref{CZ}, and combining with \eqref{Sobol} and \eqref{Sob_pot}, we obtain 
\begin{align*}
\qua{\p^{{\bar q}+1}_{b_1\cdots b_{{\bar q}+1}}\hh}_{C^{0,\alpha}(\B_{4^{-{\bar q}}\beta}(w))}\leq C.
\end{align*}
Together with \eqref{Ca1}, \eqref{Ca2} and \eqref{Ca3}, this gives 
\[
\qua{\hat f}_{C^{0,\alpha}(\B_{4^{-{\bar q}}\beta}(w))}\leq C.
\]
Hence, all the three addends in the RHS of \eqref{schauder} are upper bounded by a positive constant $C$, so that we have shown the validity of \eqref{2_lem_main} with $q={\bar q}+1$. This concludes the proof of Lemma \ref{lem_main}.
\end{proof}

We can now come back to the proof of Theorem \ref{th_distancefunctionRic}. 
Note that by assumption (A1) $\theta(t)\dot=t\lambda(t)$ is smooth, increasing on $[R_1,+\infty)$ and $\lim_{t\to +\infty}\theta(t)=+\infty$, hence $\theta^{-1}:[\theta(R_1),+\infty)\to [R_1,+\infty)$ is well-defined, smooth and increasing. Define $H=H_k\in C^\infty(M\setminus B_{R_H}(o))$ by $H=\theta^{-1}\circ h$, where $R_H$ is chosen large enough so that $h\geq \theta(R_1)$ on $M\setminus B_{R_H}(o)$. First, since $\theta^{-1}$ and $\lambda$ are increasing, from $C_h^{-1}\theta(r)\leq h \leq C_h \theta(r)$ we deduce that 
\[
H(x)\geq \theta^{-1}(C_h^{-1}\theta(r(x)))\geq \theta^{-1}(C_h^{-1}r(x)\lambda(C_h^{-1}r(x)))=C_h^{-1}r(x),\]
and similarly $H(x)\leq C_hr(x)$. To finish the proof of Theorem \ref{th_distancefunction} we are going to prove by induction on $j$ that, for $1\leq j \leq k$, we have that
\[
\left|\nabla_g^{j} H\right|(x)\leq C^{j-1}\max\{\lambda(r(x))^{j-1},1\}.
\]

We first consider the case $j=1$. Then
$
|\nabla_g h| (x) = \theta'(H(x))|\nabla_g H| (x)
$, so that, using also Lemma \ref{lem_A7},
\begin{align*}
|\nabla_g H| (x) &= \frac{|\nabla_g h| (x)}{\theta'(H(x))}\\
&\leq \frac{C_h \lambda(r(x))}{\theta'(H(x))}\\
&\leq C.
\end{align*}

%\begin{align*}
%|\nabla H| (x) &= \frac{|\nabla h| (x)}{H(x)\lambda'(H(x))+\lambda(H(x))}\\
%&\leq \frac{C_h \lambda(r(x))}{C_h^{-1}r(x)\lambda'(C_h^{-1}r(x))+\lambda(C_h^{-1}r(x))}\\
%&\leq C \frac{ \lambda(r(x))}{\lambda(C_h^{-1}r(x))}\\
%&\leq C.
%\end{align*}
Similarly, suppose that for some $1\leq \bar j \leq k$ and all $1\leq j \leq \bar j-1$, $\left|\nabla_g^{j} H\right|(x)\leq C^{j-1}\max\{\lambda(r(x))^{j-1},1\}$.
We have that $\nabla_g^{\bar j}h(x)- \nabla_g^{\bar j}H(x)\theta'(H)$ can be written as a linear combination of terms of the form $\theta^{(s)}(H)\otimes_{i=1}^s\nabla_g^{e_i}H$, with $2\leq s \leq \bar j$ and $\sum_{i=1}^se_i={\bar j}$.
By inductive hypothesis, using also Lemma \ref{lem_A8}, we have that
\begin{align*}
|\theta^{(s)}(H)\bigotimes_{i=1}^s\nabla_g^{e_i}H|&\leq |\theta^{(s)}(H)|\prod_{i=1}^s|\nabla_g^{e_i}H|\\
&\leq |\theta^{(s)}(H)|\prod_{i=1}^s \lambda (r(x))^{e_i-1}\\
&\leq C \lambda(H(x))^s\lambda(r(x))^{\bar j -s}\\
&\leq C\lambda(r(x))^{\bar j}.
\end{align*}
In particular, by Lemmas \ref{lem_main} and \ref{lem_conf},
\begin{align*}
|\nabla_g^{\bar j} H(x) | &\leq C\frac{|\nabla_g^{\bar j} h (x)| + \sum_{s=2}^{\bar j}|\theta^{(s)}(H)\bigotimes_{i=1}^s\nabla_g^{e_i}H|}{\theta'(H(x))}\\
&\leq C\frac{\lambda_1^{\bar j}}{\lambda(r(x))}\\
&\leq C \lambda(r(x))^{\bar j-1}.
\end{align*}
We have thus proved that there exists a function $H\in C^\infty (M\setminus R_H)$ such that for any $x$ outside a compact set $K$ of $M$ 
\begin{itemize}
\item[(i)]$C_h^{-1}r(x)\leq H(x)\leq C_hr(x)$;
\item[(ii)] for $1\leq j \leq k$, $\left|\nabla_g^{j} H\right|(x)\leq C \lambda(r(x))^{j-1}$.
\end{itemize}
Replacing $H$ with $H/C_h$, and choosing a suitable smooth continuation of $H|_{M\setminus K}$ inside the compact $K$, we get that, up to possibly increase the value of $C$, it holds
\begin{itemize}
\item[(i)]$C^{-2}r(x)\leq H(x)\leq \max\left\{r(x), 1\right\}$;
\item[(ii)] for $1\leq j \leq k$, $\left|\nabla_g^{j} H\right|(x)\leq C^{j-1}\max\{\lambda(r(x))^{j-1},1\}$;
\end{itemize}
as desired. 
%%%%%%%%%%%%%%%%%%%%%%%%%%%%%%%%%%%%%%%%%
\subsection{Proof of Theorem \ref{th_distancefunction}: case (b)}
In this section, we are going to prove Theorem \ref{th_distancefunction} in the set of assumptions (b), i.e.

\begin{theorem}\label{th_distancefunctionSect}
Let $(M^m, g)$ be a complete Riemannian manifold and $o\in M$ a fixed reference point, $r(x)\doteq \mathrm{dist}_{g}(x,o)$.  Let $k\in \mathbb N^{+}$.  
If $k\geq 2$, suppose in addition that for some $D>0$,
\[
\ |\nabla_g^j\mathrm{Ric}_{g}|(x)\leq \lambda(r(x))^{2+j},\ 0\leq j \leq k-2,\quad|\mathrm{Sect}_{g}|(x)\leq D\lambda(r(x))^{2} \quad\mathrm{on}\,\,M,
\]
where the function $\lambda$ satisfies assumptions (A1), (A2), (A3), (A4($j$)) for $j=1,\dots,k$.
Then there exists an exhaustion function $H=H_k\in C^{\infty}(M)$ such that for some positive constant $C_H>1$ independent of $x$, we have on $M$ that 
\begin{itemize}
\item[(i)]$C_H^{-2}r(x)\leq H(x)\leq \max\left\{r(x), 1\right\}$;
\item[(ii)] for $1\leq j \leq k$, $\left|\nabla_g^{j} H\right|(x)\leq C_H^{j-1}\max\{\lambda(r(x))^{j-1},1\}$.
\end{itemize}
\end{theorem}

We proceed as in \cite[Subsection 4.2]{IRV-HessCutOff}. Note that the first part of the proof of Theorem \ref{th_distancefunctionRic} does not require the control on the injectivity radius. In particular, we can suppose by induction that there exists a distance-like function $H_{k-1}\in C^{\infty}(M)$ such that for some positive constant $C_{k-1}>1$ independent of $x$ and $o$, we have on $M$ that 
\begin{itemize}
\item[(i)]$C_{k-1}^{-2}r(x)\leq H_{k-1}(x)\leq \max\left\{r(x), 1\right\}$;
\item[(ii)] for $1\leq j \leq k-1$, $\left|\nabla_g^{j} H_{k-1}\right|_g(x)\leq C_{k-1}^{j-1}\max\{\lambda(r(x))^{j-1},1\}$.
\end{itemize}
By Theorem \ref{refinedBS} we get also in the present assumptions the validity of Proposition \ref{BS}. In particular there exists $h=h_k\in C^{\infty}(M)$ such that 
\begin{itemize}
\item[(i)] $C_h^{-1}r(x)\lambda(r(x))\leq h(x) \leq C_h\max\{1;r(x)\lambda(r(x))\}$;
\item[(ii)] $|\nabla_g h|_g\leq C_h\lambda(r(x))$ on $M$;
\item[(iii)] $\Delta_g h = |\nabla_g h|^2_g-C_\theta (\theta'(H_{k-1}(x)))^2$ on $M\setminus B_{R_\theta}(o)$ for some constants $C_\theta>0$ and $R_\theta >0$.
\end{itemize}
Fix $R_{0}\in\mathbb{R}^{+}$ such that $\lambda(2R_0+1)>\frac{\pi}{2R_0}$. This is always possible since $\lambda$ is strictly positive and non-decreasing. If $x\in M$ satisfies $r(x)>1+R_0$, then on $B_{R_0}(x)$
\[
\ |\mathrm{Sect}_{g}|\leq D^2\lambda^2(R_0+r(x))\doteq D^2K^2_{x}.
\]
By a localized version of the Cartan-Hadamard theorem (see e.g. \cite[Lemma 2.7]{GuneysuPigola_17}) we have that, for every $0<R<\min\left\{\frac\pi{DK_{x}},R_0\right\}$, there exists  a smooth complete Riemannian manifold $(\bar{M},\bar{g})$, $\bar{x}\in \bar{M}$ and a smooth surjective local isometry
\[
\ F\doteq F_{g,x,R}:B_{R}^{\bar{g}}(\bar{x})\to B_{R}^{g}(x)
\]
such that
\begin{itemize}
\item $F(\bar{x})=x$;
\item $\mathrm{inj}_{\bar{g}}(\bar{x})\geq R$;
\item $|\mathrm{Sect}_{\bar{g}}|\leq D^2K^2_{x}$ on $B_{R}^{\bar{g}}(\bar{x})$;
\item $F(B_{r}^{\bar{g}}(\bar{x}))=B_{r}^{g}(x)$, for all $0<r<R$.
\end{itemize}
In particular, for every $\bar{y}\in B^{\bar{g}}_{R/2}(\bar{x})$ ,we have that
\begin{equation}\label{Sect1}
|\mathrm{Sect}_{\bar{g}}|(\bar{y})\leq D^2K^2_{x},\quad|\nabla_{\bar{g}}^j \mathrm{Ric}_{\bar{g}}|(\bar{y})\leq K^{2+j}_{x},\quad
\mathrm{inj}_{\bar{g}}(\bar{y})\geq d_{\bar{g}}\left(\bar{y},\partial B_{R}^{\bar{g}}(\bar{x})\right)\geq \frac{R}{2}.
\end{equation}
Letting  $\lambda_{R_0}^2\doteq (m-1)K^2_{x}$, we set
\[
\ \bar{g}_{\lambda}=\lambda_{R_0}^2\bar{g}.
\]
Then, by \eqref{Sect1},
\begin{equation*}
|\mathrm{Ric}_{\bar{g}_{\lambda}}|(\bar{y})\leq D^2,\quad|\nabla_{\bar{g}_{\lambda}}^j \mathrm{Ric}_{\bar{g}_{\lambda}}|(\bar{y})\leq 1,\quad
\mathrm{inj}_{\bar{g}_{\lambda}}(\bar{y})\geq \lambda_{R_0}\frac{R}{2}.
\end{equation*}
We can choose $R=\frac{\pi}{2DK_{x}}<R_0$, obtaining that
\[
\ \mathrm{inj}_{\bar{g}_{\lambda}}(\bar{y})\geq \frac{\sqrt{m-1}\pi}{4D}\doteq i_{0}.
\]
Proposition \ref{HarmRadEst} hence yields that there exists a constant $C_{HR}=C_{HR}(m,Q,i_0,k)$ independent of $x$ such that on $B_{C_{HR}}^{\bar g_\lambda}(x)\subset\bar M$ there exists a centered harmonic chart $\varphi_H=(y^1,\dots,y^m):B_{C_{HR}}^{\bar g_\lambda}(x)\to U\subset \mathbb R^m$ such that $\varphi_{H}(x)=0$ and, setting $\hat g_\l=\bar g_\lambda \circ  \varphi_H^{-1}$, it holds

\begin{enumerate}
\item[$(1_{HC})$] $Q^{-1}\delta_{ij}\leq \ton{\hat g_\l}_{ij}\leq Q\delta_{ij}$ as bilinear forms;
\item[$(2_{HC})$] $\sum_{1\leq|\gamma|\leq k-1}C_{HR}^{|\gamma|}\sup_y\left|\partial_{\gamma}\ton{\hat g_\l}_{ij}(y)\right|\leq Q-1$,
\item[$(1'_{HC})$] $Q^{-1}\delta^{ij}\leq \hat g_\l^{ij}\leq Q\delta^{ij}$;
\item[$(2'_{HC})$] $\sum_{1\leq|\gamma|\leq k-1}\sup_y\left|\partial_{\gamma}\hg_\l^{ij}(y)\right|\leq C(Q)$,
\end{enumerate}
for some constant $C(Q)$, depending only on $Q$.

For a fixed $x$, we can define $\bar{h}:B_{R}^{\bar{g}}(\bar{x})\to\mathbb{R}$ by $\bar{h}=h\circ F$ and $\hat{h}:\B_{\beta}\to\mathbb{R}$ by $\hat h =\bar{h}\circ \varphi_H^{-1}$.

Since $F$ is a local isometry, at this stage, we can estimate the covariant derivatives of $\hat h$ exactly as we did in the proof of Theorem \ref{th_distancefunctionRic}. In order to get our conclusion, we deduce from these estimates a control on the covariant derivatives of $\bar h$, hence of $h$, and finally of $H=\theta^{-1}\circ h$ outside a compact set of $M$. 

%%%%%%%%%%%%%%%%%%%%%%%%%%%%%%%%%%%%%%%%%
\section{$k$-th order cut-offs and application to the density problem}\label{SectCut&Dens}
In the following corollary we notice that higher order exhaustion functions, as the ones which we obtained in the proof of Theorem \ref{th_distancefunction}, permit to construct $k$-th order cut-off functions. These, in turn, will allow us to conclude the proof of the density result Theorem \ref{th_main1}.

\begin{corollary}\label{kCutOff}
Let $(M, g)$ be a complete Riemannian manifold and $o\in M$ a fixed reference point, $r(x)\doteq \mathrm{dist}_{g}(x,o)$. Let $k\geq 2$ be an integer.  Let $\lambda$ satisfy assumption (A1), (A2), (A3) and (A4($j$)), $j=1,\dots,k$, and suppose that  $\lambda^{1-k}\not\in L^1([1,+\infty))$.
Suppose that one of the following curvature assumptions holds
\begin{itemize}
\item[(a)] for some $i_0>0$,
\[
\ |\nabla_g^j\mathrm{Ric}_{g}|(x)\leq \lambda(r(x))^{2+j},\ 0\leq j \leq k-2,\quad\mathrm{inj}_{g}(x)\geq \frac{i_0}{\lambda(r(x))}>0\quad\mathrm{on}\,\,M.
\]
\item[(b)] for some $D>0$,
\[
\ |\nabla_g^j\mathrm{Ric}_{g}|(x)\leq \lambda(r(x))^{2+j},\ 1\leq j \leq k-2,\ |\mathrm{Sect}_{g}|(x)\leq D^2\lambda(r(x))^2\quad\mathrm{on}\,\,M.
\]
\end{itemize}

Then there exist a family of cut-off functions $\left\{\chi_{R}\right\}_{R>3}\subset C^\infty_c(M)$ and a constant $C_\chi>0$ independent of $R$ such that
\begin{enumerate}
\item $\chi_{R}=1$ on $B_{C_H^{-1}(R-2)}(o)$, with $C_H$ the constant appearing in Theorem \ref{th_distancefunction};
\item $|\nabla_{g}^{j}\chi_{R}|\leq C_\chi$ for $j=1,\ldots,k$.
\end{enumerate}
\end{corollary}

\begin{proof}
Let $\alpha_R$ be a positive constant and define the family $\{\psi_R\}_{R>3}$ of functions  on $\R$ by
\[
\psi_R(t)\dot= \begin{cases}\max\{1-\int_R^t\alpha_R\lambda^{1-k}(s)ds;0\},&\text{if }t>R\\1,&\text{if }t\leq R.\end{cases}
\]
We note that $\psi_R\in C^0(\R)$ and, by the assumptions on $\lambda$, for every $R$ there exists a constant $T(R)\in\R$ depending on $R$ such that $\psi_R(t)\equiv 0$ if and only if $t\in[T(R),\infty)$. We can choose the positive constant $\alpha_R$ small enough so that $T(R)-R>4$. Note also that $\alpha_R\doteq\alpha$ can be chosen independent of $R$ since $\lambda$ is increasing.
Clearly the $\psi_R$ are not $C^2$, so we want to regularize it in a neighborhood of $R$ and $T(R)$, keeping the derivatives controlled. 

Let $\zeta:\R\to\R$ be a $C^\infty$ increasing function such that $\zeta \equiv 0$ in $(-\infty,-1]$, $\zeta \equiv 1$ in $[1,+\infty)$.

There exists $q_R\in[-1,1]$ such that 
\begin{equation}\label{claim1}
\int_0^{R+2}\frac{\zeta(t+R-q_R)}{\lambda^{k-1}(t)}dt=\int_0^{R+2}\frac{\mathds{1}_{[R,+\infty)}(t)}{\lambda^{k-1}(t)}dt =\int_R^{R+2}\frac{1}{\lambda^{k-1}(t)}dt.
\end{equation}
In fact, we have that 
\[
\int_0^{R+2}\frac{\zeta(t+R-1)}{\lambda^{k-1}(t)}dt
\leq \int_0^{R+2}\frac{\mathds{1}_{[R,+\infty)}(t)}{\lambda^{k-1}(t)}dt \leq 
\int_0^{R+2}\frac{\zeta(t+R+1)}{\lambda^{k-1}(t)}dt,
\]
so that \eqref{claim1} follows by continuity. 
Similarly, we have also that there exists $Q_R\in[-1,1]$ such that 
\begin{equation}\label{claim2}
\int_{R+2}^{T(R)+2}\frac{\zeta(-t+T(R)+Q_R)}{\lambda^{k-1}(t)}dt=\int_{R+2}^{T(R)+2}\frac{\mathds{1}_{[0,T(R)]}(t)}{\lambda^{k-1}(t)}dt =\int_{R+2}^{T(R)}\frac{1}{\lambda^{k-1}(t)}dt.
\end{equation}
Now for every $R>1$, define the real smooth function 
\[
\mu_R(t)\dot=\frac{\zeta(-t+T(R)+Q_R)\zeta(t-R+q_R)}{\lambda^{k-1} (t)}\]
and
\[
\phi_R(t)\dot = 1-\alpha\int_{R-2}^t \mu_R(s)ds.
\]
Then $\phi_R\in C^\infty(\R)$, $\phi_R$ is decreasing, $\phi_R\equiv 1$ if $t<R-2$ and $\phi_R = \psi_R \equiv 0$ if $t>T(R)+2$ because of \eqref{claim1} and \eqref{claim2}. Moreover 
\[
|\phi_R'(t)| = \alpha|\mu_R(t)| \leq \frac{\alpha}{\lambda ^{k-1}(t)}.
\]
Similarly, let $Z_R(t)\dot=\zeta(-t+T(R)+Q_R)\zeta(t-R+q_R)$ and $\Lambda(t)\dot= \lambda^{1-k}(t)$. Then, for $j=2,\dots,k$
\begin{align*}
|\phi_R^{(j)}(t)|=\alpha|\mu_R^{(j-1)}(t)|\leq\alpha\sum_{s=1}^{j-1}\binom{j-1}{s}|Z_R^{(s)}(t)||\Lambda^{(j-1-s)}(t)|.
\end{align*}
We have that 
\begin{align*}
|Z_R^{(s)}(t)| \leq 2^s \max_{0\leq l \leq k}\|\zeta^{(l)}\|_\infty^{2}<C
\end{align*}
independently of $R$. Moreover  $\Lambda^{(j-1-s)}(t)$ can be written as a linear combination of terms of the form $\lambda^{1-k-u}(t)\prod_{i=1}^u\lambda^{(e_i)}(t)$, with $1\leq u \leq j-1-s$ and $\sum_{i=1}^ue_i=j-1-s$. Since, by assumption (A4(j)), $\lambda^{(e_i)}(t)\leq M_{4(e_i)}\lambda^{e_i}(t)$, we get
\[\left|\lambda^{1-k-u}(t)\prod_{i=1}^u\lambda^{(e_i)}(t)\right|\leq C\lambda^{1-k-u}(t)\lambda^{j-1-s}(t)\]
and
\begin{align*}
|\Lambda^{(j-1-s)}(t)|\leq C\lambda^{1-k}(t)
\end{align*}
for every $s=1,\dots,j-1$,
so that 
\begin{align}\label{bd_deriv-phi}
|\phi_R^{(j)}(t)|\leq C\lambda^{1-k}(t)
\end{align}
for every $j=2,\dots,k$. 

Now, define the family of cut-off functions $\{\chi_R\}_{R>3}\subset C^\infty_c(M)$ by $\chi_R\dot = \phi_R \circ H$, where $H$ is the $k$-th order distance-like function given by Theorem \ref{th_distancefunction}.
Note that if $r(x)\leq C_H^{-1}(R-2)$, then $\chi_R(x)=1$. Moreover, we have that for every $j=1,\dots,k$, $\nabla_g^{j}\chi_R(x)$ can be written as a linear combination of terms of the form $\phi_R^{(s)}(H)\bigotimes_{i=1}^s\nabla_g^{e_i}H$, with $1\leq s \leq j$ and $\sum_{i=1}^se_i=j$.
We have
\begin{align*}
|\phi_R^{(s)}(H)\bigotimes_{i=1}^s\nabla_g^{e_i}H|&\leq |\phi_R^{(s)}(H)|\prod_{i=1}^s|\nabla_g^{e_i}H|\\
&\leq |\phi_R^{(s)}(H)|\prod_{i=1}^s \lambda (r(x))^{e_i-1}\\
&\leq C \lambda(H(x))^{1-k}\lambda(r(x))^{j -s}\\
&\leq C,
\end{align*}
and thus
\[ |\nabla_g^{j}\chi_R|\leq C \]
for every $j=1,\dots,k$.
\end{proof}
%%%%%%%%%%%%%%%%%%%%%%%%%%%%%%%%%%%%%%%%%

\subsection{Proof of the density result}\label{ProofDens}
We can now give the proof of the following result which we stated in Section \ref{Intro} as Theorem \ref{th_main1}.
\begin{theorem}\label{th_main1Rep}
In the assumptions of Corollary \ref{kCutOff}, we have that $W^{k,p}(M)=W^{k,p}_0(M)$ for all $p\in[1,+\infty)$. 
\end{theorem}

\begin{proof}[Proof (of Theorem \ref{th_main1})]
We can apply Corollary \ref{kCutOff} and get the existence of a sequence of  cut-off functions $\left\{\chi_{n}\right\}$ with uniformly bounded covariant derivatives up to order $k$.  This suffices to get the desired density result. Indeed, first recall that $C^{\infty}(M)\cap W^{k,p}(M)$ is dense in $W^{k,p}(M)$ (see for instance \cite{GuidettiGuneysuPallar}). Then, given a smooth function $f\in C^{\infty}(M)\cap W^{k,p}(M)$, define $f_{n}\doteq\chi_{n}f$. One obtains that
\begin{align}
\left\|f_{n}-f\right\|_{L^{p}}=&\left\|(1-\chi_{n})f\right\|_{L^{p}},\label{fn1}\\
\left\|\nabla_{g} (f_{n}-f)\right\|_{L^{p}}\leq&\left\|f\nabla_{g}\chi_{n}\right\|_{L^{p}}+\left\|(1-\chi_{n})\nabla_{g} f\right\|_{L^{p}},\label{fn2}\\
\left\|\nabla_{g}^{j}(f_{n}-f)\right\|_{L^{p}}\leq& \sum_{l=0}^{j-1} \left\|C_{l}|\nabla_{g}^{j-l}\chi_{n}||\nabla_{g}^{l}f|\right\|_{L^{p}}\label{fnj}+\left\|(1-\chi_{n})\nabla_{g}^{j}f\right\|_{L_{p}},\, j=2,\ldots,k,
\end{align}
where $C_l$ are integer constants depending on $j$ and $l$. Note that each of $(1-\chi_{n})$, $\left\{\nabla^{j}\chi_{n}\right\}_{j=1,\ldots,k}$ is uniformly bounded and supported in $\textrm{supp}(1-\chi_{n})$. Moroever, given any compact set $K\subset M$, we have that $\mathrm{supp}(1-\chi_{n})\subset M\setminus K$ for $n\gg1$. since $f\in W^{k,p}(M)$ this permits to conclude that all the terms at the RHS of \eqref{fn1}, \eqref{fn2}, \eqref{fnj} tend to $0$ as $n\to\infty$. More precisely we have that both the terms of the form
\[
 \left\||\nabla_{g}^{j-l}\chi_{n}||\nabla_{g}^{l}f|\right\|_{L^{p}}
 \leq C_\chi\left\|\nabla_{g}^{l}f\right\|_{L^{p}(\textrm{supp}(1-\chi_{n}))}
\]
and the terms of the form
\[ \left\|(1-\chi_{n})\nabla_{g}^{j}f\right\|_{L^{p}} \leq  \left\|\nabla_{g}^{j}f\right\|_{L^{p}(\textrm{supp}(1-\chi_{n}))} 
\]
go to $0$ as $n\to \infty$ since $f\in W^{k,p}(M)$.
\end{proof}
%%%%%%%%%%%%%%%%%%%%%%%%%%%%%%%%%%%%%%%%

\section{Case p=2}\label{p2}
\subsection{$k$-th order (rough) Laplacian cut-offs}

In this subsection we prove versions of Theorem \ref{th_distancefunctionSect} and Corollary \ref{kCutOff} under weaker assumptions. Namely, we assume a control on the derivatives of the Ricci curvature up to a smaller order. As a price to pay, we do not get a control on the whole $k$-th order covariant derivative $\nabla^k H$ of the distance-like function $H$, but only on its trace, i.e. the rough Laplacian of the $(k-2)$-th covariant derivative of $H$. This result will be used in Corollary \ref{kLaplCutOff} to construct a family of $k$-th order (rough) Laplacian cut-offs. In the rest of this section we will combine these (rough) Laplacian cut-offs and the Weitzenb\"ock formula for the Sampson-Lichnerowicz Laplacian to get the density of smooth compactly supported function in $W^{k,2}(M)$.

\begin{theorem}\label{th_distancefunctionLapl}
Let $(M, g)$ be a complete Riemannian manifold and $o\in M$ a fixed reference point, $r(x)\doteq \mathrm{dist}_{g}(x,o)$.  Let $k\ge 3$ be an integer.  Let $\lambda$ satisfy assumption (A1), (A2), (A3) and (A4($j$)), $j=1,\dots,k-1$.
Suppose that 
\begin{equation}\label{Ass_p=2}
|\mathrm{Sect}|(x)\leq \lambda(r(x))^{2},\quad\text{ and }\quad  |\nabla^{j}\mathrm{Ric}|(x)\leq \lambda(r(x))^{2+j},\qquad 0\leq j\leq k-3.
\end{equation}
Then there exists an exhaustion function $H\in C^{\infty}(M)$ such that for some positive constant $C_\Delta>1$ independent of $x$, we have on $M$ that 
\begin{itemize}
\item[(i)]$C_\Delta^{-2}r(x)\leq H(x)\leq \max\left\{r(x), 1\right\}$;
\item[(ii)] $
\ |\nabla^{j}H|\leq C_\Delta^{j-1}\max\left\{\lambda(r(x))^{j-1}, 1\right\},\qquad1\leq j\leq k-1;$
\item[(iii)] $|\Delta\nabla^{k-2}H|\leq C_\Delta\max\left\{\lambda(r(x))^{k-1}, 1\right\}$.
\end{itemize}
\end{theorem}

When acting on sections of a tensor bundle, $\Delta$ denotes the rough Laplacian, i.e. $\Delta T = -\nabla^\ast\nabla T = \mathrm{tr}_{12} \nabla^2 T$, for any tensor field $T$. Note that $\Delta$ is equal to minus the Bochner Laplacian $\Delta_B$ we will use in Subsection \ref{subsec_Weitz}.  
\begin{proof}
	Note that in the proof of Theorem \ref{th_distancefunctionSect}, the assumption $|\nabla^{k-2}\mathrm{Ric}|\leq \lambda(r(x))^{k}$ was used only to control $|\nabla^{k} h|$. Accordingly, by the proof of Theorem \ref{th_distancefunctionSect} we already know that there exists a smooth exhaustion function $h\in C^{\infty}(M)$ such that 
\begin{align}
&C^{-1}r(x)\lambda(r(x))\leq h(x)\leq C\max\left\{1;r(x)\lambda(r(x))\right\}\quad\,\forall\, x\in M\nonumber\\
%&|\nabla h|(x)\leq C\lambda(r(x))\quad\,\forall\,x\in M\setminus \overline{B}_{\bar{\rho}}(o)\nonumber\\
&|\nabla^{j} h|(x)\leq C\lambda^{j}(r(x))\quad\,\forall\,x\in M\setminus\overline{B}_{\bar{\rho}}(o),\,\,j=1,\ldots,k-1.\label{hj}
\end{align}
Moreover, by construction
\begin{equation}\label{Eqh}
\ \Delta h=|\nabla h|^2-C_{\theta}\left(\theta^{\prime}(H_{k-2}(x))\right)^2=|\nabla h|^2-C_{\theta}\Theta(H_{k-2}(x))\quad\,\mathrm{on}\,M\setminus\bar{B}_{R_{\theta}}(o),
\end{equation}
for some constants $C_{\theta}>0$ and $R_{\theta}>0$. Here $H_{k-2}$ is such that for some positive constant $C>1$ (independent of $x$ and $o$) we have on $M$ that
\begin{align*}
&Cr(x)\leq H_{k-2}(x)\leq \max\left\{r(x), 1\right\}\\
&|\nabla H_{k-2}|(x)\leq 1\\
&|\nabla^{j}H_{k-2}|(x)\leq C\max\left\{\lambda^{j-1}(x),1\right\},\quad\,j=2,\ldots,k-2.
\end{align*}
Taking $\nabla^{k-2}$ of \eqref{Eqh} we obtain that
\[
\ \nabla^{k-2}\Delta h
=2\nabla^{k-3}\left(\nabla^{2}h (\nabla h, \cdot)\right)-C_{\theta}\nabla^{k-2}(\Theta (H_{k-2})).
\]
Hence
\[
\ |\nabla^{k-2}\Delta h|\leq C\left[|\nabla^{k-1}h||\nabla h|+|\nabla^{k-2}h||\nabla^{2}h|+\ldots+|\nabla^{2}h||\nabla^{k-2}h|\right]+C_{\theta} |\mathcal{L}|.
\]
Here $\mathcal{L}$ is a linear combination of terms of the form
\[
\ \Theta^{(j)}(H_{k-2})\prod_{s=1}^{j}\nabla^{e_{s}-e_{s-1}}H_{k-2},
\]
where $e_{0}=0$, $j\in\left\{1,\ldots,k-2\right\}$ and $\left\{e_{s}\right\}_{s=1}^{j}$ is an increasing subset of $\left\{1,\ldots,k-1\right\}$ with $e_{j}=k-2$.\\
Using the above properties of $h$ and $H_{k-2}$ and Lemma \ref{lem_A6} we hence get that
\[
\ |\nabla^{k-2}\Delta h|\leq C\lambda^{k}(r(x)).
\]
By Lemma \ref{Comm} in Appendix \ref{App_comm}, we thus obtain that
\begin{align}\label{DeltaNablah}
\ |\Delta\nabla^{k-2}h|=&|\nabla^{k-2}\Delta h|+|\mathrm{Riem}*\nabla^{k-2}h
+\nabla \mathrm{Ric}*\nabla^{k-3} h+\ldots+\nabla^{k-3}\mathrm{Ric}*\nabla h
|\\
\leq&|\nabla^{k-2}\Delta h|+C\left(|\mathrm{Riem}||\nabla^{k-2}h|
+|\nabla \mathrm{Ric}||\nabla^{k-3} h|+\ldots+|\nabla^{k-3}\mathrm{Ric}||\nabla h
|\right)\nonumber\\
\leq&C\left(\lambda(r(x))^k+\sum_{j=0}^{k-3}\lambda(r(x))^{2+j}\lambda(r(x))^{k-2-j}\right)\nonumber\\
\leq&C\lambda(r(x))^{k}.\nonumber
\end{align}
For the definition of the notation $\ast$ appearing in the latter formula see Section \ref{SectNotations}.

As in the proof of Theorem \ref{th_distancefunction} define $H\doteq H_{k-1}\in C^{\infty}(M\setminus B_{R_{H}}(o))$ by $H=\theta^{-1}\circ h$. Then $H$ is distance-like and
\[
\ |\nabla^{j}H|\leq C^{j-1}\max\left\{\lambda(r(x))^{j-1}, 1\right\}\qquad1\leq j\leq k-1.
\] 
Moreover we have that
\[
\ \Delta\nabla^{k-2}h=\Delta\left(\theta^{\prime}(H)\nabla^{k-2} H\right)+\mathcal{L}_{1},
\]
where $\mathcal{L}_{1}$ is a linear combination of terms of the form $\Delta\left(\theta^{(s)}(H)\bigotimes\nabla^{e_{i}}H\right)$, with $2\leq s\leq k-2$ and $\sum_{i=1}^{s}e_{i}=k-2$. Letting $\left\{E_{i}\right\}$ being a local orthonormal frame on $M$, note that
\begin{align*}
\Delta(\theta^{\prime}(H)\nabla^{k-2}H)=&\sum_{i}\nabla_{E_{i}}\left(\theta^{\prime\prime}(H)\nabla_{E_{i}}H\nabla^{k-2}H+\theta^{\prime}(H)\nabla_{E_{i}}\nabla^{k-2}H\right)\\
=&\theta^{\prime\prime\prime}(H)|\nabla H|^2\nabla^{k-2}H+\theta^{\prime\prime}(H)\Delta H\nabla^{k-2}H\\
&+2\theta^{\prime\prime}(H)\nabla_{\nabla H}\nabla^{k-2}H+\theta^{\prime}(H)\Delta\nabla^{k-2}H,
\end{align*}
and
\begin{align*}
\Delta\left[\theta^{(s)}(H)\otimes_{i=1}^{s}\nabla^{e_{i}}H\right]=&\sum_{i}\nabla_{E_{i}}\left(\theta^{(s+1)}(H)\nabla_{E_{i}}H\otimes\bigotimes_{i=1}^{s}\nabla^{e_{i}}H+\theta^{(s)}(H)\nabla_{E_{i}}\left(\otimes_{i=1}^{s}\nabla^{e_{i}}H\right)\right)\\
=&\theta^{(s+2)}(H)|\nabla H|^2\bigotimes_{i=1}^{s}\nabla^{e_{i}}H+\theta^{(s+1)}(H)\Delta H\bigotimes_{i=1}^{s}\nabla^{e_{i}}H\\
&+2\theta^{(s+1)}(H)\nabla_{\nabla H}\left(\bigotimes_{i=1}^{s}\nabla^{e_{i}}H\right)+\theta^{(s)}(H)\Delta\left(\bigotimes_{i=1}^{s}\nabla^{e_{i}}H\right).
\end{align*}
Hence, using Lemma \ref{lem_A7}, Lemma \ref{lem_A8}, and proceeding as for the end of the proof of Theorem \ref{th_distancefunctionRic}, we get that
\begin{align*}
|\Delta\nabla^{k-2}H|=&\frac{C}{\theta^{\prime}(H)}\left[|\Delta \nabla^{k-2}h|+|\theta'''(H)||\nabla H|^2|\nabla^{k-2}H|+|\theta''(H)||\Delta H||\nabla^{k-2}H|\right.\\
&\left.+|\theta''(H)||\nabla_{\nabla H}\nabla^{k-2}H|+\max_{s=2,\ldots, k-2}|\theta^{(s+2)}(H)||\nabla H|^2|\bigotimes_{i=1}^{s}\nabla^{e_{i}}H|\right.\\
&\left.+\max_{s=2,\ldots, k-2}|\theta^{(s+2)}|\theta^{(s+1)}(H)||\Delta H \bigotimes_{i=1}^{s}\nabla^{e_{i}}H|\right.\\
&\left.+\max_{s=2,\ldots,k-2}(|\theta^{(s+1)}(H)||\nabla_{\nabla H}\bigotimes_{i=1}^{s}\nabla^{e_{i}}H|+|\theta^{(s)}(H)||\Delta(\bigotimes_{i=1}^{s}\nabla^{e_{i}}H)|)\right]\\
\leq& C\lambda(r)^{k-1}.
\end{align*}
\end{proof}
\bigskip
Using the function $H$ coming from Theorem \ref{th_distancefunctionLapl},
we want now to produce a sequence of higher order (rough) Laplacian cut-off functions. This will be done in the following 

\begin{corollary}\label{kLaplCutOff}
Let $(M, g)$ be a complete Riemannian manifold and $o\in M$ a fixed reference point $r(x)\doteq\mathrm{dist}_{g}(x,o)$. Let $k\geq 3$ be an integer. Let $\lambda$ satisfy assumptions (A1), (A2), (A3), and (A4(j)) for $j=1,\ldots,k-1$, and suppose that $\lambda^{1-k}\notin L^{1}([1,+\infty))$. Suppose that 
\begin{equation}\label{Ass_p=22}
|\nabla^{j}\mathrm{Riem}|(x)\leq \lambda(r(x))^{2+j},\qquad 0\leq j\leq k-3.
\end{equation}
Then there exists a family of cut-off functions $\left\{\chi_{R}\right\}\subset C_{c}^{\infty}(M)$, and a constant $C>0$ independent of $R$ such that, 
\begin{enumerate}
\item $\chi_{R}=1$ on $B_{C_{H}^{-1}(R-2)}(o)$;
\item $|\nabla^{j}\chi_{R}|\leq C\lambda^{-k+j}$, \quad $j=1,\ldots,k-1$;
\item $|\Delta\nabla^{k-2}\chi_{R}|\leq C$,
\end{enumerate}
\end{corollary}

\begin{proof}
%Let $\Gamma=\frac{D_{2}}{D_{1}}\geq 1$ and $\gamma>\Gamma^{\frac{1}{\beta}}$ a real number. Let $\phi\in C^{\infty}(\mathbb{R}, \left[0,1\right])$ be such that
%\[
%\ \left.\phi\right|_{(-\infty, \Gamma]}=1,\quad \left.\phi\right|_{(\gamma^{\beta}, \infty]}=0,\quad \sum_{i=1}^{k+1}|\phi^{(i)}|\leq a,
%\]
%for some $a>0$. For any $R>0$, let $\phi_{R}\in C^{\infty}\left(\left[0,+\infty\right)\right)$ be defined by 
%\[
%\ \phi_{R}(t)\doteq\phi\left(\frac{t}{D_{1}R^{\beta}}\right).
%\]
%Then
%\[
%\ |\phi_{R}^{(i)}|\leq \frac{a}{(D_{1}R^{\beta})^{i}},\quad i=1,\ldots,k+1.
%\]
As in Corollary \ref{kCutOff}, For each radius $R\gg1$, define $\chi_{R}\doteq\phi_{R}\circ H$, with $H$ the distance-like function given by Theorem \ref{th_distancefunctionLapl} and $\phi_{R}$ defined as in the proof of Corollary \ref{kCutOff}. Properties (1) and (2) follow from the proof of Corollary \ref{kCutOff}. About (3), using Lemma \ref{Comm} we note that
\begin{equation}\label{DeltaNablaChir}
|\Delta\nabla^{k-2}\chi_{R}|\leq |\nabla^{k-2}\Delta\chi_{R}|+|\mathrm{Riem}*\nabla^{k-2}\chi_{R}
+\nabla \mathrm{Ric}*\nabla^{k-3} \chi_{R}+\ldots\nabla^{k-3}\mathrm{Ric}*\nabla \chi_{R}|.
\end{equation}
About the first term on the RHS of \eqref{DeltaNablaChir}, note that
\[
\ \Delta (\phi_{R}\circ H)=\phi_{R}^{\prime}\Delta H+ \phi_{R}^{\prime\prime}|\nabla H|^{2}, 
\]
and hence
\[
\ |\nabla^{k-2}\Delta \chi_{R}|=|\nabla^{k-2}(\phi_{R}^{\prime}\Delta H)|+|\nabla^{k-2}(\phi_{R}^{\prime\prime}|\nabla H|^2)|.
\]
We have that $\nabla^{k-2}(\phi_{R}^{\prime}\Delta H)$ can be written as
\begin{equation}\label{otimes1}
\nabla^{k-2}(\phi_{R}^{\prime}\Delta H)=\sum_{\mathbf{c}}C_{\mathbf{c}}\phi^{(c_{1}+\ldots+c_{k-3}+1)}\left[\bigotimes_{t=1}^{k-3}(\nabla^{t}H)^{\otimes c_{t}}\right]\otimes\nabla^{c_{k-2}}\Delta H,
\end{equation}
where $\mathbf{c}$ varies among all $(k-2)$-vectors of nonnegative integers such that $\sum_{t=1}^{k-3}tc_{t}+c_{k-2}=k-2$, and the $C_{\mathbf{c}}$ are positive integer constants.\\
With the same notations, $\nabla^{k-2}(\phi_{R}^{\prime\prime}|\nabla H|^2)$ can be written as
\begin{align}
\nabla^{k-2}(\phi_{R}^{\prime\prime}|\nabla H|^2)&=\sum_{\mathbf{c}}C_{\mathbf{c}}\phi^{(c_{1}+\ldots+c_{k-3}+2)}\left[\bigotimes_{t=1}^{k-3}(\nabla^{t}H)^{\otimes c_{t}}\right]\otimes\nabla^{c_{k-2}}|\nabla H|^2\label{otimes2}\\
&=\sum_{\mathbf{c}}C_{\mathbf{c}}\phi^{(c_{1}+\ldots+c_{k-3}+2)}\left[\bigotimes_{t=1}^{k-3}(\nabla^{t}H)^{\otimes c_{t}}\right]\otimes\nabla^{c_{k-2}-1}(\nabla^{2}H(\nabla H, \cdot)).\nonumber
\end{align}
Inserting \eqref{bd_deriv-phi} in \eqref{otimes1} and using Theorem \ref{th_distancefunctionLapl} (also combined with Lemma \ref{Comm}), we get
\[
\ |\nabla^{k-2}(\phi_{R}^{\prime}\Delta H)|
\leq C\lambda^{1-k}\left[\lambda^{\sum_{t=1}^{k-3}(tc_t-1) + c_{k-2}+1}\right]
\leq C\lambda^{1-k}\lambda^{k-2}=C\lambda^{-1}= C.
\]
Analogously, recalling also that
\[
\ |\nabla^{c_{k-2}-1}(\nabla^2 H(\nabla H, \cdot)|\leq C\left[|\nabla^{c_{k-2}+1}H||\nabla H|+|\nabla^{c_{k-2}}H||\nabla^{2}H|+\ldots+|\nabla^{2}H||\nabla^{c_{k-2}}H|\right],
\]
by \eqref{otimes2}, we get that
\[
\ |\nabla^{k-2}(\phi_{R}^{\prime\prime}|\nabla H|^2)|
\leq C\lambda^{1-k}\left[\lambda^{\sum_{t=1}^{k-3}(tc_t-1) + c_{k-2}}\right]\leq
C\lambda^{1-k}\left[\lambda^{k-2+1}\right]\leq C.
\]
About the second term on the RHS of \eqref{DeltaNablaChir},we have that by \eqref{Ass_p=22} and property (2) of the cut-off functions,
\begin{align*}
|\mathrm{Riem}||\nabla^{k-2}\chi_{R}|\leq& C\lambda^2\lambda^{-k+k-2}=C,\\
|\nabla^{l}\mathrm{Ric}||\nabla^{k-2-l}\chi_{R}|\leq&C\lambda^{2+l}\lambda^{-k+k-2-l}=C, 
\end{align*}
for $l=1,\ldots,k-3$. This concludes the proof of property (3) of the cut-off functions and hence yields the validity of the Lemma.
\end{proof}
%%%%%%%%%%%%%%%%%%%%%%%%%%%%%%%%%%
\subsection{Weitzenb\"ock formulas}\label{subsec_Weitz}
Fix a tensor bundle $E\to M$ with $m$-dimensional fibers, endowed with an inner product induced by the metric $g$ and a compatible connection $\nabla$ induced by the Levi-Civita connection on $M$. A Lichnerowicz Laplacian is a second order differential operator acting on the space of smooth sections $\Gamma(E)$ of the form
\begin{equation}\label{Lic}
\Delta_{L}=\Delta_{B}+c\mathfrak{Ric},
\end{equation}
for a suitable constant $c$. Here $\Delta_{B}=-\mathrm{tr}_{12}(\nabla^2)=\nabla^{*}\nabla$ is the Bochner Laplacian (with $\nabla^{*}$ the formal $L^{2}$-adjoint of $\nabla$) and $\mathfrak{Ric}$ is a smooth symmetric endomorphism of $E$ known as Weitzenb\"ock curvature operator. When $T$ is a $(0,k)$-tensor, the Weitzenb\"ock curvature operator, takes the form
\begin{equation}\label{fkric}
\ \mathfrak{Ric}(T)(X_{1},\ldots, X_{k})=\sum_{i=1}^{k}\sum_{j}\left(\mathrm{R}(E_{j}, X_{i})T\right)(X_{1}, \ldots, E_{j}, \ldots, X_{k}),
\end{equation}
where $\left\{E_{i}\right\}$ is a local orthonormal frame and
\[
\ R(X,Y)\doteq \nabla^{2}_{X,Y}-\nabla^{2}_{Y,X}=\nabla_{X}\nabla_{Y}-\nabla_{Y}\nabla_{X}-\nabla_{[X,Y]},
\]
which may be applied to any tensor field. Note that, by the classical Bochner-Weitzenb\"ock formula, the Hodge Laplacian on exterior differential forms decomposes as in \eqref{Lic} with $c=1$. 

Since, obviously, $R(X,Y)$ vanishes on functions, for any $(0,k)$-tensor $T$, we have that
\begin{align*}
\ \mathfrak{Ric}(T)(X_{1},\ldots,X_{k})=&-\sum_{i=1}^{k}\left[\sum_{j}\sum_{p\neq i}T(X_{1},\ldots,R(E_{j},X_{i})X_{p},\ldots,E_{j},\ldots,X_{k})\right.\\
&\left.+\sum_{j}T(X_{1},\ldots,R(E_{j},X_{i})E_{j},\ldots,X_{k})\right]\\
=& -\sum_{i=1}^{k}\left[\sum_{j}\sum_{p\neq i}T(X_{1},\ldots,\sum_{l}g(R(E_{j},X_{i})X_{p},E_{l})E_{l},\ldots,E_{j},\ldots,X_{k})\right.\\
&\left.-\sum_{\nu}T(X_{1},\ldots,\mathrm{Ric}(E_{\nu},X_{i})E_{\nu},\ldots,X_{k})\right]\\
=&-\sum_{i=1}^{k}\left[\sum_{j}\sum_{p\neq i}\sum_{l}T(X_{1},\ldots, E_{l},\ldots, E_{j},\ldots, X_{k})\mathrm{Riem}(E_{j},X_{i},X_{p},E_{l})\right.\\
&\left.-\sum_{\nu}T(X_{1},\ldots,E_{\nu},\ldots, X_{k})\mathrm{Ric}(X_{i},E_{\nu})\right],
\end{align*}
where we are setting
\begin{align*}
&\mathrm{Riem}(X,Y,Z,W)=g(\mathrm{R}(X,Y)Z,W)\\
&\mathrm{Ric}(X,Y)=\mathrm{tr}\,\mathrm{Riem}(X,\cdot, \cdot, Y).
\end{align*}
\medskip

This curvature term has a quite complicated expression, but it can be estimated in terms of the curvature operator $\mathcal{R}$ of $M$ (the linear extension to $\Lambda^{2}TM$ of the $(2,2)$-Riemann curvature tensor); see e.g. \cite[Corollary 9.3.4]{Petersen_ed3}.

\begin{proposition}\label{ContrCurvTerm}
Let $T$ be a $(0,s)$ tensor. If the curvature operator $\mathcal{R}$ satisfies $\mathcal{R}\geq \alpha$, for some constant $\alpha<0$, then $\langle\mathfrak{Ric}(T), T\rangle\geq \alpha C|T|^2$, with $C$ constant depending only on $s$.
\end{proposition}
%%%%%%%%%%%%%%%%%%%%%%%%%%%%%%%
\subsection{A Lichnerowicz Laplacian on symmetric $(0,k)$-tensors}
Another remarkable example of differential operator which can be rewritten as a Lichnerowicz Laplacian was introduced by J. H. Sampson in \cite{Sampson} on smooth sections of the bundle $S^{(0,k)}(M)$ of totally symmetric $(0,k)$-tensors (see also \cite{Lichnerowicz}). 

Namely, consider the symmetrization operator $s_{k}$, i.e. the projection of the full tensor bundle $T^{0,k}(M)$ on $S^{(0,k)}(M)$. Given $h$ a $(0,k-1)$-tensor field, we can define the totally symmetric tensor
\[
h^{S}(X_{1},\ldots,X_{k-1})\dot =s_{k-1}(h)(X_{1},\ldots,X_{k-1})=\frac{1}{(k-1)!}\sum_{\sigma\in\Pi_{k-1}}h(X_{\sigma(1)},\ldots,X_{\sigma(k-1)}).
\] 
Let us define the operator $D_{S}:\Gamma S^{(0,k-1)}(M)\to\Gamma S^{(0,k)}(M)$ by
\begin{align*}
(D_{S}h^{S})(X_{0}, \ldots, X_{k})=&ks_{k}(\nabla h^{S})(X_{0}, \ldots, X_{k}).
\end{align*}
Its formal  $L^{2}$- adjoint $D_{S}^{*}: \Gamma S^{(0,k)}(M)\to\Gamma S^{(0,k-1)}(M)$ is then given by
\begin{align*}
(D_{S}^{*}h^{S})(X_{1}, \ldots, X_{k-2})=&-\sum_{i}(\nabla_{E_{i}}h^{S})(E_{i}, X_{1}, \ldots, X_{k-2}).
\end{align*}
%Letting $\Delta_{\mathrm{Sym}}=D_{S}^{*}D_{S}-D_{S}D_{S}^{*}$, we have by  \cite{Sampson} (see also Appendix \ref{AppB} for a proof) that
%\[
%\ \Delta_{\mathrm{Sym}}h^{S}=\nabla^{*}\nabla h^{S}-\mathfrak{Ric}(h^{S})=\nabla^{*}\nabla h^{S}-\tilde{\mathfrak{R}}^{(k-1)}(h^{S})+\tilde{\mathfrak{S}}^{(k-1)}(h^{S}),
%\]
%where
%\begin{align*}
% \tilde{\mathfrak{R}}^{(k-1)}(h^{S})(X_{1},\ldots,X_{k-1})\doteq&\sum_{\nu=1}^{k-1}\sum_{l} h^{S}(X_{1},\ldots,E_{l},\ldots,X_{k-1})\mathrm{Ric}(X_{\nu},E_{l});\\
%\tilde{\mathfrak{S}}^{(k-1)}(h^{S})(X_{1},\ldots, X_{k-1})\doteq&\sum_{\mu<\nu}\sum_{j,l}h^{S}(X_{1},\ldots,E_{j},\ldots,E_{l},\ldots,X_{k-1})\cdot\\
%&\cdot\left(\mathrm{R}( E_{j},X_{\mu},X_{\nu},E_{l})+\mathrm{R}(E_{j},X_{\nu},X_{\mu},E_{l})\right),
%\end{align*}
%that is to say, $\Delta_{\mathrm{Sym}}$ is of type \eqref{Lic} for $c=-1$.
Letting $\Delta_{\mathrm{Sym}}=D_{S}^{*}D_{S}-D_{S}D_{S}^{*}$, we have by  \cite{Sampson} (see also Appendix \ref{AppB} for a proof) that
\[
\ \Delta_{\mathrm{Sym}}h^{S}=\nabla^{*}\nabla h^{S}-\mathfrak{Ric}(h^{S}),
\]
that is to say, $\Delta_{\mathrm{Sym}}$ is of type \eqref{Lic} for $c=-1$.

In particular, one can compute that
\begin{align}
\frac{1}{2}\Delta |h^{S}|^2=&-\langle\nabla^{*}\nabla h^{S}, h^{S}\rangle+|\nabla h^{S}|^2\label{Samps2}\\
=&-\langle\Delta_{\mathrm{Sym}}h^{S}, h^{S}\rangle-\langle\mathfrak{Ric}(h^{S}), h^{S}\rangle+|\nabla h^{S}|^2\nonumber
\end{align}

It is worth mentioning that for $(0,2)$-tensors, in \cite[Chapter 9]{Petersen_ed3} it was introduced a different Lichnerowicz Laplacian acting on smooth sections of $S^{(0,2)}(M)$, with particular focus on applications of the Bochner technique to this operator. However the operator $\Delta_{\mathrm{Sym}}$ seems to conform better to our scope.
%%%%%%%%%%%%%%%%%%%%%%%%%%%%%%%%%%%%%%%%%%
\subsection{Proof of Theorem \ref{Dens_p=2}}
In our assumptions, we know by Corollary \ref{kLaplCutOff} that there exists a sequence of cut-off functions $\left\{\chi_{n}\right\}\subset C_{c}^{\infty}(M)$, and a constant $C>0$ independent of $n$ such that, 
\begin{enumerate}
\item $\chi_{n}=1$ on $B_{C_{H}^{-1}(n-2)}(o)$;
\item $|\nabla^{j}\chi_{n}|\leq C\lambda^{-k+j}$, \quad $j=1,\ldots,k-1$;
\item $|\Delta\nabla^{k-2}\chi_{n}|\leq C$,
\end{enumerate}
Since smooth functions are dense in $W^{k,2}(M)$, to prove the density result it is sufficient to consider $f\in C^{\infty}(M)\cap W^{k,2}(M)$; see for instance \cite{GuidettiGuneysuPallar}. We want to prove that $\chi_{n}f$ converges to $f$ in $W^{k,2}(M)$. The lower order terms can be treated as in the proof of Theorem \ref{th_main1Rep}, by using the dominated convergence theorem and the properties of the cut-off functions. Here we prove that
\[
\ \int_{M}|\nabla^{k}(\chi_{n}f)-\nabla^{k}f|^2d\mathrm{vol}_{g}\to 0,
\]
as $n\to\infty$. 
Note that
\begin{align*}
\int_{M}|\nabla^{k}(\chi_{n}f)-\nabla^{k}f|^2d\mathrm{vol}_{g}=&\int_{M}\left|\left[\sum_{i=0}^{k}{{k}\choose{i}} \nabla^{k-i}\chi_{n}\otimes\nabla^{i}f\right]-\nabla^{k}f\right|^2d\mathrm{vol}_{g}\\
\leq&\int_{M}(1-\chi_{n})^2|\nabla^{k}f|^2+\sum_{i=0}^{k-1}{{k}\choose{i}}\int_{M}|\nabla^{k-i}\chi_{n}|^2|\nabla^{i} f|^2d\mathrm{vol}_{g}
\end{align*}
Taking into account the properties of the cut-off functions, the only non-trivial term to study is
\[
\ \int_{M}|f|^{2}|\nabla^{k}\chi_{n}|^2
\]
To prove that this goes to $0$ as $n\to\infty$ we are going to use \eqref{Samps2} with $h\doteq \nabla^{k-1}\chi_{n}$ and hence $h^{S}=s_{k-1}(\nabla^{k-1}\chi_{n})$. First note that, in general,
\begin{align*}
\frac{1}{2}\mathrm{div}\left(f^{2}\nabla\left|h^{S}\right|^2\right)\leq&f^{2}\left[-\langle\Delta_{\mathrm{Sym}}h^{S}, h^{S}\rangle-\langle\mathfrak{Ric}(h^{S}), h^{S}\rangle+|\nabla h^{S}|^2\right]\\
&+2|f||h^{S}||\langle\nabla f, \nabla |h^{S}|\rangle|.
\end{align*}
Hence, by  Young's and Kato's inequality, we get that, for any $0<\eta<1$, 
\begin{align*}
\frac{1}{2}\mathrm{div}\left(f^{2}\nabla |h^{S}|^2\right)+f^{2}\langle\Delta_{\mathrm{Sym}}h^{S}, h^{S}\rangle\leq&-f^{2}\langle\mathfrak{Ric}(h^{S}), h^{S}\rangle+f^2|\nabla h^{S}|^2\\
&+ \eta f^{2}|\nabla|h^{S}||^2+\frac{1}{\eta}|\nabla f|^2|h^{S}|^2\\
\leq&- f^{2}\langle\mathfrak{Ric}(h^{S}), h^{S}\rangle+(1+\eta)f^2|\nabla h^{S}|^2+\frac{1}{\eta}|\nabla f|^2|h^{S}|^2.
\end{align*}
Integrating, we get that
\begin{align}\label{IneqDeltaSym}
\int_{M}\langle\Delta_{\mathrm{Sym}}h^{S}, f^{2}h^{S}\rangle d\mathrm{vol}_{g}\leq &-\int_{M}f^{2}\langle\mathfrak{Ric}(h^{S}), h^{S}\rangle d\mathrm{vol}_{g}\\
&+(1+\eta)\int_{M}f^{2}|\nabla h^{S}|^2d\mathrm{vol}_{g}+\frac{1}{\eta}\int_{M}|\nabla f|^2|h^{S}|^2d\mathrm{vol}_{g},\nonumber
\end{align}
for any $\eta>0$. Recall that $h^{S}=s_{k-1}(\nabla^{k-1}\chi_{n})$. Note that, by Cauchy-Schwarz inequality,
\begin{equation}\label{ConsCS}
\ |s_{k-1}(\nabla^{k-1}\chi_{n})|^2\leq |\nabla^{k-1}\chi_{n}|^2.
\end{equation}
Hence, by the properties of the cut-off functions $\chi_{n}$, the dominated convergence theorem, and the fact that $f\in W^{k,2}(M)$, we have that
\[
\ \int_{M}|\nabla f|^2|s_{k-1}(\nabla^{k-1}\chi_{n})|^2d\mathrm{vol}_{g}\to 0
\]
as $n\to\infty$. Furthermore, the curvature term in \eqref{IneqDeltaSym} can be controlled, under our assumptions,  using Proposition \ref{ContrCurvTerm}. More precisely we have that
\begin{equation}\label{bd_curv}
-\int_{M}f^{2}\langle\mathfrak{Ric}(h^{S}), h^{S}\rangle d\mathrm{vol}_{g}
\leq (-\alpha)C\int_{M}f^{2}|h^{S}|^2 d\mathrm{vol}_{g}.
\end{equation}

Let us now analyze the LHS of \eqref{IneqDeltaSym}. We let
\begin{align*}
A=& \int_{M}\langle D_{S}^{*}D_{S}(s_{k-1}(\nabla^{k-1}\chi_{n})), f^{2}s_{k-1}(\nabla^{k-1}\chi_{n})\rangle d\mathrm{vol}_{g}\\
B=&\int_{M}\langle D_{S}D_{S}^{*}(s_{k-1}(\nabla^{k-1}\chi_{n})),f^{2}s_{k-1}(\nabla^{k-1}\chi_{n})\rangle d\mathrm{vol}_{g},
\end{align*}
so that
\begin{equation}\label{AB}
\int_{M}\langle\Delta_{\mathrm{Sym}}h^S, f^{2}h^S\rangle d\mathrm{vol}_{g}=A-B.
\end{equation}
First, let us deal with the term $B$. Using \eqref{ConsCS} we get
\begin{align}\label{B}
B=&\int_{M}\langle D_{S}^{*}(s_{k-1}(\nabla^{k-1}\chi_{n})), D_{S}^{*}(f^{2}s_{k-1}(\nabla^{k-1}\chi_{n}))\rangle d\mathrm{vol}_{g}\\
=& \int_{M}\left[f^2|D_{S}^{*}(s_{k-1}(\nabla^{k-1}\chi_{n}))|^2-2f\langle i_{\nabla f}(s_{k-1}(\nabla^{k-1}\chi_{n})), D_{S}^{*}(s_{k-1}(\nabla^{k-1}\chi_{n}))\rangle\right]d\mathrm{vol}_{g}\nonumber\\
\leq &2\int_{M}f^{2}|D_{S}^{*}(s_{k-1}(\nabla^{k-1}\chi_{n}))|^2d\mathrm{vol}_{g}+\int_{M}|\nabla f|^{2}|s_{k-1}(\nabla^{k-1}\chi_{n})|^2d\mathrm{vol}_{g}\nonumber\\
\leq &2\int_{M}f^{2}|D_{S}^{*}(s_{k-1}(\nabla^{k-1}\chi_{n}))|^2d\mathrm{vol}_{g}+\int_{M}|\nabla f|^2|\nabla^{k-1}\chi_{n}|^2d\mathrm{vol}_{g}.\nonumber
\end{align}
\smallskip

\emph{In the following we will use  the "$*$" notation we defined in Section \ref{SectNotations}. Moreover, we will work in a normal orthonormal frame $\left\{E_{i}\right\}$ at $p\in M$,  and in frame computations we will use the convention of lowering all indices, summing over repeated indices.
}
\smallskip

Note that
\begin{equation}\label{DS*}
|D_{S}^{*}(s_{k-1}(\nabla^{k-1}\chi_{n}))|^2=\sum_{i_{2}, \ldots, i_{k-1}}\left[\frac{1}{(k-1)!}\sum_{i_{1}}\nabla_{i_1}\left(\sum_{\pi\in\Pi_{k-1}}\nabla^{k-1}_{i_{\pi(1)}\ldots i_{\pi(k-1)}}\chi_{n}\right)\right]^2.
\end{equation}
Recall that for any $(0,r)$-tensor $\alpha$, with $r\geq 1$, the standard commutation formula gives that
\begin{equation}\label{CommTens}
\ \left(\nabla_{i}\nabla_{j}-\nabla_{j}\nabla_{i}\right)\alpha_{i_{1}\ldots i_{r}}=\mathrm{Riem}*\alpha
\end{equation}
Hence for each of the terms in square parentheses on the RHS of \eqref{DS*} we can trace back to a rough Laplacian of a $(k-2)$-th covariant derivative of $\chi_{n}$. For instance, consider the term $\nabla_{i_{1}}\nabla^{k-1}_{i_{2}\ldots i_{k-1}i_{1}}\chi_{n}$. We can compute (at the point $p\in M$ about which we have selected the normal orthonormal frame):
\begin{align*}
\sum_{i_{1}}\nabla_{i_{1}}\nabla^{k-1}_{i_{2}\ldots i_{k-1}i_{1}}\chi_{n}=&\sum_{i_{1}}\nabla_{i_{1}}\left[\left(\nabla^{k-1}_{i_{2}\ldots i_{k-1}i_{1}}-\nabla^{k-1}_{i_{2}\ldots i_{k-2}i_{1}i_{k-1}}\right)\chi_{n}\right.
\\
&\quad\quad\quad\quad\left.+\left(\nabla^{k-1}_{i_{2}\ldots i_{k-2}i_{1}i_{k-1}}-\nabla^{k-1}_{i_{2}\ldots i_{k-3}i_{1}i_{k-2}i_{k-1}}\right)\chi_{n}+\ldots+\nabla^{k-1}_{i_{1}i_{2}\ldots i_{k-1}}\chi_{n}\right]\\
=&\sum_{i_{1}}\nabla_{i_{1}}\left(\left(\nabla^{k-4}\mathrm{Riem}*\nabla \chi_{n}+\ldots+\mathrm{Riem}*\nabla^{k-3}\chi_{n}\right)_{i_{1}\ldots i_{k-1}}+\nabla^{k-1}_{i_{1}i_{2}\ldots i_{k-1}}\chi_{n}\right)\\
=&\left(\nabla^{k-3}\mathrm{Riem}*\nabla\chi_{n}+\ldots+\mathrm{Riem}*\nabla^{k-2}\chi_{n}\right)_{i_{2}\ldots i_{k-1}}+\Delta\nabla^{k-2}_{i_{2}\ldots i_{k-1}}\chi_{n}.
\end{align*} 
The other terms can be treated similarly. Hence, by \eqref{DS*} and Young's inequality, we get that
\begin{align*}
|D_{S}^{*}(s_{k-1}(\nabla^{k-1}\chi_{n}))|^2\leq C(m,k)\left(|\Delta\nabla^{k-2}\chi_{n}|^2+|\mathrm{Riem}|^2|\nabla^{k-2}\chi_{n}|^2+\ldots+|\nabla^{k-3}\mathrm{Riem}|^2|\nabla\chi_{n}|^2\right),
\end{align*}
where (here and from now on) $C(m,k)$ is a constant, depending only on $k$ and $m$, which can possibly change from line to line. Integrating, we get from \eqref{B} that
\begin{align}\label{B1}
\ B\leq &\int_{M} 2C(m,k)f^2\left(|\Delta\nabla^{k-2}\chi_{n}|^2+|\mathrm{Riem}|^2|\nabla^{k-2}\chi_{n}|^2+\ldots+|\nabla^{k-3}\mathrm{Riem}|^2|\nabla\chi_{n}|^2\right)d\mathrm{vol}_{g} \\&+\int_{M}|\nabla f|^2|\nabla^{k-1}\chi_{n}|^2d\mathrm{vol}_{g}.\nonumber
\end{align}
On the other hand, by Young's inequality and reasoning as in \eqref{ConsCS}, we have that
\begin{align}\label{A1}
A=&\int_{M}\langle D_S(s_{k-1}(\nabla^{k-1}\chi_{n})), D_S(f^2s_{k-1}(\nabla^{k-1}\chi_{n}))\rangle d\mathrm{vol}_{g}\\
=&\int_{M}f^{2}|D_{S}(s_{k-1}(\nabla^{k-1}\chi_{n}))|^2d\mathrm{vol}_{g}\nonumber\\
&+\int_{M}k\langle D_{S}(s_{k-1}(\nabla^{k-1}\chi_{n})), 2fs_{k}\left(df\otimes s_{k-1}(\nabla^{k-1}\chi_{n})\right)\rangle d\mathrm{vol}_{g}\nonumber\\
\geq& (1-\delta)\int_{M}f^{2}|D_{S}(s_{k-1}(\nabla^{k-1}\chi_{n}))|^2d\mathrm{vol}_{g}-\frac{k^2}{\delta}\int_{M}|\nabla f|^2|\nabla^{k-1}\chi_{n}|^2d\mathrm{vol}_{g},\nonumber
\end{align}
for any $\delta>0$.
Substituting \eqref{bd_curv} in \eqref{IneqDeltaSym}, then this latter, \eqref{B1} and \eqref{A1} in \eqref{AB}, we hence get that
\begin{align*}
&(-\alpha)C\int_{M}f^{2}|s_{k-1}(\nabla^{k-1}\chi_{n})|^2 d\mathrm{vol}_{g}+(1+\eta)\int_{M}f^{2}|\nabla(s_{k-1}(\nabla^{k-1}\chi_{n}))|^2d\mathrm{vol}_{g}\\
&+\frac{1}{\eta}\int_{M}|\nabla f|^2|s_{k-1}(\nabla^{k-1}\chi_{n})|^2d\mathrm{vol}_{g}\\
\geq&(1-\delta)\int_{M}f^{2}|D_{S}(s_{k-1}(\nabla^{k-1}\chi_{n}))|^2d\mathrm{vol}_{g}-\frac{k^2}{\delta}\int_{M}|\nabla f|^2|\nabla^{k-1}\chi_{n}|^2d\mathrm{vol}_{g}\\
&-\int_{M} C(m,k)f^2\left(|\Delta\nabla^{k-2}\chi_{n}|^2+|\mathrm{Riem}|^2|\nabla^{k-2}\chi_{n}|^2+\ldots+|\nabla^{k-3}\mathrm{Riem}|^2|\nabla\chi_{n}|^2\right)d\mathrm{vol}_{g} \\
&-\int_{M}|\nabla f|^2|\nabla^{k-1}\chi_{n}|^2d\mathrm{vol}_{g}.
\end{align*}
Using our assumptions, the properties of the rough Laplacian cut-off functions $\chi_{n}$, Proposition \ref{ContrCurvTerm}, the dominated convergence theorem, and the fact that $f\in W^{k,2}(M)$, this yields that
\begin{equation}\label{deltaeta}
\limsup_{n\to \infty}\int_{M}f^{2}\left[(1-\delta)|D_{S}(s_{k-1}(\nabla^{k-1}\chi_{n}))|^2-(1+\eta)|\nabla s_{k-1}(\nabla^{k-1}\chi_{n})|^2\right]d\mathrm{vol}_{g}\le 0.
\end{equation}

\smallskip

Let us first study the term $|\nabla s_{k-1}(\nabla^{k-1}\chi_{n})|^2$. Since, by definition,
\[
s_{k-1}(\nabla^{k-1}\chi_{n})(E_{i_{1}},\ldots,E_{i_{k-1}})=\frac{1}{(k-1)!}\left(\sum_{\pi\in\Pi_{k-1}}\nabla^{k-1}_{i_{\pi(1)}\ldots i_{\pi(k-1)}}\chi_{n}\right),
\]
using \eqref{CommTens}, we have that
\begin{align*}
&((k-1)!)^2|\nabla s_{k-1}(\nabla^{k-1}\chi_{n})|^2\\
=&\sum_{i_{1},\ldots,i_{k}}\left[\sum_{\pi\in\Pi_{k-1}}\nabla^{k}_{i_{k}i_{\pi(1)}\ldots i_{\pi(k-1)}}\chi_{n}\right]^2\\
=&(k-1)!|\nabla^{k}\chi_{n}|^2+\sum_{i_{1},\ldots,i_{k}}\sum_{\pi\in \Pi_{k-1}}(\nabla^{k}_{i_{k}i_{\pi(1)}\ldots i_{\pi(k-1)}}\chi_{n}\cdot\sum_{\sigma\in\Pi_{k-1}\setminus\left\{\pi\right\}}\nabla^{k}_{i_{k}i_{\sigma(1)}\ldots i_{\sigma(k-1)}}\chi_{n})\\
=&\left((k-1)!+((k-1)!-1)(k-1)!\right)|\nabla^{k}\chi_{n}|^2\\
&+\left(\nabla^{k}\chi_{n}*\left(\mathrm{Riem}*\nabla^{k-2}\chi_{n}+\nabla\mathrm{Riem}*\nabla^{k-3}\chi_{n}+\ldots\nabla^{k-3}\mathrm{Riem}*\nabla\chi_{n}\right)\right)\\
=&\left((k-1)!\right)^2|\nabla^{k}\chi_{n}|^2+\left(\nabla^{k}\chi_{n}*\left(\mathrm{Riem}*\nabla^{k-2}\chi_{n}+\nabla\mathrm{Riem}*\nabla^{k-3}\chi_{n}+\ldots\nabla^{k-3}\mathrm{Riem}*\nabla\chi_{n}\right)\right).
\end{align*}
Hence, using Young's inequality,
\begin{align}\label{nablask-1}
|\nabla s_{k-1}(\nabla^{k-1}\chi_{n})|^2\leq & |\nabla^{k}\chi_{n}|^2+C(m,k)\left[\varepsilon|\nabla^{k}\chi_{n}|^2\right.\\
&\left.+\frac{1}{\varepsilon}\left(|\mathrm{Riem}|^2|\nabla^{k-2}\chi_{n}|^2+|\nabla\mathrm{Riem}|^2|\nabla^{k-3}\chi_{n}|^2+\ldots+|\nabla^{k-3}\mathrm{Riem}|^2|\nabla\chi_{n}|^2\right)\right]\nonumber
\end{align}
for any $\varepsilon>0$. On the other hand, concerning the other term in \eqref{deltaeta}, we have that
\begin{align*}
|D_{S}(s_{k-1}(\nabla^{k-1}\chi_{n}))|^2=&k^2|s_{k}(\nabla s_{k-1}(\nabla^{k-1}\chi_{n}))|^2.
\end{align*} 
Since $s_{k}(\nabla s_{k-1}(\nabla^{k-1}\chi_{n}))=s_{k}(\nabla^{k}\chi_{n})$, we thus have that
\begin{align*}
&|D_{S}(s_{k-1}(\nabla^{k-1}\chi_{n}))|^2=k^{2}|s_{k}(\nabla^{k}\chi_{n})|^2\\
=&\frac{k^2}{(k!)^2}\sum_{i_{1},\ldots,i_{k}}\left[\sum_{\pi\in\Pi_{k}}\nabla^{k}_{i_{\pi(1)}\ldots,i_{\pi(k)}}\chi_{n}\right]^2\\
=&\frac{k^2}{(k!)^2}\left[k!|\nabla^{k}\chi_{n}|^2+\sum_{i_{1},\ldots,i_{k}}\sum_{\pi\in \Pi_{k}}(\nabla^{k}_{i_{\pi(1)}\ldots i_{\pi(k)}}\chi_{n}\cdot\sum_{\sigma\in\Pi_{k}\setminus\left\{\pi\right\}}\nabla^{k}_{i_{\sigma(1)}\ldots i_{\sigma(k)}}\chi_{n})\right]\\
=&\frac{k^2}{(k!)^2}\left[\left(k!+(k!-1)k!\right)|\nabla^{k}\chi_{n}|^2\right.\\
&\left.+\left(\nabla^{k}\chi_{n}*\left(\mathrm{Riem}*\nabla^{k-2}\chi_{n}+\nabla\mathrm{Riem}*\nabla^{k-3}\chi_{n}+\ldots\nabla^{k-3}\mathrm{Riem}*\nabla\chi_{n}\right)\right)\right]\\
=&\frac{k^2}{(k!)^2}\left[\left(k!\right)^2|\nabla^{k}\chi_{n}|^2\right.\\
&\left.+\left(\nabla^{k}\chi_{n}*\left(\mathrm{Riem}*\nabla^{k-2}\chi_{n}+\nabla\mathrm{Riem}*\nabla^{k-3}\chi_{n}+\ldots\nabla^{k-3}\mathrm{Riem}*\nabla\chi_{n}\right)\right)\right].
\end{align*}
Using again Young's inequality we get
\begin{align}\label{DSsk-1}
&|D_{S}(s_{k-1}(\nabla^{k-1}\chi_{n}))|^2\geq k^2|\nabla^{k}\chi_{n}|^2+C(m,k)\left[-\varepsilon|\nabla^{k}\chi_{n}|^2\right.\\
&\left.-\frac{1}{\varepsilon}\left(|\mathrm{Riem}|^2|\nabla^{k-2}\chi_{n}|^2+|\nabla\mathrm{Riem}|^2|\nabla^{k-3}\chi_{n}|^2+\ldots+|\nabla^{k-3}\mathrm{Riem}|^2|\nabla \chi_{n}|^2\right)\right],\nonumber
\end{align}
for any $\varepsilon>0$. Using \eqref{nablask-1} and \eqref{DSsk-1}, we get that
\begin{align*}
&\int_{M}f^{2}\left[(1-\delta)\left(k^2-C(m,k)\varepsilon\right)-(1+\eta)\left(1+C(m,k)\varepsilon\right)\right]|\nabla^{k}\chi_{n}|^2d\mathrm{vol}_{g}\\
\leq&\int_{M}f^{2}\left[(1-\delta)|D_{S}(s_{k-1}(\nabla^{k-1}\chi_{n}))|^2-(1+\eta)|\nabla s_{k-1}(\nabla^{k-1}\chi_{n})|^2\right]d\mathrm{vol}_{g}\\
&+ \int_{M}f^{2}\left[(1-\delta)C(m,k)\frac{1}{\varepsilon}\left(|\mathrm{Riem}|^2|\nabla^{k-2}\chi_{n}|^2+\ldots+|\nabla^{k-3}\mathrm{Riem}|^2|\nabla \chi_{n}|^2\right)\right]d\mathrm{vol}_{g}\\
&+\int_{M}f^{2}\left[(1+\eta)C(m,k)\frac{1}{\varepsilon}\left(|\mathrm{Riem}|^2|\nabla^{k-2}\chi_{n}|^2+\ldots+|\nabla^{k-3}\mathrm{Riem}|^2|\nabla \chi_{n}|^2\right)\right]d\mathrm{vol}_{g}.
\end{align*}
By  \eqref{deltaeta} and reasoning as above we know that the RHS converge to $0$ as $n\to\infty$. Hence, suitably choosing $\varepsilon,\,\eta,\,\delta\ll 1$ we finally obtain the desired conclusion:
\[
\ \int_{M}f^{2}|\nabla^{k}\chi_{n}|^2d\mathrm{vol}_{g}\to 0,
\]
as $n\to\infty$.
%%%%%%%%%%%%%%%%%%%%%%%%%%%%%%%%%%%%%%%%%%

\section{Some sharp applications}\label{SectSharpAppl}
\subsection{Disturbed Sobolev inequalities}

First, we point out the following

\begin{theorem}\label{th_sob}
Let $(M^m,g)$ be a smooth, complete non-compact Riemannian manifold without boundary. Let $o\in M$, $r(x)\doteq \mathrm{dist}_{g}(x,o)$ and suppose that for some $\eta>0$, $D>0$ and some $i_{0}>0$,
\begin{equation*}%\label{HpSob}
\ |\mathrm{Ric}_{g}|(x)\leq D^2(1+r(x)^2)^{\eta},\quad\mathrm{inj}_g(x)\geq \frac{i_{0}}{D(1+r(x))^\eta}.
\end{equation*}
Let $p\in[1,m)$ and $q\in [p,mp/(m-p)]$. Then there exist constants $A_1>0$, $A_2>0$, depending on $m$, $p$, $q$ and the constant $C$ from Theorem \ref{th_distancefunction}, such that for all $\varphi\in C^\infty_c(M)$ it holds
\begin{align}\label{w-sob}
\left(\int_M |\varphi|^{q}d\mathrm{vol}_{g}\right)^{\frac{1}{q}}
&\leq A_1 \left(\int_M |\nabla \varphi|^p d\mathrm{vol}_{g}\right)^{1/p} + A_2 \left(\int_M H^{2\eta} |\varphi|^p d\mathrm{vol}_{g}\right)^{1/p},
\end{align}
where $H\in C^\infty(M)$ is the distance-like function given by Theorem \ref{th_distancefunction}.
\end{theorem}

\begin{remark}\label{rmketa}{\rm Theorem \ref{th_sob} was proved in \cite{IRV-HessCutOff} with $\eta\leq 1$. However the proof therein works also when $\eta>1$ up to replace \cite[Theorem 1.5]{IRV-HessCutOff} with the $k=2$ case of Theorem \ref{th_distancefunction} above, with $\lambda(t)=D^2(1+t^2)^{\frac{\eta}{2}}$. Note that one could also state the theorem for more general growth functions $\lambda(r)$ as in Theorem \ref{th_distancefunction}.}
\end{remark} 

Since, under a $C^k$-control on the curvature, there exists distance-like functions with controlled higher order derivatives, one naturally expects that some sort of improved higher order Sobolev inequality should be obtained exploiting the control on the higher derivatives of the curvature. However, for the moment, this possible phenomenon remains unclear to us. Indeed, generalizing a fact remarked in \cite[Proposition 2.11]{Aubin} for the standard (i.e. non-disturbed) Sobolev inequalities, higher order disturbed Sobolev inequalities hold true under exactly the same assumptions as Theorem \ref{th_sob}.

\begin{proposition}\label{th_sob_HO}
Let $(M^m,g)$ be a smooth, complete non-compact Riemannian manifold without boundary. Let $o\in M$, $r(x)\doteq \mathrm{dist}_{g}(x,o)$ and suppose that for some $\eta>0$, $D>0$ and some $i_{0}>0$,
\begin{equation*}%\label{HpSob}
\ |\mathrm{Ric}_{g}|(x)\leq D^2(1+r(x)^2)^{\eta},\quad\mathrm{inj}_g(x)\geq \frac{i_{0}}{D(1+r(x))^\eta}.
\end{equation*}
Let $p\in[1,m)$ and let $k$ be an integer in $[1,\frac mp)$. Then there exists a constant $A>0$ depending on $m$, $p$, $k$ and the constant $C$ from Theorem \ref{th_distancefunction}, such that for all $\varphi\in C^\infty_c(M)$ it holds
\begin{align}\label{w-sob-ho}
\|\varphi\|_{L^{\frac{pm}{m-kp}}(M)}\leq A \sum_{r=-1}^{k-1}\left\| H^{\frac\eta{mp}(r+1)(2m-rp)}|\nabla^{k-r-1}\varphi|\right \|_{L^p(M)},
\end{align}
where $H\in C^\infty(M)$ is the distance-like function given by Theorem \ref{th_distancefunction}.
\end{proposition}

\begin{proof}
For the ease of notation we will write $\|\cdot\|_{p}$ for $\|\cdot\|_{L^p(M)}$. For $j=1,\dots,k$, define $q_j\doteq pm/(m-jp)$.
We prove that for every $s=0,\dots,k-1$ it holds
\begin{align}\label{w-sob-hoind}
\|\varphi\|_{q_k}\leq C \sum_{r=s-1}^{k-1}\left\| H^{2\eta\sum_{j=s}^r q_j^{-1}}|\nabla^{k-r-1}\varphi|\right \|_{q_s},
\end{align}
with the convention that $\sum_{j=0}^{-1} q_j^{-1}=0$. Since $\sum_{j=0}^r q_j^{-1}=(2m-rp)(r+1)/(2mp)$, the inequality \eqref{w-sob-ho} is equivalent to \eqref{w-sob-hoind} when $s=0$.

Note that $p=q_0<q_1<\cdots <q_k$, and that $q_j=nq_{j-1}/(n-q_{j-1})$, so that Theorem \ref{th_sob} applies with $p=q_{j-1}$ and $q=q_j$. In particular, \eqref{w-sob-hoind} holds true when $s=k-1$. For general $s$, we prove now the validity of \eqref{w-sob-hoind} by backward induction. Namely, suppose that for some integer $t\in [1, k-1]$, \eqref{w-sob-hoind} holds true with $s=t$, i.e.
\begin{align*}
\|\varphi\|_{q_k}\leq C  \sum_{r=t-1}^{k-1}\left\| H^{2\eta\sum_{j=t}^r q_j^{-1}}|\nabla^{k-r-1}\varphi|\right \|_{q_t}.
\end{align*}
Applying \eqref{w-sob} with $p=q_{t-1}$ and $q=q_t$ at each terms of RHS gives 
\begin{align*}
\|\varphi\|_{q_k}&\leq C \left[\sum_{r=t-1}^{k-1}\left\| H^{\frac{2\eta}{q_{t-1}}}H^{2\eta\sum_{j=t}^r q_j^{-1}}|\nabla^{k-r-1}\varphi|\right \|_{q_{t-1}} +\sum_{r=t-1}^{k-1}\left\| \nabla\left(H^{2\eta\sum_{j=t}^r q_j^{-1}}|\nabla^{k-r-1}\varphi|\right)\right \|_{q_{t-1}}\right]\\
&\leq C \left[ \sum_{r=t-1}^{k-1}\left\| H^{2\eta\sum_{j=t-1}^r q_j^{-1}}|\nabla^{k-r-1}\varphi|\right \|_{q_{t-1}} + 
\sum_{r=t-1}^{k-1}\left\| \left(H^{2\eta\sum_{j=t}^r q_j^{-1}-1}|\nabla H||\nabla^{k-r-1}\varphi|\right)\right \|_{q_{t-1}}\right.\\
&+ \left.\sum_{r=t-1}^{k-1}\left\| H^{2\eta\sum_{j=t}^r q_j^{-1}}\nabla|\nabla^{k-r-1}\varphi|\right \|_{q_{t-1}}
\right].
\end{align*}
Recalling that $|\nabla H|\leq 1$, and that $|\nabla|\nabla^{k-r-1}\varphi||\leq |\nabla^{k-r}\varphi|$ (see \cite[p. 36, (1)]{Aubin}), we obtain
\begin{align*}
\|\varphi\|_{q_k}
&\leq C \left[ \sum_{r=t-1}^{k-1}\left\| H^{2\eta\sum_{j=t-1}^r q_j^{-1}}|\nabla^{k-r-1}\varphi|\right \|_{q_{t-1}} + 
\sum_{r=t-1}^{k-1}\left\| \left(H^{2\eta\sum_{j=t}^r q_j^{-1}}|\nabla^{k-r-1}\varphi|\right)\right \|_{q_{t-1}}\right.\\
&+ \left.\sum_{r=t-1}^{k-1}\left\| H^{2\eta\sum_{j=t}^r q_j^{-1}}|\nabla^{k-r}\varphi|\right \|_{q_{t-1}}
\right]\\
&\leq C \left[ \sum_{r=t-1}^{k-1}\left\| H^{2\eta\sum_{j=t-1}^r q_j^{-1}}|\nabla^{k-r-1}\varphi|\right \|_{q_{t-1}} + \sum_{r=t-2}^{k-2}\left\| H^{2\eta\sum_{j=t}^{r+1} q_j^{-1}}|\nabla^{k-r-1}\varphi|\right \|_{q_{t-1}}
\right].
\end{align*}
Since $\sum_{j=t}^{r+1} q_j^{-1}\leq\sum_{j=t-1}^{r} q_j^{-1}$, we finally get
\begin{align*}
\|\varphi\|_{q_k}\leq C  \sum_{r=t-2}^{k-1}\left\| H^{2\eta\sum_{j=t-1}^r q_j^{-1}}|\nabla^{k-r-1}\varphi|\right \|_{q_{t-1}},
\end{align*}
i.e. \eqref{w-sob-hoind} holds true for $s=t-1$ as desired.
\end{proof}

%%%%%%%%%%%%%%%%%%%%%%%%%%%%%%%%%%%%%%%%%%
\subsection{Calder\'on-Zygmund inequalities}
Calder\'on-Zygmund inequalities are a powerful tool in Euclidean analysis which permits to control the $L^p$-norm of the Hessian of a function $u$ in terms of the $L^p$-norms of the Laplacian of $u$ and of $u$ itself. On a complete non-compact Riemannian manifolds $(M,g)$, it was proved by B. G\"uneysu and S. Pigola in \cite{GuneysuPigola} that the global Calder\'on-Zygmund inequality
\begin{equation}\label{CZp}\tag{CZ(p)}
\forall u \in C^\infty_c(M),\quad \||\mathrm{Hess}\,u|_{g}\|_{L^p}^p \leq  A_1 \|u\|_{L^p}^p +A_2 \|\Delta u\|_{L^p}^p
\end{equation}
holds for $p\in[1,+\infty)$ if either 
\begin{itemize}
\item $p=2$ and $\mathrm{Ric}_g\geq -C$ for some $C>0$, or
\item $|\mathrm{Ric}_g| \leq C$ for some $C>0$ and $\mathrm{inj}_g(M)>i_0>0$.
\end{itemize}
In general there is no hope to get \eqref{CZp} on an arbitrary complete non-compact manifold, due to counterexamples, \cite{GuneysuPigola,Shi}. However one can weaken the assumptions in \cite{GuneysuPigola} obtaining a weaker version of \eqref{CZp} in which an unbounded weight function appears in the $\|u\|_{L^p}$ term of \eqref{CZp}. This approach was considered for instance in \cite{IRV-HessCutOff}, where the authors proved (a slightly weaker version of) the following

\begin{theorem}\label{th_CZ_dist}
Let $(M^m,g)$ be a smooth, complete non-compact Riemannian manifold without boundary. Let $o\in M$, $r(x)\doteq \mathrm{dist}_{g}(x,o)$ and suppose  that one of the following curvature assumptions holds
\begin{itemize}
\item[(a)] for some $\eta>0$, some $D>0$ and some $i_0>0$,
\[
\ |\mathrm{Ric}_{g}|(x)\leq D^2(1+r(x)^2)^{\eta},\quad\mathrm{inj}_{g}(x)\geq \frac{i_0}{D(1+r(x))^{\eta}}>0\quad\mathrm{on}\,\,M.
\]
\item[(b)] for some $\eta>0$ and some $D>0$,
\[
\ |\mathrm{Sect}_{g}|(x)\leq D^2(1+r(x)^2)^\eta.
\]
\end{itemize}
Then there exist constants $A_1>0$, $A_2>0$, depending on $m$, $\eta$, $D$ and the constant $C$ from Theorem \ref{th_distancefunction}, such that for all $\varphi\in C^\infty_c(M)$ it holds
\begin{align}\label{CZ-state}
\||\mathrm{Hess}\,\varphi|_g\|_{L^2}^2 \leq  A_1 \|H^{2\eta}\varphi\|_{L^2}^2 +A_2 \|\Delta \varphi\|_{L^2}^2,
\end{align}
where $H\in C^\infty(M)$ is the distance-like function given by Theorem \ref{th_distancefunction}.
\end{theorem}

\begin{remark}{\rm The same observation as in Remark \ref{rmketa} applies also to Theorem \ref{th_CZ_dist}.}
\end{remark} 

As commented in \cite{IRV-HessCutOff}, obtaining a weighted $L^p$ Calder\'on-Zygmund inequality under the same assumptions of Theorem \ref{th_CZ_dist} is a non-trivial problem. The main issue is to keep a control on the injectivity radius under the conformal deformation. However, it turns out that this can be done at least under slightly stronger assumptions, i.e. if we assume both a control on the sectional curvatures and on the injectivity radius. Accordingly, we can obtain Theorem \ref{th_CZp_dist} stated in Section \ref{Intro}, which we state here again for readers' convenience.

\begin{theorem}
Let $(M^m,g)$ be a smooth, complete non-compact Riemannian manifold without boundary. Let $o\in M$, $r(x)\doteq \mathrm{dist}_{g}(x,o)$ and suppose  that  for some $\eta>0$, some $D>0$ and some $i_0>0$,
\[
\ |\mathrm{Sect}_{g}|(x)\leq D^2(1+r(x)^2)^{\eta},\quad\mathrm{inj}_{g}(x)\geq \frac{i_0}{D(1+r(x))^{\eta}}>0\quad\mathrm{on}\,\,M.
\]
Then there exist constants $A>0$ depending on $m$, $\eta$, $D$, $i_0$ and the constant $C$ from Theorem \ref{th_distancefunction}, such that for all $\varphi\in C^\infty_c(M)$ it holds
\begin{align*}
\||\mathrm{Hess}\,\varphi|_g\|_{L^p}^p \leq  A \left[ \|H^{2\eta}\varphi\|_{L^p}^p + \|\Delta \varphi\|_{L^p}^p\right],
\end{align*}
where $H\in C^\infty(M)$ is the distance-like function given by Theorem \ref{th_distancefunction}.
\end{theorem}

\begin{proof}
We can write the geometric assumption in the following form:  for some $\eta>0$, some $D'>0$ and some $i'_0>0$, 
\[
\ |\mathrm{Sect}_{g}|(x)\leq \bar D^2r(x)^{2\eta},\quad\mathrm{inj}_{g}(x)\geq \frac{\bar i_0}{r(x)^{\eta}}>0\quad\mathrm{on}\,\,M\setminus B_1(o).
\]

As in the proof of \cite[Theorem 1.7]{IRV-HessCutOff}, define the new complete conformal metric $\tilde g=H^{2s}g$, which satisfies in particular
\begin{equation}\label{sect_bd}
|\mathrm{Sect}_{\tilde g}|\leq K^2
\end{equation}
on $M$ for some constant $K>0$. Here $H$ is the second order distance-like function whose existence is guaranteed in our assumptions by Theorem \ref{th_distancefunction}.
For later purposes, observe that reasoning as in the proof of \cite[Lemmas 7.3 and 7.4]{IRV-HessCutOff} one can prove the existence of a constant $\tilde c$ such that for all $j$ large enough and all $\rho>0$, 
\[B_{\tilde c^{-1} \rho r(x_j)^{-\eta}}^g(x_j)\subset B^{\tilde g}_\rho(x_j)\subset B_{\tilde c\rho r^{-\eta}(r_j)}^g(x_j).\]

We claim that 
\begin{equation}\label{claim}
\exists\, \tilde i_0>0,\quad \forall x\in M,\quad \mathrm{inj}_{\tilde g}(x)\geq \tilde  i_0>0.\end{equation} 
Suppose it is not the case. Then there exists a sequence of points $\{x_j\}_{j=1}^{\infty}\subset M\setminus B_1^g(0)$ such that $\mathrm{inj}_{\tilde g}(x_j)< \frac{1}{j}$ and $\mathrm{dist}_{\tilde g}(x_j,o)\to\infty$ as $j\to\infty$. 

We call $F_j:\mathbb B_{\frac{\pi}{2K}}(0)\subset T_{x_j}M\to M$ the exponential map $\exp_{x_j}$ of the metric $\tilde g$ restricted to $\mathbb B_{\frac{\pi}{2K}}(0)\subset T_{x_j}M$. Since $\mathrm{Sect}_{\tilde g}\leq K^2$, by Rauch comparison theorem in $B^{\tilde g}_{\frac{\pi}{2K}}(x_j)$ there are no points conjugated to $x_j$ for the metric $\tilde g$. In particular $F_j$ is a smooth local diffeomorphism onto $B^{\tilde g}_{\frac{\pi}{2K}}(x_j)$ and a smooth surjective local isometry once $T_{x_j}M$ is endowed with the pulled-back metric $\exp_{x_j}^{\ast} \tilde g$.
Hence, $\mathrm{inj}_{\tilde g}(x_j)< \frac{1}{j}$ implies that for $j>\frac{2K}{\pi}$, there exists a point $y_j\in B^{\tilde g}_{j^{-1}}(x_j)$ such that $y_j=F(Y_{j,1})=F(Y_{j,2})$ for two different points $Y_{j,1}$ and $Y_{j,2}$ in $\mathbb B_{j^{-1}}(0)\subset T_{x_j}M$. Also, there are two different constant speed minimizing geodesics $\tilde\gamma_{j,1}(t)$ and $\tilde\gamma_{j,1}(t)$ wrt $\tilde g$ which connect $x_j$ to $y_j$ and lifts to $\tilde\Gamma_{j,1}(t)=tY_{j,1}$ and $\tilde\Gamma_{j,2}(t)=tY_{j,2}$ respectively. 

Let $j>\frac{2\tilde c^2K}{\pi}$, so that 
\begin{equation}\label{ballschain}
B^{\tilde g}_{j^{-1}}(x_j)\subset B_{\tilde cj^{-1} r^{-\eta}(r_j)}^g(x_j)\subset B^{\tilde g}_{\tilde c^2j^{-1}}(x_j)\subset B^{\tilde g}_{\frac{2K}{\pi}}(x_j).\end{equation}
Since $F$ is a local isometry, $F^{-1}(y_j)$ is a discrete set, so that $\tilde\gamma_{j,1}$ and $\tilde\gamma_{j,2}$ belong necessarily to two different relative homotopy classes $[\tilde\gamma_{j,1}]_{\{x_j,y_j\}}$ and $[\tilde\gamma_{j,2}]_{\{x_j,y_j\}}$ of paths in $B^{\tilde g}_{\tilde c^2j^{-1}}(x_j)$ with fixed boundary points $x_j$ and $y_j$. 
Because of \eqref{ballschain} we have also that $\gamma_{j,1}$ and $\gamma_{j,2}$ are curves contained in $B_{\tilde cj^{-1} r^{-\eta}(r_j)}^g(x_j)$ which belong to two different relative homotopy classes $[\tilde\gamma_{j,1}]_{\{x_j,y_j\}}$ and $[\tilde\gamma_{j,2}]_{\{x_j,y_j\}}$ of paths in $B_{\tilde cj^{-1} r^{-\eta}(x_j)}^g(x_j)$ with fixed boundary points $x_j$ and $y_j$. 
However, 
\[
\mathrm{inj}_{g}(x_j)\geq \frac{\bar i_0}{r(x_j)^{\eta}}>\tilde cj^{-1} r^{-\eta}(x_j),
\]
as soon as $j>\tilde c/i_0$. Hence, for $j$ large enough, $B_{\tilde cj^{-1} r^{-\eta}(x_j)}^g(x_j)$ is contractible, which contradicts the existence of two different homotopy classes of curves. The claim \eqref{claim} is thus proved.

Recalling also \eqref{sect_bd}, we have thus the validity on $(M,\tilde g)$ of an $L^p$ Calder\'on-Zygmund inequality, \cite[Theorem C]{GuneysuPigola}, i.e.,
\begin{align}\label{CZtilde}
\forall u \in C^\infty_c(M),\quad \||\mathrm{Hess}_{\tilde{g}}\,u|_{\tilde g}\|_{\widetilde{L^p}}^p \leq  A_1 \|u\|_{\widetilde{L^p}}^p +A_2 \|\Delta_{\tilde{g}} u\|_{\widetilde{L^p}}^p.
\end{align}
Here and on $\|\cdot\|_{\widetilde{L^p}}$ is the $L^p$ norm of a function computed with respect to the conformally deformed metric $\tilde g$ on $M$. 
Given $\varphi\in C^\infty_c(M)$, define $u\in C^\infty_c(M)$ by $H^{\eta(\frac{m}{p}-2)}u=\varphi$. Setting $e^\phi=H^\eta$ and computing as in \cite[Section 7]{IRV-HessCutOff},we get that
\begin{align*}
e^{4\phi}|\mathrm{Hess}_{\tilde{g}}\,u|_\tg^2 
&= |\mathrm{Hess}\,u|_g^2+2 |\nabla u|_g^2|\nabla \phi|_g^2 +(m-2)g(\nabla u,\nabla\phi)^2 - 2g(\nabla u,\nabla\phi)\Delta_gu - 4\mathrm{Hess}\,u(\nabla u,\nabla \phi)\\
&\geq |\mathrm{Hess}\,u|_g^2+2 |\nabla u|_g^2|\nabla \phi|_g^2 -(m-2)|\nabla u|_g^2|\nabla\phi|_g^2 - |\nabla u|_g^2|\nabla\phi|_g^2 - |\Delta_gu|^2\\ 
&- \frac{1}{2} |\mathrm{Hess}\,u|_g^2 - 8 |\nabla u|_g^2|\nabla\phi|_g^2\\
&\geq \frac{1}{2} |\mathrm{Hess}\,u|_g^2 - (m+5)|\nabla u|_g^2|\nabla\phi|_g^2- |\Delta_gu|^2. 
\end{align*}
Accordingly
\begin{align*}
|\mathrm{Hess}\,u|_g^2
&\leq C \left[H^{4\eta}|\mathrm{Hess}_{\tilde{g}}\,u|_\tg^2 
+ \eta^2|\nabla u|_g^2|\nabla\log H|_g^2+ |\Delta_gu|^2\right],
\end{align*}
and
\begin{align}\label{eqhess}
H^{(m-2 p)\eta}|\mathrm{Hess}\,u|_g^p\, d\mathrm{vol}_g
\leq& C \left[|\mathrm{Hess}_{\tilde{g}}\,u|_\tg^p 
\,d\mathrm{vol}_{\tilde g}\right.\\
&\,\,\,\,\,\,\left.+ H^{(m-2 p)\eta}|\nabla u|_g^p|\nabla\log H|_g^p \,d\mathrm{vol}_g+ H^{(m-2 p)\eta}|\Delta_gu|^p\, d\mathrm{vol}_g\right].\nonumber
\end{align}
From \eqref{CZtilde} we get
\begin{align*}
\int_M |\mathrm{Hess}_{\tilde{g}}\,u|_\tg^p 
\,d\mathrm{vol}_{\tilde g} 
&\leq A_1 \int_M |u|^p 
\,d\mathrm{vol}_{\tilde g} + A_2\int_M |\Delta_{\tilde{g}} u|^p 
\,d\mathrm{vol}_{\tilde g}\\
&\leq C\left[ 
\int_M |u|^p 
\,d\mathrm{vol}_{\tilde g} + \int_M H^{-2\eta p}|\Delta_g u + (m-2) g(\nabla u, \eta\nabla \log H)|^p 
\,d\mathrm{vol}_{\tilde g}
\right]\\
& \leq C\left[ 
\int_M H^{m\eta}|u|^p 
\,d\mathrm{vol}_{g} + \int_M H^{(m-2p)\eta}|\Delta_g u|^p 
\,d\mathrm{vol}_{g}\right.\\
&\,\,\,\,\,\,\,\,\,\,\,\,\,\,\left. + \int_M H^{(m-2p)\eta}|\nabla u|_g^p|\nabla \log H|_g^p 
\,d\mathrm{vol}_{g}
\right].
\end{align*}
Inserting in \eqref{eqhess} gives
\begin{align}\label{eqhess2}
\int_M H^{(m-2 p)\eta}|\mathrm{Hess}\,u|_g^p\, d\mathrm{vol}_g
 \leq &C\left[ 
\int_M H^{m\eta}|u|^p 
\,d\mathrm{vol}_{g} + \int_M H^{(m-2p)\eta}|\Delta_g u|^p 
\,d\mathrm{vol}_{g}\right.\\
 &\,\,\,\,\,\,\,+ \left.\int_M H^{(m-2p)\eta}|\nabla u|_g^p|\nabla \log H|_g^p 
\,d\mathrm{vol}_{g}
\right].\nonumber
\end{align}
Recall that $u=H^{-\eta(\frac{m}{p}-2)}\varphi$. Then 
\[
H^{\eta(\frac{m}{p}-2)}|\nabla_g u| \leq C [|\nabla_g\varphi| + H^{-1}\varphi|\nabla_g H| ],
\]
\[
H^{\eta(\frac{m}{p}-2)}|\Delta_g u| \leq |\Delta_g \varphi| + C [H^{-1}|\nabla_g\varphi||\nabla_g H| + H^{-2}\varphi |\Delta_g H|],
\]
and 
\[
H^{\eta(\frac{m}{p}-2)}|\mathrm{Hess}\, u| \geq |\mathrm{Hess}\, \varphi| - C [H^{-1}|\nabla_g\varphi||\nabla_g H| + H^{-2}\varphi |\mathrm{Hess}\, H|].
\]
Combining these latter with \eqref{eqhess2} we obtain
\begin{align*}%\label{eqhess3}
\int_M |\mathrm{Hess}\,\varphi|_g^p\, d\mathrm{vol}_g
 \leq &C\left[ 
\int_M H^{2p\eta}|\varphi|^p 
\,d\mathrm{vol}_{g} + \int_M |\Delta_g \varphi|^p 
\,d\mathrm{vol}_{g}\right.\\
 &+ \left.\int_M H^{-p}|\nabla \varphi|_g^p|\nabla H|_g^p 
\,d\mathrm{vol}_{g} + \int_M H^{-2p}|\varphi|^p|\mathrm{Hess}\, H|_g^p 
\,d\mathrm{vol}_{g} 
\right]\nonumber\\
\leq &C\left[ 
\int_M H^{2p\eta}|\varphi|^p 
\,d\mathrm{vol}_{g} + \int_M |\Delta_g \varphi|^p 
\,d\mathrm{vol}_{g}
%\right.\nonumber\\ &+ \left.
+ \int_M |\nabla \varphi|_g^p 
\,d\mathrm{vol}_{g} 
\right].\nonumber
\end{align*}
where we used the fact that $H$ is a strictly positive exhaustion function, $|\nabla H|$ is uniformly bounded and $|\mathrm{Hess}\, H|\leq C H^\eta$. Applying \cite[Proposition 3.10 a)]{GuneysuPigola} with $\varepsilon$ small enough, we finally get 
\begin{align*}%\label{eqhess3}
\int_M |\mathrm{Hess}\,\varphi|_g^p\, d\mathrm{vol}_g
\leq C\left[ 
\int_M H^{2p\eta}|\varphi|^p 
\,d\mathrm{vol}_{g} + \int_M |\Delta_g \varphi|^p 
\,d\mathrm{vol}_{g}
\right]\nonumber
\end{align*}
as desired.
\end{proof}
%%%%%%%%%%%%%%%%%%%%%%%%%%%%%%%%%%%%%%%%%%
\subsection{Omori-Yau maximum principle}\label{subsec_OY}
We end this section with a remark on the Omori-Yau maximum principle. Recall that a Riemannian manifold $(M, g)$ is said to satisfy the full Omori-Yau maximum principle for the Hessian if for any function $u\in C^{2}(M)$ with $u^{*}=\sup_{M}u<+\infty$, there exists a sequence $\left\{x_{n}\right\}_{n}\subset M$ with the properties
\[
\ \mathrm{(i)}\,u(x_{n})>u^{*}-\frac{1}{n},\quad\mathrm{(ii)}\,|\nabla u (x_{n})|<\frac{1}{n},\quad\mathrm{(iii)}\,\mathrm{Hess}(u)(x_{n})<\frac{1}{n}g,
\]
for each $n\in\mathbb{N}$. In \cite{PRS_Mem} it was proved that
the full Omori-Yau maximum principle for the Hessian holds e.g. if the radial sectional curvature of $M$ (i.e. the sectional curvature of $2$-planes containing $\nabla r$), satisfies
\[
\ \mathrm{Sect}_{\mathrm{rad}}\geq -C^2(1+r^{2})\prod_{j=1}^{\bar{j}}\left(\ln^{[j]}(r)\right)^2,
\]
where $\ln^{[j]}$ stands for the $j$-th iterated logarithm. In \cite{IRV-HessCutOff} we proved that the same is true if
\begin{equation*}
\left|\mathrm{Ric}_{g}\right|(x)\leq D^2(1+r(x)^2),\quad\mathrm{inj}_{g}(x)\geq \frac{i_0}{D(1+r(x))}>0\quad\mathrm{on}\,\,M.
\end{equation*}
Using Theorem \ref{th_distancefunction} and reasoning as in \cite{IRV-HessCutOff}, we get that assuming
\begin{equation*}
\left|\mathrm{Ric}_{g}\right|(x)\leq C^2r^{2}\prod_{j=1}^{\bar{j}}\left(\ln^{[j]}(r)\right)^2,\quad\mathrm{inj}_{g}(x)\geq \frac{i_0}{r\prod_{j=1}^{\bar{j}}\ln^{[j]}(r)}
\end{equation*}
outside a compact set of $M$ is enough.
%%%%%%%%%%%%%%%%%%%%%%%%%%%%%%%%%%%%%%%%%%
\begin{appendix}
\section{Some commutation formulas}\label{App_comm}

Using again the  ``$*$'' notation defined in Section \ref{SectNotations}, we have the validity of the following

\begin{lemma}\label{Comm}
Let $u\in C^{\infty}(M)$. 
%The following commutation rules hold true
%\begin{align*}
%\Delta\nabla_{i} u-\nabla_{i}\Delta u=& \RR_{ij}\nabla_{j} u=\mathrm{Ric}*\nabla u,\\
%\Delta \nabla^{2}_{ji}u-\nabla^{2}_{ji}\Delta u=&\nabla^{2}_{ti}u\RR_{tj}+\nabla^{2}_{tj}u\RR_{ti}+\nabla^{2}_{ts}u\RR_{isjt}+\nabla_{t}u\left(\nabla_{j}\RR_{it}-\nabla_{i}\RR_{jt}+\nabla_{t}\RR_{ij}\right)\\
%=& \mathrm{Riem}*\nabla^{2}u
%+\nabla \mathrm{Ric}*\nabla u,\\
%\ \Delta\nabla^{3} u-\nabla^{3}\Delta u
%=& \mathrm{Riem}*\nabla^{3}u
%+\nabla \mathrm{Ric}*\nabla^{2} u+\ldots\nabla^{2}\mathrm{Ric}*\nabla u
%\end{align*}
Then
\[
\ \Delta\nabla^{q-2} u-\nabla^{q-2}\Delta u
= \mathrm{Riem}*\nabla^{q-2}u
+\nabla \mathrm{Ric}*\nabla^{q-3} u+\ldots\nabla^{q-3}\mathrm{Ric}*\nabla u
\]
\end{lemma}
\begin{proof}[Proof. (Sketch)]
	We work in a normal frame $\left\{E_{i}\right\}$ orthonormal at $p\in M$,  and in frame computations we will use the convention of lowering all indices, summing over repeated indices. Recall that we are adopting the following sign conventions for curvatures:
\begin{align*}
&\mathrm{R}(X,Y)=\nabla_{X}\nabla_{Y}-\nabla_{Y}\nabla_{X}-\nabla_{[X,Y]};\\
&\mathrm{Riem}(X,Y,Z,W)=g(\mathrm{R}(X,Y)Z,W)\\
&\mathrm{Ric}(X,Y)=\mathrm{tr}\,\mathrm{Riem}(X,\cdot, \cdot, Y).
\end{align*}
Moreover, for the ease of notation we will write in coordinates $\RR_{ijkl}$ for $\mathrm{Riem}_{ijkl}$. Letting $u\in C^{\infty}(M)$, recall the following commutation rules 
\begin{align}
&\left(\nabla_{i}\nabla_{j}-\nabla_{j}\nabla_{i}\right)u=0,\label{(2)}\\
&\left(\nabla_{k}\nabla_{j}-\nabla_{j}\nabla_{k}\right)\nabla_{i}u=-\nabla_{t}u\RR_{tijk}\label{(3)}.
\end{align}
More generally one can compute that, for any $l\geq 2$,
\begin{align}\label{(s)}\tag{s}
\left(\nabla_{i_{l}}\nabla_{i_{l-1}}-\nabla_{i_{l-1}}\nabla_{i_{l}}\right)\nabla_{i_{l-2}}\ldots\nabla_{i_{1}}u=&-\nabla_{i_{l-2}}\ldots\nabla_{i_{2}}\nabla_{t}u
\RR_{ti_{1}i_{l-1}i_{l}}-\nabla_{i_{l-2}}\ldots\nabla_{i_{3}}\nabla_{t}\nabla_{i_{1}}u
\RR_{ti_{2}i_{l-1}i_{l}}\\
&-\ldots-\nabla_{t}\nabla_{i_{l-3}}\nabla_{i_{1}}u\RR_{ti_{l-2}i_{l-1}i_{l}}.\nonumber
\end{align}
If one is interested in the commutation rule for the $(s-1)$-th and the $s$-th derivative of $u$ of a total of $q$ derivatives, it suffices to take $\nabla^{q-s}$ of formula  \eqref{(s)}. 
When studying a term like $\Delta\nabla^{q-2}u-\nabla^{q-2}\Delta u$, one hence realizes that the higher order derivatives of curvature terms arise from commutators of the $2$-nd and the $3$-rd (of the total $q$) derivatives. There are two such terms in the telescopic development of $\Delta\nabla^{q-2}u-\nabla^{q-2}\Delta u$, namely
\begin{align*}
&\nabla^{q-5}_{i_{q}\ldots i_{6}}\nabla_{i_{5}}\nabla_{p}\left[\left(\nabla_{p}\nabla_{i_{2}}\nabla_{i_{1}}-\nabla_{i_{2}}\nabla_{p}\nabla_{i_{1}}\right)u\right]\\=&\nabla^{q-5}_{i_{q}\ldots i_{6}}\nabla_{i_{5}}\nabla_{p}\left[-\nabla_{t}u\RR_{ti_{1}i_{2}p}\right]\\
=&-\nabla^{q-5}_{i_{q}\ldots i_{6}}\left[\nabla_{i_{5}}\nabla_{p}\nabla_{t}u\RR_{ti_{1}i_{2}p}+\nabla_{p}\nabla_{t}u\nabla_{i_5}\RR_{ti_{1}i_{2}p}+\nabla_{i_{5}}\nabla_{t}u\nabla_{i_5}\nabla_{p}\RR_{ti_{1}i_{2}p}+\nabla_{t}u\nabla_{i_5}\nabla_{p}\RR_{ti_{1}i_{2}p}\right],
\end{align*}
and
\begin{equation*}
\nabla^{q-5}_{i_{q}\ldots i_{6}}\nabla_{i_{5}}\nabla_{i_{2}}\left[\left(\nabla_{p}\nabla_{i_{1}}\nabla_{p}-\nabla_{i_{1}}\nabla_{p}\nabla_{p}\right)u\right]=\nabla^{q-5}_{i_{q}\ldots i_{6}}\nabla_{i_{5}}\nabla_{i_{2}}\left[-\nabla_{t}u\RR_{tpi_{1}p}\right]
\end{equation*}
Using Bianchi's identities each of these terms can be written in the form 
\[
\mathrm{Riem}*\nabla^{q-2}u+\nabla\mathrm{Ric}*\nabla^{q-3} u+\ldots+\nabla^{q-3}\mathrm{Ric}*\nabla u.
\]
\end{proof}
%%%%%%%%%%%%%%%%%%%%%%%%%%%%%%%%%%%%%%%%%
\section{Weitzenb\"ock formula for $\Delta_{\mathrm{Sym}}$}\label{AppB}

Given $h$ a totally symmetric $(0,k-1)$ tensor, $\left\{E_{i}\right\}$ a local orthonormal frame on $M$, recall that 
\begin{align*}
(D_{S}h)(X_{1},X_{2},\ldots, X_{k})=&ks_{k}(\nabla h)(X_{1}, X_{2} \ldots, X_{k})=\frac{1}{(k-1)!}\left(\sum_{\sigma\in \Pi_{k}}(\nabla h)(X_{\sigma(1)},\ldots,X_{\sigma(k)}),\right)\\
(D_{S}^{*}h)(X_{1},\ldots,X_{k-2})=&-\sum_{i}(\nabla_{E_{i}}h)(E_{i},X_{1},\ldots,X_{k-2}).
\end{align*}

For the sake of completeness we provide here a proof of the following
\begin{lemma}[\cite{Sampson}]
Letting $\Delta_{\mathrm{Sym}}=D_{S}^{*}D_{S}-D_{S}D_{S}^{*}$, we have that
\begin{align}\label{WeitSampsApp}
\ \Delta_{\mathrm{Sym}}h&=\Delta_{B} h-\tilde{\mathfrak{R}}^{(k-1)}(h)+\tilde{\mathfrak{S}}^{(k-1)}(h)\\
&=\Delta_{B} h-\mathfrak{Ric}(h),\nonumber
\end{align}
where
\begin{align*}
 \tilde{\mathfrak{R}}^{(k-1)}(h)(X_{1},\ldots,X_{k-1})\doteq&\sum_{i=1}^{k-1}\sum_{j} h(X_{1},\ldots,E_{j},\ldots,X_{k-1})\mathrm{Ric}(X_{i},E_{j});\\
\tilde{\mathfrak{S}}^{(k-1)}(h)(X_{1},\ldots, X_{k-1})\doteq&\sum_{p<i}\sum_{j,l}h(X_{1},\ldots,E_{l},\ldots,E_{j},\ldots,X_{k-1})\cdot\\
&\cdot\left(\mathrm{Riem}( E_{l},X_{p},X_{i},E_{j})+\mathrm{Riem}(E_{l},X_{i},X_{p},E_{j})\right),
\end{align*}
the definition of $\mathfrak{Ric}$ has been recalled in \eqref{fkric}, and $\Delta_{B}=-\mathrm{tr}_{12}(\nabla^{2})=\nabla^{*}\nabla$, with $\nabla^{*}$ the formal $L^{2}$-adjoint of $\nabla$.
\end{lemma}

\begin{proof}
Assume that we perform computation at a point $p\in M$ in a normal frame, so that all covariant derivatives of vector fields vanish at $p$. We have that
\begin{align*}
&(D_{S}^{*}D_{S}h)(X_{1},\ldots, X_{k-1})=-\sum_{i}(\nabla_{E_{i}}D_{S}h)(E_{i},X_{1},\ldots, X_{k-1})\\
=&-\frac{1}{(k-1)!}\sum_{i}\nabla_{E_{i}}\left((k-1)!\nabla_{E_{i}}h(X_{1},\ldots,X_{k-1})+(k-1)!\nabla_{X_{1}}h(E_{i},X_{2},\ldots, X_{k-1})+\ldots\right.\\
&\left.\,\,\quad\quad\quad\quad\quad\quad\quad\quad+(k-1)!\nabla_{X_{k-1}}h(E_{i},X_{1},\ldots, X_{k-2})\right)\\
=&-\sum_{i}\left(\nabla_{E_{i}}\nabla_{E_{i}}h(X_{1},\ldots, X_{k-1})+\nabla_{E_{i}}\nabla_{X_{1}}h(E_{i}, X_{2}, \ldots, X_{k-1})\right.\\
&\quad\quad\quad\left.+\ldots+\nabla_{E_{i}}\nabla_{X_{k-1}}h(E_{i}, X_{1},\ldots, X_{k-2})\right).
\end{align*}
While,
\begin{align*}
&(D_{S}D_{S}^{*}h)(X_{1}, \ldots, X_{k-1})=(k-1)s_{k-1}(\nabla D_{S}^{*}h)(X_{1},\ldots, X_{k-1})\\
=&\frac{1}{(k-2)!}\left((k-2)!\nabla_{X_{1}}D_{S}^{*}h(X_{2}, \ldots, X_{k-1})+\ldots+ (k-2)!\nabla_{X_{k-1}}D_{S}^{*}h(X_{1},\ldots, X_{k-2})\right)\\
=&-\sum_{i}\left(\nabla_{X_{1}}\nabla_{E_{i}}h(E_{i}, X_{2},\ldots, X_{k-1})+\ldots+\nabla_{X_{k-1}}\nabla_{E_{i}}h(E_{i}, X_{1},\ldots, X_{k-2})\right).
\end{align*}
Hence
\begin{align*}
&\Delta_{\mathrm{Sym}}h=(D_{S}^{*}D_{S}-D_{S}D_{S}^{*})h\\
=&\Delta_{B}h+\sum_{i}\left(\nabla_{X_{1}}\nabla_{E_{i}}h(E_{i},X_{2},\ldots, X_{k-1})-\nabla_{E_{i}}\nabla_{X_{1}}h(E_{i},X_{2},\ldots, X_{k-1})\right.\\
&\left.\quad\quad\quad\quad\quad+\nabla_{X_{2}}\nabla_{E_{i}}h(E_{i}, X_{1},X_{3},\ldots, X_{k-1})-\nabla_{E_{i}}\nabla_{X_{2}}h(E_{i}, X_{1},X_{3},\ldots, X_{k-1})+\ldots\right.\\
&\left.\quad\quad\quad\quad\quad+\nabla_{X_{k-1}}\nabla_{E_{i}}h(E_{i},X_{1},\ldots, X_{k-2})-\nabla_{E_{i}}\nabla_{X_{k-1}}h(E_{i}, X_{1},\ldots, X_{k-2})\right)
\end{align*}

Since $h$ is symmetric, a standard computation gives that
\begin{align*}
\left(R(X,Y)h\right)(X_{1},\ldots, X_{k-1})=&-\sum_{l}\left(h(E_{l},X_{2},\ldots, X_{k-1})\mathrm{Riem}(X,Y,X_{1},E_{l})\right.\\
&\quad\quad\quad\left.+\ldots+h(X_{1},\ldots,E_{l})\mathrm{Riem}(X,Y,X_{k-1},E_{l})\right).
\end{align*}
Hence,
\begin{align*}
 &\sum_{i}\left(\nabla_{X_{1}}\nabla_{E_{i}}h(E_{i}, X_{2},\ldots, X_{k-1})-\nabla_{E_{i}}\nabla_{X_{1}}h(E_{i},X_{2},\ldots, X_{k-1})\right)\\
 =&\sum_{i}(R(X_1,E_i)h)(E_{i},X_{2},\ldots, X_{k-1})\\
 =&-\sum_{i,l}\left(h(E_{l},X_{2},\ldots, X_{k-1})\mathrm{Riem}(X_{1},E_{i},E_{i},E_{l})+h(E_{i},E_{l},X_{3},\ldots, X_{k-1})\mathrm{Riem}(X_{1},E_{i}, X_{2},E_{l})\right.\\
 &\left.\quad\quad\quad+\ldots+h(E_{i}, X_{2},\ldots, X_{k-2},E_{l})\mathrm{Riem}(X_{1},E_{i},X_{k-1},E_{l})\right),
\end{align*}
and analogously for the other terms. We thus get the validity of the first equality in \eqref{WeitSampsApp}. The second equality is a direct computation, using the expression \eqref{fkric} for $\mathfrak{Ric}$.
\end{proof}
\end{appendix}

%%%%%%%%%%%%%%%%%%%%%%%%%%%%%%%%%%%%%%%%%%%%%%
%%%%%%%%%%%%%%%%%%%%%%%%%%%%%%%%%%%%%%%%%%%%%%
\begin{acknowledgement*}
The first and second author gratefully acknowledge the partial support of the PRIN project 2017 ``Real and Complex Manifolds: Topology, Geometry and holomorphic dynamics'' and of the MIUR project ``Dipartimenti di Eccellenza 2018-2022''- CUP: E11G18000350001
The first author is partially supported by INdAM-GNSAGA. The second and the third author are partially supported by INdAM-GNAMPA.
\end{acknowledgement*}

%  \makeatletter
%    \providecommand\@dotsep{5}
%  \makeatother
%  \listoftodos\relax

\end{document}